\documentclass[a4paper,10pt]{article}

\pdfoutput=1

\usepackage[square,numbers]{natbib}
\usepackage{amsfonts,amsmath,amssymb,amsthm}
\usepackage{graphicx}
\usepackage[arrow, matrix, curve]{xy}
\usepackage[absolute]{textpos}
\usepackage{afterpage}
\usepackage{times}

\title{Quantum cluster algebras of type A and the dual canonical basis}
\author{Philipp Lampe}

\newcommand{\for}{{\rm for}\hspace{0.1cm}}
\newcommand{\ACM}{$\mathcal{A}(\mathcal{C}_M)$}

\theoremstyle{definition}
\newtheorem{theorem}{Theorem}[section]
\newtheorem{lemma}[theorem]{Lemma}
\newtheorem{proposition}[theorem]{Proposition}
\newtheorem{definition}[theorem]{Definition}
\newtheorem{corollary}[theorem]{Corollary}
\newtheorem{remark}[theorem]{Remark}
\newtheorem{example}[theorem]{Example}

\begin{document}

\maketitle

\begin{abstract}
The article concerns the subalgebra $U_v^+(w)$ of the quantized universal enveloping algebra of the complex Lie algebra $\mathfrak{sl}_{n+1}$ associated with a particular Weyl group element of length $2n$. We verify that $U_v^+(w)$ can be endowed with the structure of a quantum cluster algebra of type $A_n$. The quantum cluster algebra is a deformation of the ordinary cluster algebra Gei\ss -Leclerc-Schr\"{o}er attached to $w$ using the representation theory of the preprojective algebra. Furthermore, we prove that the quantum cluster variables are, up to a power of $v$, elements in the dual of Lusztig's canonical basis under Kashiwara's bilinear form. 
\end{abstract}

\section{Introduction} 
\label{intro}

{\it Cluster algebras} are commutative algebras created in 2000 by Fomin-Zelevinsky \cite{FZ:ClusterAlgI} in the hope to obtain a combinatorial description of the dual of Lusztig's {\it canonical basis} of a quantum group.

A cluster algebra comes with a distinguished set of generators called {\it cluster variables}. Each cluster variable belongs to several overlapping {\it clusters}. Every cluster, and hence every cluster variable, is obtained from an initial cluster by a sequence of {\it mutations}. Every mutation replaces an element in a cluster by an explicitly defined rational function in the variables of that cluster. We refer to Fomin-Zelevinsky \cite{FZ:ClusterAlgI} for definitions and to Fomin-Zelevinsky \cite{FZ:ClusterAlgIV} for a good survey about cluster algebras. 

It quickly turned out that Fomin-Zelevinsky's theory of cluster algebras has many interesting applications and coheres with various mathematical objects. Let us mention the representation theory of quivers and finite-dimensional algebras, the representation theory of preprojective algebras, root systems of Kac-Moody algebras, Calabi-Yau categories, quantum groups, and Lusztig's canonical basis of universal enveloping algebras.

A momentous step in the development was the {\it additive categorification} of {\it acyclic} cluster algebras by {\it cluster categories}. Cluster categories were defined by Buan-Marsh-Reineke-Reiten-Todorov \cite{BMRRT} and independently by Caldero-Chapoton-Schiffler \cite{CCS}  for type $A$. The cluster category $\mathcal{C}_Q$ associated with a quiver $Q$ is an orbit category of the bounded derived category of the category of representations of $Q$. Keller \cite{Keller:Overview} proved that cluster categories are triangulated categories. Key ingredients for the verification of the additive categorification of acyclic cluster algebras by cluster categories are due to Buan-Marsh-Reiten-Todorov \cite{BMRT}, Buan-Marsh-Reiten \cite{BMR1,BMR2}, Gei\ss -Leclerc-Schr\"oer \cite{GLS:Sem1,GLS:Sem2}, and Caldero-Keller \cite{CK2,CK1}. The process of mutation in the cluster algebra resembles the process of {\it tilting} in the cluster category. Hence, we obtain a link between quiver representations and triangulated categories on one side and a large class of cluster algebras on the other side.

Furthermore, Gei\ss -Leclerc-Schr\"oer \cite{GLS:Unipotent} provided an additive categorification by categories arising from the study of {\it Kac-Moody groups} and {\it unipotent cells}. In this construction the categorified cluster algebras are not necessarily acyclic. Let $\mathfrak{g}$ be the Kac-Moody Lie algebra attached to $Q$ and let $\mathfrak{g}=\mathfrak{n}_{-}\oplus \mathfrak{h} \oplus \mathfrak{n}$ be its triangular decomposition. Gei\ss -Leclerc-Schr\"oer's construction is related to the {\it preprojective algebra} $\Lambda$ associated with $Q$. Buan-Iyama-Reiten-Scott \cite{BIRS} attached to every element $w$ in the Weyl group of $\mathfrak{g}$ a subcategory $\mathcal{C}_w \subset {\rm mod}(\Lambda)$. Gei\ss -Leclerc-Schr\"oer \cite{GLS:Unipotent} endow the coordinate ring $\mathbb{C}[N(w)]$ of the unipotent group $N(w)$ with the structure of a cluster algebra $\mathcal{A}(w)$. Here, $N$ denotes the pro-unipotent pro-group associated with the completion $\widehat{\mathfrak{n}}$ and $N(w)=N\cap(w^{-1}N_{-}w)$. The coordinate ring $\mathbb{C}[N(w)]$ is naturally isomorphic to a subalgebra of the graded dual $U(\mathfrak{n})^*_{gr}$ of the universal enveloping algebra of $\mathfrak{n}$. The cluster variables are $\delta$-functions of rigid modules over the preprojective algebra. All cluster monomials lie in the dual semicanonical basis. 

Let us mention that cluster algebras also gained popularity in other branches of mathematics; for example, for Poisson geometry, see Gekhtman-Shapiro-Vainshtein \cite{GSV}, for Teichm\"{u}ller theory, see Fock-Goncharov \cite{FG}, for combinatorics, see Musiker-Propp \cite{MP}, for integrable systems, see Fomin-Zelevinsky \cite{FZ5}, for Donaldson-Thomas invariants and mathematical physics see Kontsevich-Soibelman \cite{KS}, etc. 

In this article we transfer to the quantized setup. We consider the case $\mathfrak{g}=\mathfrak{sl}_{n+1}$ (for some natural number $n$) and choose a particular Weyl group element $w$ of length $2n$ which is equal to the square of a Coxeter element. The reasons for this particular choice are the following three: First, in this case the stable category $\underline{\mathcal{C}_w}$ is triangle equivalent to the corresponding cluster category $\mathcal{C}_Q$ by a result of Gei\ss -Leclerc-Schr\"oer \cite[Theorem 11.1]{GLS:Unipotent}; second, there exist recursions for quantum cluster variables (simpler than the mutation relations) which allow to connect quantum cluster variables with canonical basis elements; third, cluster algebras of almost all types can be realized as $\mathcal{A}(w)$ for some square $w$ of a Coxeter element in the Weyl group of a Kac-Moody Lie algebra of the same type, see Gei\ss -Leclerc-Schr\"oer \cite[Section 2.6]{GLS:KacMoody}.  

We establish a quantum cluster algebra structure on a subalgebra $U^+_v(w)$ of the {\it quantized} universal enveloping algebra $U_v(\mathfrak{n})$ of $\mathfrak{n}$. Quantum cluster algebras were introduced by Berenstein-Zelevinsky \cite{BZ:QuantumClusterAlg}. The quivers corresponding to  $\mathfrak{g}=\mathfrak{sl}_{n+1}$ are {\it Dynkin} quivers of type $A_n$. We choose a particular orientation: let $Q=(Q_0,Q_1)$ be the Dynkin quiver of type $A_n$ with an alternating orientation beginning with a source. We denote the set of vertices by $Q_0=\left\lbrace1,2,\ldots,n\right\rbrace$. Figure \ref{fig:QuiverA} illustrates the example $n=13$. The choice of the orientation matches the choice of the Weyl group element $w$. The reduced expression of $w$ (that is used to construct $U^+_v(w)$) and its initial subsequences (that are used to construct the generators $U_v^+(w)$) are related to the indecomposable injective modules over the {\it path algebra} of $Q$ and their {\it Auslander-Reiten translates}, respectively.

\begin{figure}[h!]
\centering
\includegraphics{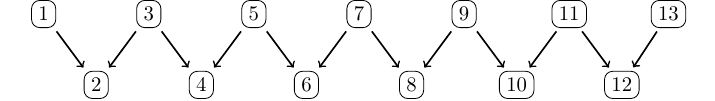}
\caption{The quiver $Q$ of type $A_{13}$} \label{fig:QuiverA}
\end{figure}

The construction of $U_v^+(w)$ is due to Lusztig \cite{Lusztig:QuantumGroups}. The algebra is generated by $2n$ elements that satisfy straightening relations; it degenerates to a commutative algebra in the classical limit $v=1$. The generators are constructed via Lusztig's $T$-automorphisms. The quantized universal enveloping algebra $U_v(\mathfrak{n})$ is a self-dual {\it Hopf algebra}. Also, the subalgebra $U_v^+(w)$ is isomorphic to its dual, i.e., it is isomorphic to the quantized coordinate ring $\mathbb{C}_v[N(w)]$. The algebra $U_v^+(w)$ possesses several distinguished bases, including a {\it Poincar\'e-Birkhoff-Witt basis} for every reduced expression for $w$, a {\it canonical basis}, and their duals. The article concerns the dual of Lusztig's canonical basis under Kashiwara's bilinear form \cite{Kashiwara}.

It is conjectured (see for example Kimura \cite{Kimura}) that the quantized coordinate rings $\mathbb{C}_v[N(w)]$ are {\it quantum cluster algebras} $\mathcal{A}_v(w)$ in general and that the set $\mathcal{M}_v$ of all quantum cluster monomials, taken up to powers of $v$, is a subset of the dual canonical basis $\mathcal{B}^*$, i.e., the following diagram commutes:
\begin{center}
\hspace{0cm}
\begin{xy}
\xymatrix{\mathcal{A}_v(w) \ar[r]  &  \mathbb{C}_v[N(w)] & \hspace{-1cm} \subset U_v(\mathfrak{n})_{gr}^{\ast}  \\
\mathcal{M}_v \ar@{^{(}->}[r] \ar@{^{(}->}[u] &  \mathcal{B}^* \ar@{^{(}->}[u] &}
\end{xy}
\end{center}
The conjecture has only been verified in very few cases, see Berenstein-Zelevinsky \cite{BZ:StringBases} for type $A_2$ and $A_3$, and the author \cite{Lampe} for an example of Kronecker type. 

We verify that in our case the integral form $U_v^+(w)_{\mathbb{Z}}$ of $U_v^+(w)$ is (after extending coefficients) a quantum cluster algebra. The proof relies on the exact form of the straightening relations. The description of the straightening relations features (besides Lusztig's $T$-automorphisms) Leclerc's embedding \cite{Leclerc:Shuffles} of $U_v(\mathfrak{n})$ in the {\it quantum shuffle algebra}. The exact form of the straightening relations enables us to verify that recursively defined variables satisfy a lattice property and an invariance property so that they are elements in the dual canonical basis.

The cluster algebra $\mathcal{A}(w)$, just as the quantum cluster algebra $\mathcal{A}_v(w)$, is of type $A_n$. Every cluster contains $n$ {\it frozen} and $n$ {\it mutable} cluster variables. Altogether there are $n+\frac{n(n+1)}{2}$ mutable and $n$ frozen cluster variables. Most of the cluster variables can be realized as minors of certain matrices, see Section \ref{DescriptionOfClusterVariables}. The structure of these minors implies that there is (besides the usual cluster exchange relation) a recursive way to compute these cluster variables avoiding denominators. Theorem \ref{QuantRecursionForClusterVar}, the main theorem, asserts that the recursion can be quantized to a recursion for the corresponding quantum cluster variables. The quantized recursions imply quantum exchange relations so that the integral form $U_v^+(w)_{\mathbb{Z}}$ of $U_v^+(w)$ becomes (after extending coefficients) a quantum cluster algebra.

Furthermore, it follows from our construction that the quantum cluster variables are (up to a power of $v$) elements in the dual of Lusztig's canonical basis under Kashiwara's bilinear form \cite{Kashiwara}.

\section{Representation theory of the quiver of type A and cluster algebras}
\label{RepTheo}
\subsection{The indecomposable modules over the path algebra}
\label{IndecMod}
Let $k$ be a field. In what follows we study the category ${\rm rep}_k(Q)$ of finite-dimensional $k$-representations of $Q$ over the field $k$. (For more detailed information on representations of quivers see for example Crawley-Boevey \cite{CB:RepTheory}.) The category ${\rm rep}_k(Q)$ is equivalent to the category ${\rm mod}(kQ)$ of finite-dimensional modules over the {\it path algebra} $kQ$. Gabriel's theorem \cite{Ga:Unzerlegbar} asserts that the quiver $Q$ admits (up to isomorphism) only finitely many indecomposable representations. In fact there are $\frac{(n+1)n}{2}$ indecomposable representations (up to isomorphism) and they are in bijection with the set of intervals $[i,j]=\left\lbrace i,i+1,i+2,\ldots j\right\rbrace \subset \mathbb{Z}$ with $1 \leq i \leq j \leq n$. The indecomposable representation associated with the interval $[i,j]$ is $V_{[i,j]}=\left((V_s)_{s \in Q_0}, (V_a)_{a \in Q_1} \right)$ defined by $k$-vector spaces $$V_s=\begin{cases} k, & {\rm if} \ i \leq s \leq j; \\ 0, & {\rm otherwise}; \end{cases}$$ associated with vertices $s$, and $k$-linear maps $$V_a=\begin{cases} 1, & {\rm if} \ V_s = V_t = k; \\ 0, & {\rm otherwise}; \end{cases}$$ associated with arrows $a \colon s \to t$. 

All further considerations will basically depend on the parity of $n$. For a compact and effective handling of all cases we make the assumption that $n$ is odd. Denote by $Q'=(Q_0',Q_1')$ to be the quiver obtained from $Q$ by removing the vertex $n$. The quiver $Q'$ is of type $A_{n-1}$, and the examination of both $Q'$ and $Q$ covers all cases. Note that every representation of $Q'$ can be viewed as a representation of $Q$ supported on the first $n-1$ vertices. An example of the quiver $Q'$ is shown in Figure \ref{fig:QuiverA'}.

\begin{figure}
\centering
\includegraphics{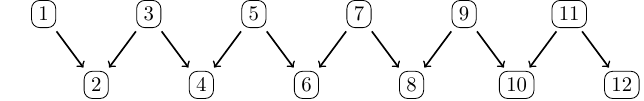}
\caption{The quiver $Q'$ of type $A_{12}$} \label{fig:QuiverA'}
\end{figure}

If $n=1$, i.e., the quiver has one vertex and no arrows, then the category ${\rm rep}_k(Q)$ can easily be described. In this case modules over the path algebra $kQ$ are $k$-vector spaces, and the $k$-vector space $k$ of dimension $1$ is the only indecomposable module. In the other cases, the most suggestive way to illustrate the category ${\rm rep}_k(Q)$ is given by its {\it Auslander-Reiten quiver}. For an introduction to Auslander-Reiten theory we refer to Assem-Simson-Skowronski \cite[Chapter IV]{ASS:Lectures}. 

The simplest non-trivial example is the Auslander-Reiten quiver of type $A_2$ which can be seen in Figure \ref{fig:ARA2}. In this case there are (up to isomorphism) three indecomposable representations, two of which are injective. The representations are displayed by numbers that represent basis vectors and composition series; cf. Gei\ss -Leclerc-Schr\"oer \cite[Section 7.5]{GLS:Unipotent}. The solid arrows represent {\it irreducible maps}; the dashed arrow represents the {\it Auslander-Reiten translation}. Note that the Auslander-Reiten translate of the injective representation associated with vertex $2$ is the zero representation.

\begin{figure}[h!]
\centering
\includegraphics{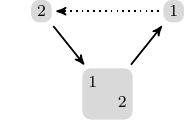}
\caption{The Auslander-Reiten quiver for $n=2$}\label{fig:ARA2}
\end{figure}

In what follows we are interested in the indecomposable injective $kQ$-modules $I_i$ associated with vertices $i \in Q_0$ and their Auslander-Reiten translates $\tau_{kQ}(I_i)$. Similarly, we are interested in the indecomposable injective $kQ'$-modules $I'_i$ associated with vertices $i \in Q'_0$ and their Auslander-Reiten translates $\tau_{kQ'}(I'_i)$. (For simplicity, we drop from now on the index attached to $\tau$ whenever it is clear which algebra we are referring to.) The choice of the alternating orientations of the quivers $Q$ and $Q'$ ensure that from type $A_3$ onwards we have $\tau(I) \neq 0$ for every indecomposable injective $kQ$-module. (This would not be true for the linear orientation of the Dynkin diagram $A_n$. The Auslander-Reiten translate of the indecomposable injective representation corresponding to the sink would be zero in this case.) The direct sum $M=\bigoplus_{i=1}^n I_i \oplus \tau(I_i)$ is a {\it terminal} $kQ$-module in the sense of Gei\ss -Leclerc-Schr\"{o}er \cite[Section 2.2]{GLS:Unipotent}, and so is the $kQ'$ module $M'=\bigoplus_{i=1}^{n-1} I'_i \oplus \tau(I'_i)$.

\begin{figure}[h!]
\centering
\includegraphics{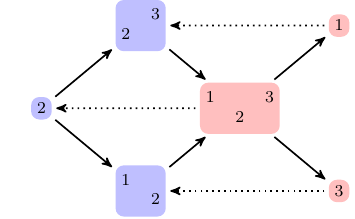}
\caption{The Auslander-Reiten quiver for $n=3$}\label{fig:ARQuiverA3}
\end{figure}

The small cases $A_3$ and $A_4$ will have to be treated separately. Figure \ref{fig:ARQuiverA3} and Figure \ref{fig:ARQuiverA4} display the indecomposable injective modules (red), their Auslander-Reiten translates (blue), and irreducible maps between them for the case $n=3$ and $n=4$, respectively. 

If $n=3$, then $M$ is the direct sum of all indecomposable $kQ$-modules, i.e., ${\rm mod}(kQ)={\rm add}(M)$.

\begin{figure}[h!]
\centering
\includegraphics{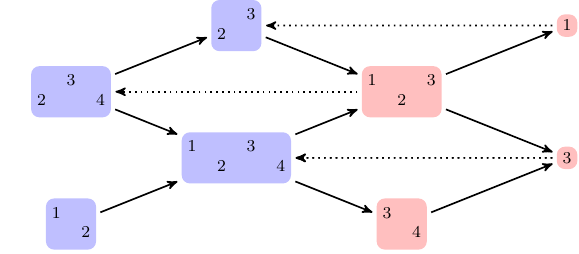}
\caption{A part of the Auslander-Reiten quiver for $n=4$}\label{fig:ARQuiverA4}
\end{figure}

From $A_5$ onwards a uniform description is possible. For type $A_n$ (remember that $n$ is assumed to be odd) the indecomposable components of $M$ can be written down explicitly: 
\begin{align*}
&I_i = V_{[i,i]}, &&{\rm if \ i \ is \ odd \ and } \ 1 \leq i \leq n, \\ 
&I_i = V_{[i-1,i+1]}, &&{\rm if \ i \ is \ even \ and} \ 2 \leq i \leq n,\\ 
&\tau(I_1) = V_{[2,3]}, && \\
&\tau(I_i) = V_{[i-2,i+2]}, &&{\rm if \ i \ is \ odd \ and} \ 3 \leq i \leq n-2,\\
&\tau(I_n) = V_{[n-2,n-1]}, && \\
&\tau(I_2) = V_{[2,5]}, && \\
&\tau(I_i) = V_{[i-3,i+3]}, &&{\rm if \ i \ is \ even \ and} \ 4 \leq i \leq n-3,\\
&\tau(I_{n-1}) = V_{[n-4,n-1]}. &&
\end{align*}
We display the relevant part of the Auslander-Reiten quiver of $A_n$ in Figure \ref{fig:AR} for the case $n=13$. As above, the indecomposable injective modules are colored red, their Auslander-Reiten translates blue.

\begin{figure}
\centering
\includegraphics{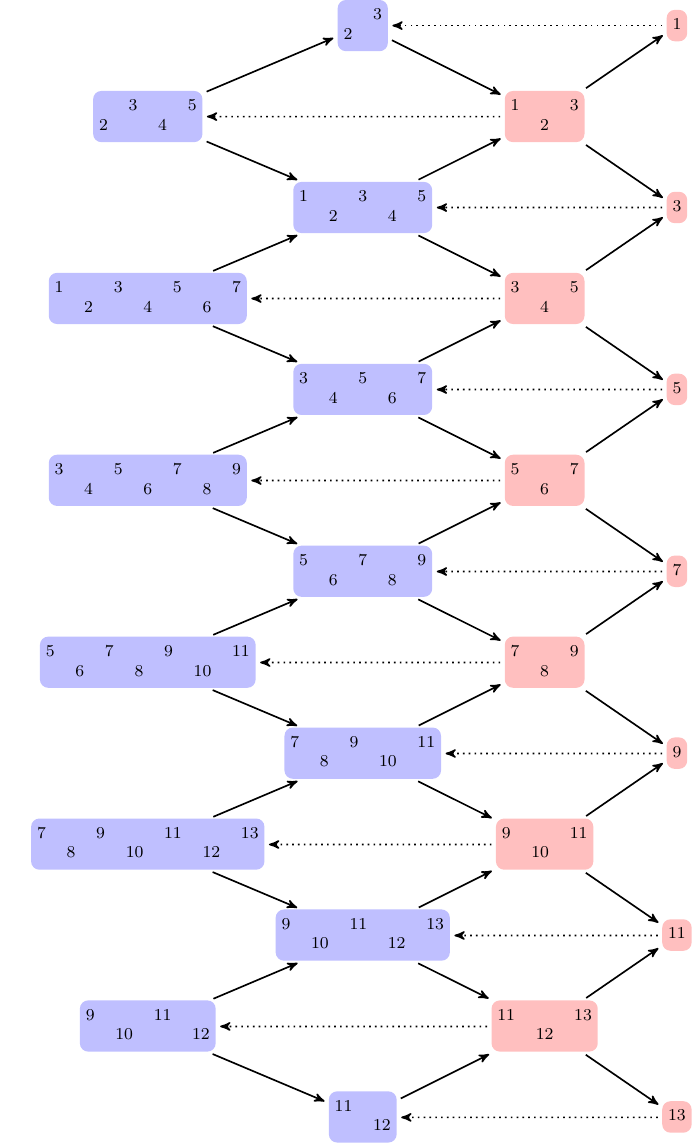}
\caption{A part of the Auslander-Reiten quiver of ${\rm mod}(kQ)$}\label{fig:AR}
\end{figure}

There are only a few changes if we restrict $Q$ to $Q'$. Observe that $I'_i=I_i$ for $i \in \left\lbrace 1,2,\ldots,n-3\right\rbrace$, and that $\tau(I'_i)=\tau(I_i)$ for $i \in \left\lbrace 1,2,\ldots,n-3\right\rbrace$. Note that the latter modules are $kQ$-modules supported on the first $n-1$ vertices and may therefore be viewed as $kQ'$-modules. Furthermore, we have 
\begin{align*}
&I'_{n-1}=V_{[n-2,n-1]},\\
&\tau(I'_{n-3})=V_{[n-6,n-1]},\\
&\tau(I'_{n-2})=V_{[n-4,n-1]},\\
&\tau(I'_{n-1})=V_{[n-4,n-3]}.
\end{align*}
An example of type $A_{12}$ is illustrated in Figure \ref{fig:AReven}.

\begin{figure}
\centering
\includegraphics{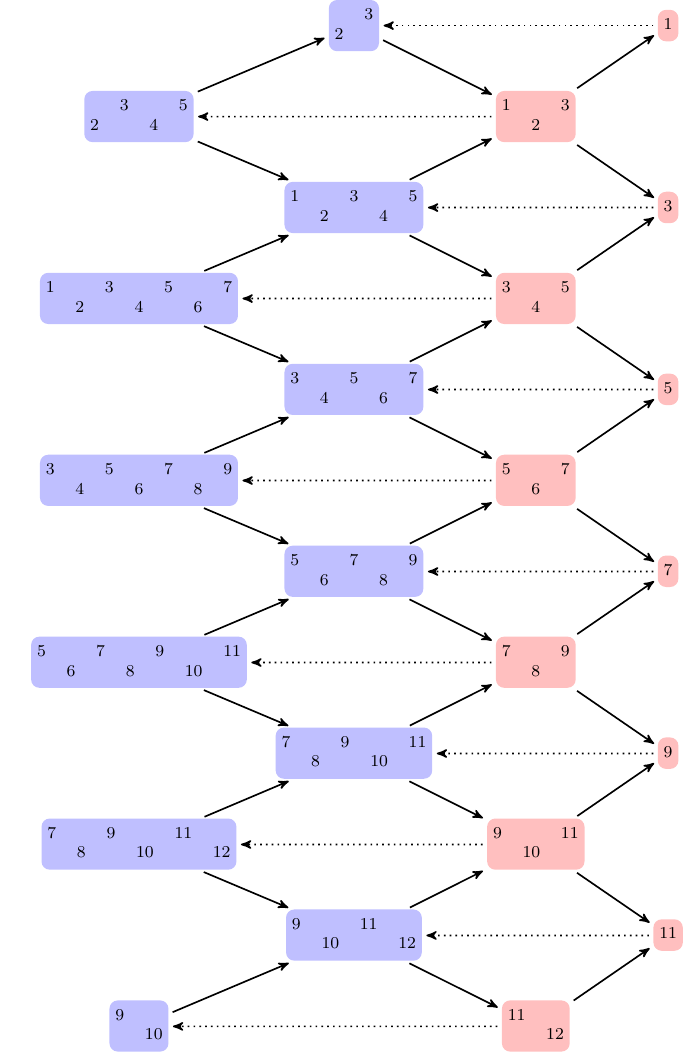}
\caption{A part of the Auslander-Reiten quiver of ${\rm mod}(kQ')$}\label{fig:AReven}
\end{figure}

\subsection{The preprojective algebra and rigid modules}
\label{PreprojAlgRigidMod}
The representation theory of the path algebra $kQ$ is closely related to the representation theory of the corresponding {\it preprojective algebra} $\Lambda$ defined as follows. For every arrow $a \colon s \to t$ in $Q_1$ introduce an additional arrow $a^{*} \colon t \to s$ in reverse direction and denote by $Q_1^{*}=\left\lbrace a^{*} \colon a \in Q_1\right\rbrace$ the set of all reversed arrows. The {\it double quiver} of $Q$ is defined to be the quiver $\overline{Q}=(\overline{Q}_0,\overline{Q}_1)$ given by a vertex set $\overline{Q}_0=Q_0$ and an arrow set $\overline{Q}_1=Q_1 \cup Q_1^{*}$. The preprojective algebra is defined to be $$\Lambda=k\overline{Q}/(c)$$  where the ideal $(c)$ is the two-sided ideal generated by the element $$c=\sum_{a \in Q_1}\left(a^{*}a-aa^{*} \right) \in k\overline{Q}.$$ The algebra $\Lambda$ is finite-dimensional, since $Q$ is an orientation of a Dynkin diagram, see Reiten \cite[Theorem 2.2a]{Reiten:Dynkin}. The category ${\rm mod}(\Lambda)$ of finite-dimensional $\Lambda$-modules is equivalent to the category ${\rm rep}_k(\overline{Q},(c))$ of finite-dimensional representations $M=((M_s)_{s \in Q_0}, (M_a)_{a \in \overline{Q}_1})$ of $\overline{Q}$ such that for any two vertices $s,t \in Q_0$ and any linear combination $\sum_{i=1}^m\lambda_ip_i \in (c)$ of paths $p_i \colon s \to t$ with scalars $\lambda_i \in k$ the associated linear map $\sum_{i=1}^m\lambda_iM_{p_i}$ is zero.

There is a {\it restriction functor} $\pi_Q \colon {\rm mod}(\Lambda) \to {\rm mod}(kQ)$ given by forgetting the linear maps associated with $a^*$ for all $a \in Q_1$ in the corresponding representation of the quiver $\overline{Q}$. Ringel \cite[Theorem B]{Ringel:PreprojectiveAlg} proved that the category ${\rm mod}(\Lambda)$ is equivalent to the category $C(1,\tau)$ whose objects are pairs $(X,f)$ consisting of a $kQ$-module $X$ and a $kQ$-module homomorphism $f \colon X \to \tau (X)$ from $X$ to its translate $\tau (X)$ and where a morphism $h\colon (X,f) \to (Y,g)$ is given by a $kQ$-module homomorphism $h \colon X \to Y$ for which the diagram 
\begin{center}
\hspace{0cm}
\begin{xy}
  \xymatrix{
    X \ar[r]^{h} \ar[d]^{f} &  Y  \ar[d]^{g}  \\
    \tau (X) \ar[r]^{\tau(h)} &  \tau (Y)  
  }
\end{xy}
\end{center}
commutes. 

Using the correspondence from above Gei\ss -Leclerc-Schr\"{o}er \cite[Section 7.1]{GLS:Unipotent} constructed for every $i \in Q_0$ and any natural numbers $a,b$ satisfying $0 \leq a \leq b \leq 1$ a $\Lambda$-module $T_{i,[a,b]}=\left(I_{i,[a,b]},e_{i,[a,b]} \right)$ where $I_{i,[a,b]}=\bigoplus_{j=a}^b \tau^j(I_i)$ and the map $$e_{i,[a,b]} \colon I_{i,[a,b]} \to \tau \left( I_{i,[a,b]}\right)=\bigoplus_{j=a+1}^{b+1} \tau^j(I_i) $$ is identity on every $\tau^j(I_i)$ for $a+1 \leq j \leq b$ and zero otherwise. We study $\Lambda$-modules $T_{i,[a,b]}$ for $i \in Q_0$ and $0 \leq a,b \leq 1$. We display the modules in Figures \ref{t00mod}, \ref{t01mod}, \ref{t11mod}.

\begin{figure}
\begin{center}
\includegraphics{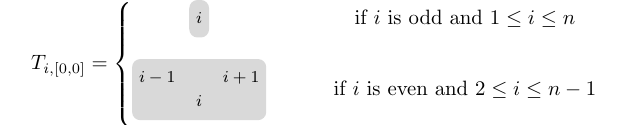}
\caption{The modules $T_{i,[0,0]}$}\label{t00mod}
\end{center}
\end{figure}

\begin{figure}
\begin{center}
\includegraphics{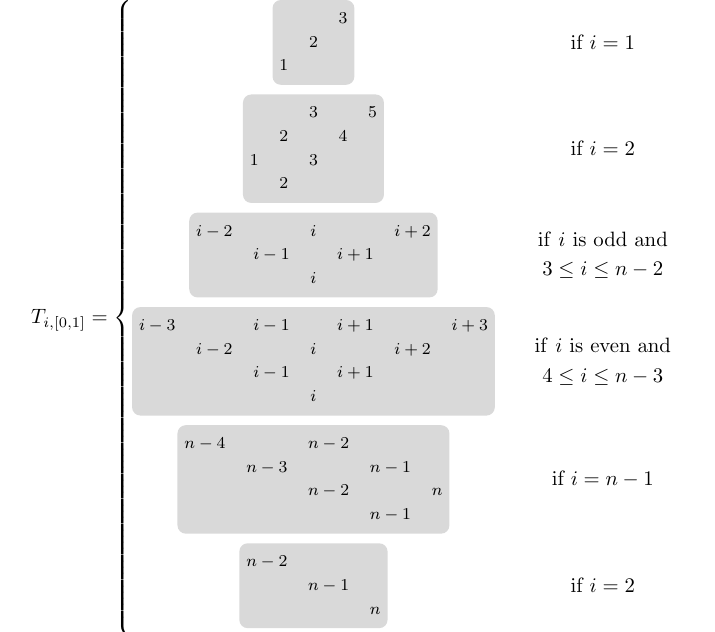}
\caption{The modules $T_{i,[0,1]}$}\label{t01mod}
\end{center}
\end{figure}

\begin{figure}
\begin{center}
\includegraphics{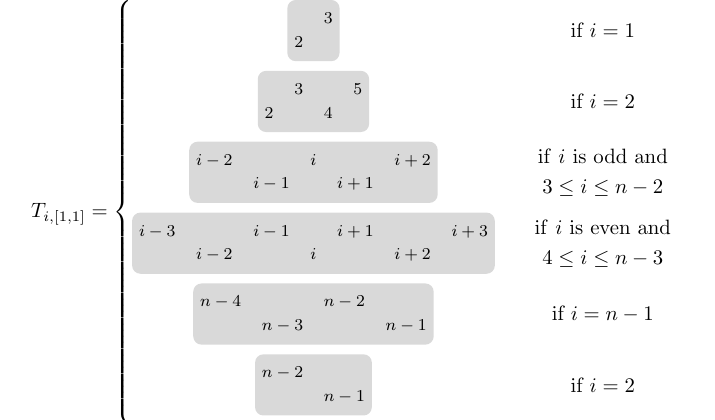}
\caption{The modules $T_{i,[1,1]}$}\label{t11mod}
\end{center}
\end{figure}

The modules $T_{i,[a,b]}$ for $i\in Q_0$ and $0 \leq a,b \leq 1$ are {\it rigid} and {\it nilpotent}. Recall that a $\Lambda$-module $T$ is said to be rigid if ${\rm Ext}_{\Lambda}^1(T,T)=0$ and it is said to be nilpotent if there exists an integer $N>0$ such that for each path $a_1a_2\cdots a_N$ of length $N$ in $\overline{Q}$ the associated linear map $T_{a_1}T_{a_2}\cdots T_{a_N}$ is zero. Rigidity follows from Gei\ss -Leclerc-Schr\"{o}er \cite[Lemma 7.1]{GLS:Unipotent}; nilpotency follows from Lusztig \cite[Proposition 14.2]{Lusztig:Quivers}.

Similarly, the representation theory of the path algebra $kQ'$ is closely related to the representation theory of the corresponding preprojective algebra $\Lambda'$.

\subsection{Notations from Lie theory}
The representation theory of the quiver $Q$ is related with Lie theory. Let $k=\mathbb{C}$. The Lie algebra associated with the Dynkin diagram $A_n$ is $\mathfrak{g}=\mathfrak{sl}_{n+1}$, i.e., the Lie algebra of $(n+1)\times(n+1)$ matrices with complex entries and vanishing trace. It admits a triangular decomposition $\mathfrak{g}=\mathfrak{n}_{-} \oplus \mathfrak{h} \oplus \mathfrak{n}$. Here, $\mathfrak{n}$ and $\mathfrak{n}_{-}$ denote the Lie algebras of strictly upper and strictly lower triangular $(n+1)\times(n+1)$ matrices, respectively, and $\mathfrak{h}$ denotes the Lie algebra of $(n+1)\times(n+1)$ diagonal matrices. The Lie algebra $\mathfrak{n}$ is called the {\it positive part} of $\mathfrak{g}$. 

Let $C=(a_{ij})_{1 \leq i,j \leq n}$ be the {\it Cartan matrix} associated with the quiver $Q$; its entries are: 
$$a_{ij}=\begin{cases}2, & {\rm if} \ i=j;\\-1, & {\rm if} \vert i-j\vert=1; \\ 0, & {\rm otherwise}.\end{cases}$$ 

The Lie algebra $\mathfrak{g}=\mathfrak{sl}_{n+1}$ is studied by its {\it roots}. The {\it root lattice} $R$ is defined to be the free abelian group generated by $\alpha_1,\alpha_2,\ldots,\alpha_n$. Here, the elements $\alpha_1,\alpha_2,\ldots,\alpha_n\in\mathfrak{h}^{\ast}$ are defined by $\alpha_i(\operatorname{diag}(d_1,d_2,\ldots,d_n))=d_i-d_{i+1}$ and they are called {\it simple roots}.  By $R^+ \subset R$ we denote the set of all linear combinations $\sum_{i=1}^nd_i\alpha_i$ with $d_i \in \mathbb{N}$. There is a symmetric bilinear form $(\cdot,\cdot) \colon R \times R \to \mathbb{R}$ which satisfies $(\alpha_i,\alpha_j)=a_{ij}$ for $1 \leq i,j \leq n$. By $\Delta^+\subseteq R$ we denote the set of {\it positive roots} of the corresponding root system. Then $\Delta^+ = \left\lbrace \alpha_i+\alpha_{i+1} + \cdots + \alpha_j \colon 1 \leq i \leq j \leq n\right\rbrace$. Under the bijection of Gabriel's theorem, a positive root $\alpha_i+\alpha_{i+1} + \cdots + \alpha_j$ with $1 \leq i \leq j \leq n$ is mapped to the indecomposable representation $V_{[i,j]}$ from Section \ref{IndecMod}.

The simple reflections $s_1$,$s_2,\ldots,s_n \colon \mathfrak{h}^{*} \to \mathfrak{h}^{*}$ act on the simple roots by
$$s_i(\alpha_j)=\begin{cases}-\alpha_i, & {\rm if} \ i=j; \\ \alpha_i+\alpha_j, & {\rm if} \ \vert i-j\vert=1; \\ \alpha_j, & {\rm otherwise}.\end{cases}$$ The group $W$ generated by the simple reflections is called the {\it Weyl group} of type $\mathfrak{g}$. The simple reflection satisfy the following relations
\begin{align}
s_is_j&=s_js_i, && {\rm if} \ \vert i-j \vert \geq 2; \label{s1}\\ 
s_is_js_i&=s_js_is_j, && {\rm if} \ \vert i-j \vert =1; \label{s2}\\
s_i^2&=1, 
\end{align}
for all $1 \leq i,j \leq n$. Therefore, the Weyl group $W$ is isomorphic to the symmetric group $S_n$.

To every terminal $kQ$-module Gei\ss -Leclerc-Schr\"{o}er \cite[Section 3.7]{GLS:Unipotent} attach a $Q^{op}$-adapted Weyl group element. The $Q^{op}$-adapted Weyl group element associated with the terminal module $M$ from Section \ref{IndecMod} is $w=s_1s_3s_5\cdots s_{n}s_2s_4s_6 \cdots s_{n-1}s_1s_3s_5 \cdots s_{n}s_2s_4s_6 \cdots s_{n-1}$. The given expression for $w$ is reduced. Let $j_1,j_2,\ldots,j_{2n}\in [1,n]$ such that for the reduced expression for $w$ from above we have $w=s_{j_1}s_{j_2}\cdots s_{j_{2n}}$. We put $\beta_1=\alpha_{j_1}$ and $\beta_k=s_{j_1}s_{j_2} \cdots s_{j_{k-1}}(\alpha_{j_k})$ for $2 \leq k \leq 2n$. Denote by $\Delta_w^+=\left\lbrace\beta_1,\beta_2,\ldots,\beta_{2n}\right\rbrace \subseteq \Delta^+$ the set of all $\beta_k$ with $1 \leq k \leq 2n$. Note that the notation is well-defined. If we choose another reduced expression $w=s_{j'_1}s_{j'_2}\cdots s_{j'_{2n}}$ for $w$, then $$\left\lbrace s_{j'_1}s_{j'_2}\cdots s_{j'_{k-1}}(\alpha_{j'_k}) \colon 1 \leq k \leq 2n \right\rbrace=\left\lbrace\beta_1,\beta_2,\ldots,\beta_{2n}\right\rbrace.$$

Furthermore, notice that under the bijection of Gabriel's theorem, the $2n$ positive roots $\beta_k$ with $1 \leq k \leq 2n$, correspond to the dimension vectors of the indecomposable direct summands of $M$ (compare Figure \ref{fig:QuiverA}). More precisely, for $n \geq 5$,
\begin{align*}
\Delta_w^+&=\left\lbrace \alpha_i \colon \ {\rm is \ odd \ and \ } 1 \leq i \leq n \right\rbrace \\
&\quad \cup \left\lbrace \alpha_{i-1}+\alpha_i+\alpha_{i+1} \colon \ {\rm is \ even \ and \ } 2 \leq i \leq n-1 \right\rbrace\\
&\quad \cup \left\lbrace \alpha_2+\alpha_3 \right\rbrace \cup \left\lbrace \alpha_{n-2}+\alpha_{n-1} \right\rbrace \\
&\quad \cup \left\lbrace \alpha_{i-2}+\cdots+\alpha_{i+2} \colon \ {\rm is \ odd \ and \ } 3 \leq i \leq n-3 \right\rbrace\\
&\quad \cup \left\lbrace \alpha_2+\alpha_3+\alpha_4+\alpha_5 \right\rbrace \cup \left\lbrace \alpha_{n-4}+\alpha_{n-3}+\alpha_{n-2}+\alpha_{n-1} \right\rbrace \\
&\quad \cup \left\lbrace \alpha_{i-3}+\cdots+\alpha_{i+3} \colon \ {\rm is \ even \ and \ } 4 \leq i \leq n-4 \right\rbrace.
\end{align*}

The universal enveloping algebra $U(\mathfrak{n})$ of $\mathfrak{n}$ is the associative $\mathbb{C}$-algebra generated by $E_i$ $(1 \leq i \leq n)$ subject to the relations 
\begin{align}
&E_iE_j = E_jE_i, && {\rm for} \ \vert i-j\vert \geq 2, \label{CommForNonNeighbors}\\
&E_i^2E_j-2E_iE_jE_i+E_jE_i^2=0, && {\rm for} \ \vert i-j\vert=1. \label{RelForNeighbors}
\end{align}
The last relation is called {\it Serre relation}.

Similarly, the representation theory of the quiver $Q'$ of type $A_{n-1}$ is linked with the Lie algebra $\mathfrak{g}'=\mathfrak{sl}_n$ with Weyl group $W'$. The Lie algebra $\mathfrak{g}'=\mathfrak{sl}_n$ similarly admits a triangular decomposition $\mathfrak{g}'=\mathfrak{n}'_{-} \oplus \mathfrak{h}' \oplus \mathfrak{n}'$. The Weyl group element associated with $M'$ is $w'=s_1s_3s_5\cdots s_{n-2}s_2s_4s_6 \cdots s_{n-1}s_1s_3s_5 \cdots s_{n-2}s_2s_4s_6 \cdots s_{n-1}\in W'$. The universal enveloping algebra $U(\mathfrak{n}')$ may be viewed as the subalgebra of $U(\mathfrak{n})$ generated by $E_i$ $(1 \leq i \leq n-1)$.

\subsection{The cluster algebra attached to the terminal module}
\label{SubSection:ClusterAlgebra}
To the terminal $\mathbb{C}Q$-module $M$ from Section \ref{RepTheo} Gei\ss -Leclerc-Schr\"oer (\cite[Section 4]{GLS:Unipotent}) attached the subcategory $\mathcal{C}_M=\pi_Q^{-1}(\operatorname{add}(M))\subseteq {\rm nil}(\Lambda)$ of the category of nilpotent $\Lambda$-modules. Here, $\operatorname{add}(M)\subseteq{\rm mod}(\mathbb{C}Q)$ is the subcategory consisting of all modules isomorphic to direct summands of direct sums of finitely many copies of $M$. 

The projective and injective objects in $\mathcal{C}_M$ coincide, so $\mathcal{C}_M$ is a {\it Frobenius category} and the stable category $\underline{\mathcal{C}}_M$ is triangulated according to Happel \cite[Section 2.6]{Happel}. Furthermore, Gei\ss -Leclerc-Schr\"oer \cite[Theorem 11.1]{GLS:Unipotent} showed that there is an equivalence of triangulated categories $\underline{\mathcal{C}}_M \simeq \mathcal{C}_Q$ between $\underline{\mathcal{C}}_M$ and the {\it cluster category} $\mathcal{C}_Q$ as defined by Buan-Marsh-Reineke-Reiten-Todorov \cite{BMRRT} to be the orbit category $\mathcal{D}^b\left({\rm mod}(\mathbb{C}Q)\right)/\tau_{\mathcal{D}}^{-1} \circ [1]$. 

With every $\mathcal{C}_M$ Gei\ss -Leclerc-Schr\"oer \cite[Section 4]{GLS:Unipotent} associated a cluster algebra $\mathcal{A}(\mathcal{C}_M)$; it is constructed as a subalgebra of the graded dual of the universal enveloping algebra of the positive part of the corresponding Lie algebra, i.e., $\mathcal{A}(\mathcal{C}_M) \subseteq U(\mathfrak{n})_{gr}^{\ast}$. For a definition of and a general introduction to cluster algebras see Fomin-Zelevinsky \cite{FZ:Notes}. The cluster algebra $\mathcal{A}(\mathcal{C}_M)$ is also called $\mathcal{A}(w)$.

There is an isomorphism between $U(\mathfrak{n})$ and an algebra $\mathcal{M}$ of $\mathbb{C}$-valued functions on ${\rm mod}(\Lambda)$. We refer to Gei\ss -Leclerc-Schr\"{o}er \cite{GLS:Unipotent} for a precise definition of $\mathcal{M}$. It is generated by functions $d_{\textbf{i}}$ that map a $\Lambda$-module $X$ to the Euler characteristic of the flag variety of $X$ of type $\textbf{i}$. Prominent elements in $\mathcal{A}(\mathcal{C}_M)$ are (under the described isomorphism) the $\delta$-functions of the rigid $\Lambda$-modules $T_{i,[a,b]}$ with $i \in Q_0$ and $0 \leq a \leq b \leq 1$. For $1 \leq i \leq n$ put 
\begin{align*}
&P_i=\delta_{T_{i,[0,1]}};\\
&Y_i=\begin{cases}\delta_{T_{i,[0,0]}}, &{\rm if} \ i \ {\rm is \ odd};\\\delta_{T_{i,[1,1]}}, &{\rm if} \ i \ {\rm is \ even};\end{cases}
\end{align*}
\begin{align*}
&Z_i=\begin{cases}\delta_{T_{i,[0,0]}}, &{\rm if} \ i \ {\rm is \ even};\\\delta_{T_{i,[1,1]}}, &{\rm if} \ i \ {\rm is \ odd}.\end{cases}
\end{align*}
Note that the module $T_{i,[0,1]}$ corresponding to the variable $P_i$ (for $1\leq i\leq n$) is a projective object in the category $\mathcal{C}_M$, but it is in general not projective in ${\rm mod}(\Lambda)$.     

The initial seed of the cluster algebra \ACM \ for the case $n=9$ is shown in Figure \ref{fig:IniCluster}. The vertices represent the cluster variables in the initial cluster, the arrows describe the initial exchange matrix. Just as in Keller's mutation applet \cite{Keller:Java}, the blue vertices are frozen, the red vertices are mutable. The frozen variables $P_1,P_2,\ldots,P_n$ may be viewed as coefficients in the sense of Fomin-Zelevinsky \cite{FZ:ClusterAlgIV}. The cluster algebra \ACM \ is of type $A_n$, and therefore of finite type. Besides the $n$ frozen variables there are $n+\frac{n(n+1)}{2}$ mutable cluster variables grouped into $C_{n+1}=\frac{1}{n+2}\binom{2n+2}{n+1}$ clusters, where $C_ {n+1}$ denotes the $(n+1)^{{\rm th}}$ Catalan number (see Fomin-Zelevinsky \cite[Section 12]{FZ:ClusterAlgII}). The Catalan number $C_{n+1}$ is the number of triangulations of a convex polygon with $n+3$ sides using only diagonals. 

\begin{figure}
\centering
\includegraphics{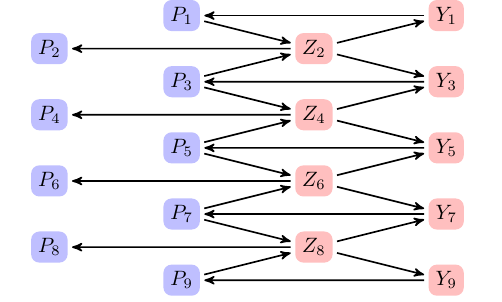}
\caption{The initial seed for the case $n=9$} \label{fig:IniCluster}
\end{figure}

The $\delta$-functions of $P_i$, $Y_i$, and $Z_i$ for $i \in Q_0$ are not algebraically independent. For example, the equation
\begin{align}
\label{InitialExchange}
 P_i=Y_iZ_i-Z_{i-1}Z_{i+1}
\end{align}
holds for every $i\in Q_0$. The equations are due to Gei\ss -Leclerc-Schr\"{o}er \cite[Theorem 18.1]{GLS:Unipotent} and called {\it determinantal identities}. Here and in what follows we use the convention $Z_0=Z_{n+1}=1$.

Similarly, we can construct a cluster algebra $\mathcal{A}(\mathcal{C}_{M'})$ associated with the terminal $\mathbb{C}Q'$-module $M'$ from Section \ref{RepTheo}. The initial seed of $\mathcal{A}(\mathcal{C}_{M'})$ is obtained from the initial seed of $\mathcal{A}(\mathcal{C}_M)$ by ignoring the vertices $Y_n$ and $P_n$ and all incident arrows. We denote the corresponding cluster variables of $\mathcal{A}(\mathcal{C}_{M'})$ by $P'_i$, $Y'_i$, and $Z'_i$ (for $1\leq i \leq n-1)$.

\subsection{The description of cluster variables}
\label{DescriptionOfClusterVariables}
In this subsection we describe the cluster variables explicitly. Note that our desription of cluster variables differs from the explicit description of Gei\ss -Leclerc-Schr\"{o}er \cite[Section 18.2]{GLS:Unipotent} due to a different choice of orientation of the quiver. Put $c_i=\frac{Z_{i-1}Z_{i+1}}{Z_i}$ for $1 \leq i \leq n$.

\begin{definition}
\label{DefDeterminant}
For two natural numbers $i,j$ with $1\leq i \leq j \leq n$ let $M_{ij}=\left((M_{ij})_{rs}\right)_{i \leq r,s \leq j}$ be the $(j-i+1)\times(j-i+1)$ matrix defined by
\begin{align*}
(M_{ij})_{rs}=\begin{cases}
\frac{Y_r}{c_r}, &{\rm if} \ r=s;\\
1, &{\rm if} \ s>r \ {\rm or} \ r=s+1;\\
0, &{\rm otherwise}.
\end{cases}
\end{align*}
Put $\Delta_{i,j}=c_ic_{i+1}\cdots c_j\cdot{\rm det}(M_{ij})$, i.e., $\Delta_{i,j}$ is given by the following determinant:
\begin{align*}\Delta_{i,j}=c_ic_{i+1} \cdots c_j \begin{vmatrix}\frac{Y_i}{c_i}&1&1&1&\cdots&1&1&1\\1&\frac{Y_{i+1}}{c_{i+1}}&1&1&\cdots&1&1&1\\0&1&\frac{Y_{i+2}}{c_{i+2}}&1&\cdots&1&1&1\\0&0&1&\frac{Y_{i+3}}{c_{i+3}}&\cdots&1&1&1\\\vdots&\vdots&\vdots&\vdots&\ddots&\vdots&\vdots&\vdots\\0&0&0&0&\cdots&\frac{Y_{j-2}}{c_{j-2}}&1&1\\0&0&0&0&\cdots&1&\frac{Y_{j-1}}{c_{j-1}}&1\\0&0&0&0&\cdots&0&1&\frac{Y_j}{c_j}\end{vmatrix}.
\end{align*}
\end{definition}

\begin{remark}
Note that
\begin{align*}
\Delta_{i,j}=\begin{vmatrix}Y_i&c_i&c_i&c_i&\cdots&c_i&c_i&c_i\\c_{i+1}&Y_{i+1}&c_{i+1}&c_{i+1}&\cdots&c_{i+1}&c_{i+1}&c_{i+1}\\0&c_{i+2}&Y_{i+2}&c_{i+2}&\cdots&c_{i+2}&c_{i+2}&c_{i+2}\\0&0&c_{i+3}&Y_{i+3}&\cdots&c_{i+3}&c_{i+3}&c_{i+3}\\\vdots&\vdots&\vdots&\vdots&\ddots&\vdots&\vdots&\vdots\\0&0&0&0&\cdots&Y_{j-2}&c_{j-2}&c_{j-2}\\0&0&0&0&\cdots&c_{j-1}&Y_{j-1}&c_{j-1}\\0&0&0&0&\cdots&0&c_j&Y_j\end{vmatrix}
\end{align*}
for $1 \leq i \leq j \leq n$. It follows that each $\Delta_{i,j}$ $(1 \leq i,j \leq n)$ is actually a polynomial in $Y_i$ $(1 \leq i \leq n)$ and $Z_i$ $(1 \leq i \leq n)$, i.e., $\Delta_{i,j} \in \mathbb{Z}[Y_k,Z_k \colon 1 \leq k \leq n]$ for all $i \leq j$. Polynomiality follows from Gei\ss -Leclerc-Schr\"{o}er \cite[Theorem 3.4]{GLS:Unipotent}, but is also follows directly from the formula above once we notice that $c_ic_{i+1} \cdots c_j\in \mathbb{Z}[Z_k \colon 1 \leq k \leq n]$ for all $i,j$ with $1 \leq i < j \leq n$.
\end{remark}

\begin{proposition}
\label{RecursionForClusterVar}
For all $i,j$ with $1 \leq i \leq j \leq n$ and $j-i \geq 3$ the equation $$\Delta_{i,j}=Y_j\Delta_{i,j-1}-Z_{j+1}P_{j-2}\Delta_{i,j-3}$$ holds.
\end{proposition}

\begin{proof}
 Perform a Laplace expansion of the determinant on the last row. The last row has only two non-zero entries and it is easy to see that the two occuring summands in the Laplace expansion are the two summands in the recursion formula.
\end{proof}

For $1 \leq i \leq n$ let $\Delta_{i,i-1}$, $\Delta_{i,i-2}$, and $\Delta_{i,i-3}$ be the unique elements from $\mathbb{Q}(Y_k,Z_k \colon 1 \leq k \leq n)$ such that the recursion formula from Proposition \ref{RecursionForClusterVar} also holds for $j=i+2$, $j=i+1$, and $j=i$. Explicitly, we put $\Delta_{i,i-1}=1$, $\Delta_{i,i-2}=\frac{1}{Y_{i-1}-c_{i-1}}$, and $\Delta_{i,i-3}=0$. The next lemma follows easily from Proposition \ref{RecursionForClusterVar}.

\begin{lemma}
\label{ExchangeRelation}
For all $i,j$ with $1 \leq i \leq j \leq n$ the equation $$\Delta_{i,j}Z_j=P_j\Delta_{i,j-1}+Z_{j+1}P_{j-1}\Delta_{i,j-2}$$ holds.
\end{lemma}

\begin{proof}
Fix $i$. We prove prove Lemma \ref{ExchangeRelation} by induction on $j$. The statement is true for $j=i$ since $Y_iZ_i=P_i+Z_{i+1}Z_{i-1}$. If the statement is true for $j-1$, then by Proposition \ref{RecursionForClusterVar}
\begin{align*}
\Delta_{i,j}Z_j&=Y_jZ_j\Delta_{i,j-1}-Z_{j+1}Z_jP_{j-2}\Delta_{i,j-3}\\
&=P_j\Delta_{i,j-1}+Z_{j+1}Z_{j-1}\Delta_{i,j-1}-Z_{j+1}Z_jP_{j-2}\Delta_{i,j-3}\\
&=P_j\Delta_{i,j-1}+Z_{j+1}P_{j-1}\Delta_{i,j-2},
\end{align*}
and the statement is true for $j$.
\end{proof}

\begin{lemma}
\label{ClusterVariables}
The mutable cluster variables are $Z_1,Z_2,Z_3,\ldots,Z_n$ and $\Delta_{i,j}$ for $1 \leq i \leq j \leq n$.
\end{lemma}

\begin{proof}
Starting with the initial seed (which is shown in Figure \ref{fig:IniCluster} for the case $n=9$) perform mutations at the odd vertices $1,3,5,\ldots,n$, consecutively. In each step, because of the equation $Y_iZ_i=P_i+Z_{i-1}Z_{i+1}$, the cluster variable $Y_i$ for odd $i$ with $1 \leq i \leq n$ is replaced by the cluster variable $Z_i$. Therefore, the mutations generate a seed whose mutable cluster variables are $Z_1,Z_2,Z_3,\ldots,Z_n$. We refer to that seed as the {\it base seed}. The exchange matrix of the base seed is described by the associated quiver. By the rules of quiver mutation the mutable vertices of the base seed form an alternating quiver of type $A_n$ isomorphic to $Q$. The only other arrows are the following. For every $i$ with $1 \leq i \leq n$ there is an arrow between $Z_i$ and $P_i$ starting in $P_i$ if $i$ is odd and starting in $Z_i$ if $i$ is even. The quiver of the base seed for the example $n=9$ is shown in Figure \ref{fig:SnakeCluster}.
\begin{figure}
\centering
\includegraphics{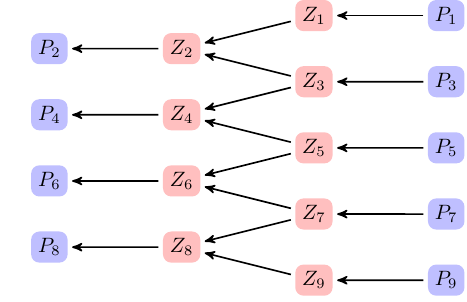}
\caption{The base seed for the case $n=9$} \label{fig:SnakeCluster}
\end{figure}

We now claim that starting from the base seed the cluster variable obtained by consecutive mutations at $i,i+1,i+2,\ldots,j$ is $\Delta_{i,j}$ for all $1 \leq i \leq j \leq n$. The equation $\Delta_{i,j}Z_j=P_j\Delta_{i,j-1}+Z_{j+1}P_{j-1}\Delta_{i,j-2}$ from Lemma \ref{ExchangeRelation} is the exchange relation. For a proof consider the mutation of the quiver of the base seed. Fix $i$. We assume that $i$ is odd. (If $i$ is even, then reverse all arrows in the following argumentation.) We prove the statement by induction on $j$. The statement is true for $i=j$ since mutation at $i$ yields $(P_i+Z_{i-1}Z_{i+1})/Z_i=Y_i=\Delta_{i,i}$. It is also true for $j=i+1$ because $\Delta_{i,i+1}Z_{i+1}=P_{i+1}\Delta_{i,i}+Z_{i+2}P_i\Delta_{i,i-1}=P_{i+1}\Delta_{i,i}+Z_{i+2}P_i$. 

Now assume that $j \geq i+2$ and that consecutive mutations at $i,i+1,i+2,\ldots,j-1$ yield cluster variables $\Delta_{i,i},\Delta_{i,i+1},\ldots,\Delta_{i,j-1}$. Let us describe the quiver after these mutations; let us first concentrate on the subquiver given by all mutable vertices. It is easy to see that the subquiver supported on vertices $(Z_1,Z_2,\ldots,Z_{i-1})$ is the same as in the base quiver; similarly, the subquiver supported on vertices $(Z_j,Z_{j+1},\ldots,Z_n)$ is unchanged. The description of the other remaining part depends on the parity of $j$. If $j$ is even, then it consists of the two sequences $Z_{i-1} \to \Delta_{i,i} \to \Delta_{i,i+2} \to \Delta_{i,i+4} \to \cdots \to \Delta_{i,j-1}$ and $\Delta_{i,j-2} \to \Delta_{i,j-4} \to \cdots \to \Delta_{i,i+3} \to \Delta_{i,i+1}$ and of an oriented triangle $Z_j \to \Delta_{i,j-1} \to \Delta_{i,j-2} \to Z_j$. If $j$ is odd, then it consists of the two sequences $Z_{i-1} \to \Delta_{i,i} \to \Delta_{i,i+2} \to \Delta_{i,i+4} \to \cdots \to \Delta_{i,j-2}$ and $\Delta_{i,j-1} \to \Delta_{i,j-3} \to \cdots \to \Delta_{i,i+3} \to \Delta_{i,i+1}$ and an oriented triangle $Z_j \to \Delta_{i,j-2} \to \Delta_{i,j-1} \to Z_j$. 
\begin{figure}
\centering
\includegraphics{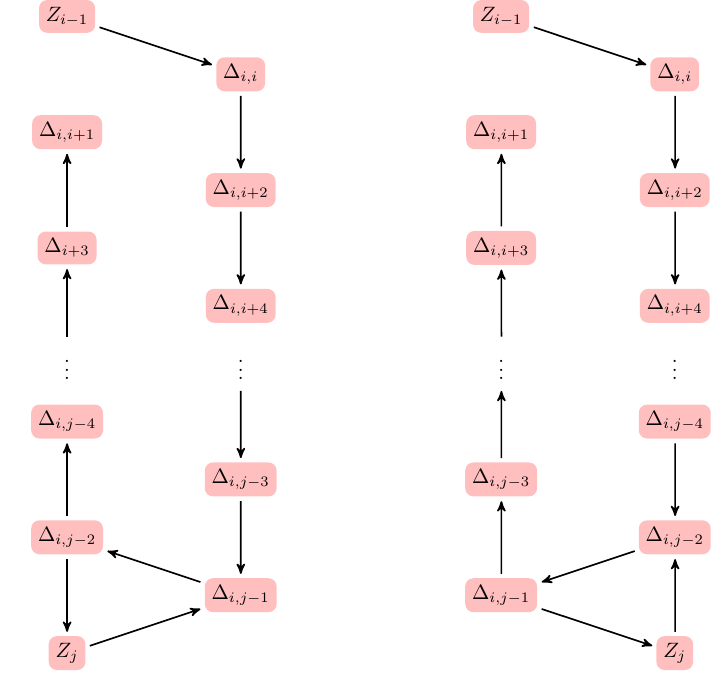}
\caption{The mutated seed for even $j$ (left) and odd $j$ (right)} \label{fig:DeltaCluster}
\end{figure}

Now let us consider frozen vertices. Consider a natural number $k$ with $i \leq k \leq j$. We are interested in the vertices $Z_l$ and $\Delta_{i,l}$ with $k \leq l \leq j-1$ to which $P_k$ is connected. In the base seed the vertex $P_k$ is only connected with $Z_k$. Let us assume that $k$ is even. (If $k$ is odd, then reverse all arrows in the following argumentation.) We have an arrow $Z_k \to P_k$ in the base seed. The adjacency relations for $P_k$ remain unaffected by mutations at $i,i+1\ldots,k-1$. After mutation at $k$ the arrows reverses (and $Z_k$ is replaced by $\Delta_{i,k}$) and we get an additional arrow $Z_{k+1} \to P_k$. Mutation at $k+1$ cancels the arrow $P_k \to \Delta_{i,k}$ whereas the arrow $Z_{k+1} \to P_k$ is replaced by an arrow $P_k \to \Delta_{i,k+1}$. Afterwards all adjacency relations for $P_k$ with vertices $Z_l$ for $k \leq l$ remain unaffected.

The adjacency relations for the vertices together with the induction hypothesis and the mutation rule for cluster variables imply that $(P_j\Delta_{i,j-1}+Z_{j+1}P_{j-1}\Delta_{i,j-2})/Z_j$ is the cluster variable obtained from consecutive mutation at $i,i+1,\ldots,j$. By Lemma \ref{ExchangeRelation} it is equal to $\Delta_{i,j}$.

The number of mutable cluster variables of a cluster algebra of finite type is the sum of the rank of the cluster algebra and the number of positive roots of the associated root system. Since the $n+\frac{n(n+1)}{2}$ cluster variables $Z_1,Z_2,Z_3,\ldots,Z_n$ and $\Delta_{i,j}$ for $1 \leq i \leq j \leq n$ are all distinct these must be all mutable cluster variables.
\end{proof}

By Lemma \ref{ClusterVariables} the recursion provided by Proposition \ref{RecursionForClusterVar} allows to compute iteratively every cluster variable in terms of the $Y_i$ and $Z_i$ $(1 \leq i \leq n)$.

\begin{example}
Let us look at an example. We put $n=3$. The initial cluster contains three mutable and three frozen variables. It is $(P_1,P_2,P_3,Y_1,Z_2,Y_3)$. One can check, by hand or by using Keller's mutation applet \cite{Keller:Java}, that the following figure describes the exchange graph of the cluster algebra \ACM \ in the case $n=3$. This particular exchange graph is known as {\it associahedron} or {\it Stasheff polytope} $K_5$, see Fomin-Reading \cite[Section 3.1]{FR}. The mutable cluster variables are colored red, the frozen cluster variables blue. Beside the $3+3$ initial cluster variables there are $6$ further cluster variables, namely $Z_1$, $Y_2$, $Z_3$, 
\begin{align*}
&\Delta_{1,2}=\begin{vmatrix}Y_1&c_1\\c_2&Y_2\end{vmatrix}=Y_1Y_2-Z_3, \quad \Delta_{2,3}=\begin{vmatrix}Y_2&c_2\\c_3&Y_3\end{vmatrix}=Y_2Y_3-Z_1,\\ &\Delta_{1,3}=\begin{vmatrix}Y_1&c_1&c_1\\c_2&Y_2&c_2\\0&c_3&Y_3\end{vmatrix}=Y_1Y_2Y_3-Y_1Z_1-Y_3Z_3+Z_2. 
\end{align*}
The cluster variables are grouped into $C_4=14$ clusters.
\end{example}

\begin{figure}
\centering
\includegraphics[scale=0.507,angle=90]{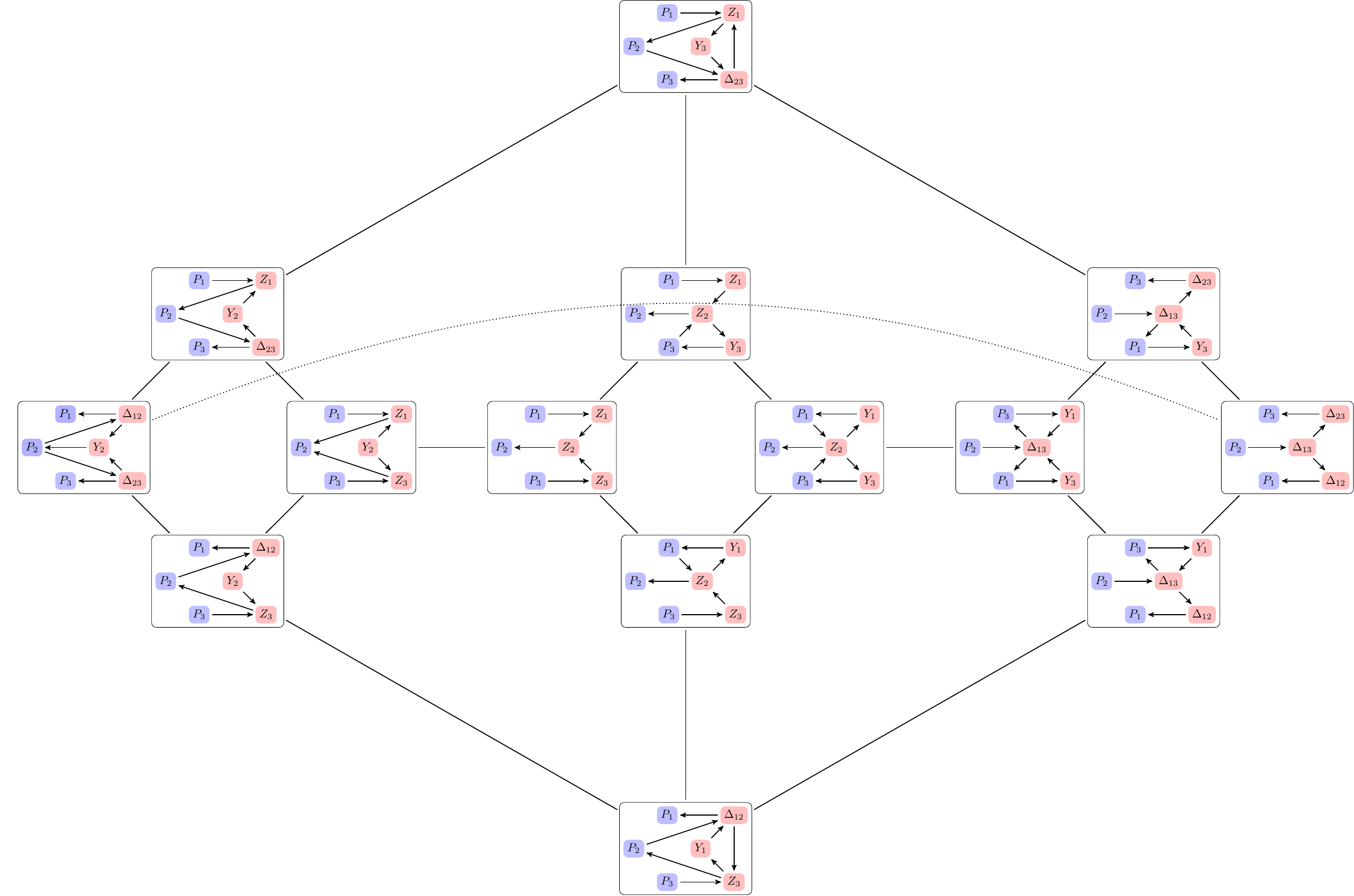}
\caption{The exchange graph for $n=3$: a Stasheff polyhedron} 
\label{Stasheff}
\end{figure}

\begin{remark}Formulae for cluster variables in $\mathcal{A}(\mathcal{C}_{M'})$ can be obtained from these formulae by setting $Y_n=Z_n=P_n=1$.
\end{remark}

\begin{remark}
The exchange relation (\ref{InitialExchange}) implies that mutation of the base seed at $Z_i$ for some $i\in\{1,2,\ldots,n\}$ yields the cluster variable $Y_i$. Note that the base seed is \textit{acyclic}, i.e., the corresponding quiver does not contain oriented cycles. Hence, by Berenstein-Fomin-Zelevinsky \cite[Theorem 1.16]{BFZ} the cluster algebra $\mathcal{A}(\mathcal{C}_{M})$ is equal to its \textit{lower bound}, i.e., the set $\{Y_i,Z_i\colon 1\leq i\leq n\}$ generates $\mathcal{A}(\mathcal{C}_{M})$.  
\end{remark}

\begin{remark}
\label{Rohleder}
The cluster variables correspond to $\delta$-functions of indecomposable rigid objects in $\mathcal{C}_M$. These objects have been classified by Rohleder \cite[Theorem 7.3]{Rohleder}. Besides the $3n$ objects of the form $T_{i,[a,b]}$ for $1 \leq i \leq n$ and $0 \leq a \leq b \leq 1$, these are (when viewed as elements in $C(1,\tau)$) the objects $M_{i,j}$
\begin{align*}
\bigoplus_{\genfrac{}{}{0pt}{}{i< r<j}{r \ {\rm odd}}} I_i \oplus \bigoplus_{\genfrac{}{}{0pt}{}{i< r<j}{r \ {\rm even}}} \tau(I_i) \xrightarrow{f} \bigoplus_{\genfrac{}{}{0pt}{}{i< r<j}{r \ {\rm odd}}} \tau(I_i)\oplus \bigoplus_{\genfrac{}{}{0pt}{}{i< r<j}{r \ {\rm even}}} \tau^2(I_i) 
\end{align*}
for $1< j \leq n$ where $f\vert_{\tau(I_i)}^{\tau(I_{i\pm1})}=1$ for all even $i$ and $f\vert_x^y=0$ for all other direct summands $X,Y$.
\end{remark}

\section{The quantized universal enveloping algebra}
\label{QuantUniv}

\subsection{Definition of the quantized enveloping algebra}
Let $v$ be an indeterminate. The quantized universal enveloping algebra $U_v(\mathfrak{g})$ is a deformation of the ordinary universal enveloping algebra $U(\mathfrak{g})$. To describe this construction we introduce quantized integers and quantized binomial coefficients. 

\begin{definition}
For a natural number $k$, denote by $\left[k \right] =\frac{v^k-v^{-k}}{v-v^{-1}}\in \mathbb{Q}\left( v\right)$ the quantum integer and by $[k]!=[k][k-1]\cdots[1]$ the quantized factorial. For two natural numbers $k$ and $l$, define the quantum binomial coefficient by $$\genfrac[]{0cm}{0}{k}{l}=\frac{\left[k \right]!}{\left[l \right]!\left[k-l \right]!} \in \mathbb{Q}\left( v\right).$$
\end{definition}

\begin{remark}
Both $[k]$ and $\genfrac[]{0cm}{1}{n}{k}$ are actually Laurent polynomials in $v$. If we specialize $v=1$, then $[k]=k$, $\genfrac[]{0cm}{1}{n}{k}=\genfrac(){0cm}{1}{n}{k}$, and $[k]!=k!$. Some authors such as Kac-Cheung \cite{KC:QuantumCalculus} use a different convention for quantum integers.
\end{remark}

\begin{definition} The {\it quantized enveloping algebra} $U_v(\mathfrak{g})$ is the $\mathbb{Q}(v)$-algebra generated by $E_i,F_i,K_i,K_i^{-1}$ for $i=1,2,\ldots,n$, subject to the following relations
\begin{align}
&K_iK_j =K_jK_i, && (i \neq j) \\
&K_iK_i^{-1} = K_i^{-1}K_i = 1, && (i=1,2,\ldots,n)  \\
&K_iE_jK_i^{-1} = v^{a_{ij}}E_j, && (1 \leq i,j \leq n) \label{nice}\\
&K_iF_jK_i^{-1} = v^{-a_{ij}}F_j, && (1 \leq i,j \leq n) \\
&E_iF_j -F_jE_i = \delta_{ij}\frac{K_i-K_i^{-1}}{v-v^{-1}},&& (1 \leq i,j \leq n) \label{commu}\\
&E_i^2E_j-[2]E_iE_jE_i+E_jE_i^2=0, && \vert i-j\vert=1,\label{qs1}\\
&F_i^2F_j-[2]F_iF_jF_i+F_jF_i^2=0, && \vert i-j\vert=1,\label{qs2}\\
&E_iE_j=E_jE_i, && \vert i-j\vert\geq2,\label{qcom1}\\
&F_iF_j=F_jF_i, && \vert i-j\vert\geq2,\label{qcom2}
\end{align}
where $\delta_{ij}$ is the Kronecker delta function. Note that $[2]=v+v^{-1}$, so we may write equation (\ref{qs1}) as $E_i^2E_j-(v+v^{-1})E_iE_jE_i+E_jE_i^2=0$.
\end{definition}

\begin{definition}
The subalgebra generated by $E_i$ for $i=1,2,\ldots,n$ is called the {\it quantized enveloping algebra} $U_v(\mathfrak{n})$.
\end{definition}

The only relations in $U_v(\mathfrak{n})$ remain (\ref{qs1}) and (\ref{qcom1}). These are called {\it quantized Serre relations}. The algebra $U_v(\mathfrak{n})$ specializes to $U(\mathfrak{n})$ in the limit $v=1$. 

\begin{remark}
The algebra $U_v(\mathfrak{g})$ is a {\it graded algebra}. It is graded by the root lattice $R$ if we set ${\rm deg}(E_i)=\alpha_i$, ${\rm deg}(F_i)=-\alpha_i$, and ${\rm deg}(K_i)=0$ for all $1 \leq i \leq n$. Note that ${\rm deg}(A) \in R^+$ for all $A \in U_v(\mathfrak{n})$. We also use the abbreviation ${\rm deg}(A)=\vert A \vert$ for $A \in U_v(\mathfrak{n})$.
\end{remark}

\begin{remark} Put $\sigma(v)=v^{-1}$ and $\sigma(E_i)=E_i$ for all $i$ with $1 \leq i \leq n$. By the symmetry of the relations (\ref{qs1}) and (\ref{qcom1}) in $U_v(\mathfrak{n})$ the map $\sigma$ extends to an algebra antihomomorphism $\sigma \colon U_v(\mathfrak{n}) \to U_v(\mathfrak{n})$, i.e., a $\mathbb{Q}$-linear map $\sigma \colon U_v(\mathfrak{n}) \to U_v(\mathfrak{n})$ such that $\sigma(AB)=\sigma(B)\sigma(A)$ for all $A,B \in U_v(\mathfrak{n})$. By construction $\sigma$ is an antiautomorphism and an involution, i.e., $\sigma^2(A)=A$ for all $A \in U_v(\mathfrak{n})$. 
\end{remark}

\begin{remark}
In literature the deformation parameter $v$ is sometimes called $q$. There are also different sign conventions for the exponent of the deformation parameter. We adopt Lusztig's convention \cite{Lusztig:QuantumGroups}. It matches Leclerc's usage \cite{Leclerc:Shuffles} if we set $q=v^{-1}$. 
\end{remark}

\begin{remark}
The quantized enveloping algebra $U_v(\mathfrak{g}')=U_v(\mathfrak{sl}_n)$ associated with $Q'$ is defined similarly and may be regarded as the subalgebra of $U_v(\mathfrak{g})=U_v(\mathfrak{sl}_{n+1})$ generated by the elements $E_i$, $F_i$, $K_i$, $K_i^{-1}$ for $1 \leq i \leq n-1$. 
\end{remark}

\subsection{The subalgebra $U_v^+(w)$ and the Poincar\'{e}-Birkhoff-Witt basis}
We introduce Lusztig's T-automorphisms. For $1 \leq i \leq j$ put 
\begin{align*}
&T_i(E_j)=\begin{cases}
-K_i^{-1}F_i, &{\rm if} \ i=j;\\
E_jE_i-v^{-1}E_iE_j, &{\rm if} \ \vert i-j \vert =1;\\
E_j, &{\rm if} \ \vert i-j \vert \geq 2;
\end{cases}\\
&T_i(F_j)=\begin{cases}
-E_iK_i, &{\rm if} \ i=j;\\
F_iF_j-vF_jF_i, &{\rm if} \ \vert i-j \vert =1;\\
F_j, &{\rm if} \ \vert i-j \vert \geq 2;
\end{cases}\\
&T_i(K_j)=K_jK_i^{-a_{ij}}.
\end{align*}
Lusztig \cite[Chapter 37]{Lusztig:QuantumGroups} shows that every $T_i$ can be extended to an $\mathbb{Q}(v)$-algebra homomorphism $T_i \colon U_v(\mathfrak{g}) \to U_v(\mathfrak{g})$. (In Lusztig's book \cite{Lusztig:QuantumGroups} it is called $T'_{i,-1}$.) In fact, every $T_i$ is an $\mathbb{Q}(v)$-algebra automorphism. The images of the generators of $U_v(\mathfrak{g})$ under the inverse $T_i^{-1}$ are given by
\begin{align*}
&T_i^{-1}(E_j)=\begin{cases}
-F_iK_i, &{\rm if} \ i=j;\\
E_iE_j-v^{-1}E_jE_i, &{\rm if} \ \vert i-j \vert =1;\\
E_j, &{\rm if} \ \vert i-j \vert \geq 2;
\end{cases}\\
&T_i^{-1}(F_j)=\begin{cases}
-K_i^{-1}E_i, &{\rm if} \ i=j;\\
F_jF_i-vF_iF_j, &{\rm if} \ \vert i-j \vert =1;\\
F_j, &{\rm if} \ \vert i-j \vert \geq 2;
\end{cases}\\
&T_i^{-1}(K_j)=K_jK_i^{-a_{ij}}.
\end{align*}

\begin{remark} If $g \in U_v(\mathfrak{g})$ is homogeneous of degree $\beta$, then $T_i(g)$ is homogeneous of degree $s_i(\beta)$. \end{remark}

\begin{remark} Furthermore, the $T_i$ satisfy {\it braid relations}. For brevity we write $T_iT_j$ for $T_i \circ T_j$ for all $i,j \in Q_0$. The braid relations are
\begin{align*}
T_iT_j&=T_jT_i, &&{\rm if} \ \vert i-j \vert \geq 2;\\
T_iT_jT_i&=T_jT_iT_j, &&{\rm if} \ \vert i-j \vert =1.
\end{align*}
\end{remark}

\begin{definition} To $w=s_1s_3s_5\cdots s_{n}s_2s_4s_6 \cdots s_{n-1}s_1s_3s_5 \cdots s_{n}s_2s_4s_6 \cdots s_{n-1}$ we attach elements in $U_v(\mathfrak{g})$. If $j_1,j_2,\ldots,j_{2n}\in [1,n]$ are indices such that for the reduced expression from above we have $w=s_{j_1}s_{j_2}\cdots s_{j_{2n}}$, then we consider the elements $T_{j_1}T_{j_2}\cdots T_{j_{k-1}}(E_{j_k})$ for $1 \leq k \leq 2n$. Since ${\rm deg}(T_{j_1}T_{j_2}\cdots T_{j_{k-1}}(E_{j_k}))=s_{j_1}s_{j_2} \cdots s_{j_{k-1}}(\alpha_{j_k})=\beta_k$ for all $k$, we introduce the shorthand notation $E(\beta_k)=T_{j_1}T_{j_2}\cdots T_{j_{k-1}}(E_{j_k})$ for all $1 \leq k \leq 2n$. Here, $E(\beta_1)$ is assumed to be $E(\beta_1)=E_{j_1}$.
\end{definition}

\begin{definition}
For $i \in Q_0$ and $a \in \mathbb{N}$ put $E_i^{(a)}=\frac{1}{[a]!}E_i^a \in U_v(\mathfrak{n})$. For a natural number $k$ with $1 \leq k \leq 2n$ and $a \in \mathbb{N}$ put $$E^{(a)}(\beta_k)=T_{j_1}T_{j_2}\cdots T_{j_{k-1}}(E_{j_k}^{(a)})=\frac{1}{[a]!}E(\beta)^a.$$
\end{definition}

The following theorem is due to Lusztig \cite[Theorem 40.2.1]{Lusztig:QuantumGroups}. It also contains the definition of the subalgebra $U_v^+(w)$ which is crucial for our further studies; moreover, it enables us to define the Poincar\'{e}-Birkhoff-Witt basis of $U_v^+(w)$. For an idea of a proof different from Lusztig's \cite{Lusztig:QuantumGroups} see Bergman's diamond lemma \cite{Bergman:Diamond}.

\begin{theorem}
\label{ExistenceOfPBW}
The set $$\mathcal{P}=\left\lbrace  E^{(a_1)}(\beta_1)E^{(a_2)}(\beta_2)\cdots E^{(a_{2n})}(\beta_{2n}) \colon (a_1,a_2,\ldots,a_{2n})\in \mathbb{N}^{2n}\right\rbrace$$ is linearly independent over the field $\mathbb{Q}(v)$. It forms a $\mathbb{Q}(v)$-basis of a $\mathbb{Q}(v)$-subalgebra $U_v^+(w)\subset U_v(\mathfrak{n})$. Moreover, $U_v^+(w)$ is well-defined in the sense that it is independent of the choice of the reduced expression for $w$. If we choose another reduced expression $w=s_{j'_1}s_{j'_2}\cdots s_{j'_{2n}}$ for $w$, then the set of all $$E_{j'_1}^{(a_1)} \cdot T_{j'_1}(E_{j'_2}^{(a_2)}) \cdots T_{j'_1}T_{j'_2}\cdots T_{j'_{2n-1}}(E_{j'_{2n}}^{(a_{2n})})$$ for all sequences $(a_1,a_2,\ldots,a_{2n}) \in \mathbb{N}^{2n}$ is also a basis of the same subalgebra $U_v^+(w)\subset U_v(\mathfrak{n})$.
\end{theorem}

\begin{remark} The basis $\mathcal{P}$ is called the {\it Poincar\'{e}-Birkhoff-Witt basis} of $U_v^+(w)$ associated with the reduced expression. Unlike the canonical basis which we will define later the Poincar\'{e}-Birkhoff-Witt basis $\mathcal{P}$ depends on the choice of the reduced expression for $w$. Every choice of a reduced expression for $w$ induces a bijection between $\mathbb{N}^{2n}$ and a basis for $U_v^+(w)$. In this sense $\mathcal{P}$ is not canonical. The various bijections are called {\it Lusztig parametrizations}.
\end{remark}

\begin{remark} Theorem \ref{ExistenceOfPBW} particularly implies that $E(\beta_k) \in U_v(\mathfrak{n})$ for every $1 \leq k \leq 2n$ which is not abvious from the definition of the $T$-automorphisms. 
\end{remark}

For any vector $\textbf{a}=(a_1,a_2,\ldots,a_{2n}) \in \mathbb{N}^{2n}$ we introduce the shorthand notation $E[\textbf{a}]$ for $E^{(a_1)}(\beta_1)E^{(a_2)}(\beta_2)\cdots E^{(a_{2n})}(\beta_{2n})$. We also use a different notation for $E(\beta_k)$ with $1 \leq k \leq 2n$; namely we put
\begin{align*}
&u_i=T_1T_3T_5\cdots T_{i-2}(E_i), &&{\rm for \ odd} \ i \ {\rm with} \ 1 \leq i \leq n,\\
&v_i=T_1T_3T_5\cdots T_nT_2T_4\cdots T_{i-2}(E_i), &&{\rm for \ even} \ i \ {\rm with} \ 2 \leq i \leq n-1,\\
&w_i=TT_1T_3T_5\cdots T_{i-2}(E_i), &&{\rm for \ odd} \ i \ {\rm with} \ 1 \leq i \leq n,\\
&x_i=TT_1T_3T_5\cdots T_nT_2T_4\cdots T_{i-2}(E_i), &&{\rm for \ even} \ i \ {\rm with} \ 2 \leq  i \leq n-1,
\end{align*}
where $T=T_1T_3T_5\cdots T_nT_2T_4T_6\cdots T_{n-1}$. In what follows we use the convention $T_i={\rm id}_{U_v(\mathfrak{g})}$ for $i\notin Q_0$. Because of the braid relation $T_iT_j=T_jT_i$ for $\vert i-j \vert \geq 2$ and $T_i(E_j)=E_j$, $T_i(F_j)=F_j$, and $T_i(K_j)=K_j$ for $\vert i-j \vert \geq 2$ the formulae simplify to
\begin{align*}
&u_i=E_i,&&{\rm for \ odd} \ i \ {\rm s.t.} \ 1 \leq i \leq n,\\
&v_i=T_{i-1}T_{i+1}(E_i),&&{\rm for \ even} \ i \ {\rm s.t.} \ 2 \leq i \leq n-1,\\
&w_i=T_{i-2}T_iT_{i+2}T_{i-1}T_{i+1}(E_i),&&{\rm for \ odd} \ i \ {\rm s.t.} \ 1 \leq i \leq n,\\
&x_i=T_{i-3}T_{i-1}T_{i+1}T_{i+3}T_{i-2}T_iT_{i+2}T_{i-1}T_{i+1}(E_i),&&{\rm for \ even} \ i \ {\rm s.t.} \ 2 \leq i \leq n-1.
\end{align*}

Note that $w_i=Tu_i$ for all odd $i$ with $1 \leq i \leq n$ and $x_i=Tv_i$ for all even $i$ with $2 \leq i \leq n-1$.

\begin{remark}The degrees of these variables are in bijection with the dimension vectors of the indecomposable direct summands of the terminal module $T$, compare Figure \ref{fig:AR}. To be more precise, let $1\leq i\leq n$. If $i$ is odd, then we have $E(\beta_i)=u_i$ and $E(\beta_{n+i})=w_i$; if $i$ is even, then we have $E(\beta_i)=v_i$ and $E(\beta_{n+i})=x_i$.
\end{remark}

\begin{remark}
Similarly, we can associate elements $E(\beta_k) \in U_v(\mathfrak{n}') \subset U_v(\mathfrak{n})$ for $1 \leq k \leq 2(n-1)$ to the reduced expression $s_1s_3s_5\cdots s_{n-2}s_2s_4s_6 \cdots s_{n-1}$ $s_1s_3s_5 \cdots s_{n-2}$ $\cdot s_2s_4s_6 \cdots s_{n-1}$ of the Weyl group element $w'\in W'$. The elements generate an algebra $U_v^+(w')\subset U_v(\mathfrak{n}')$, and the set of all ordered products of the $E(\beta_k)$ is a Poincar\'{e}-Birkhoff-Witt basis of $U_v^+(w')$ just as above. Elements $u'_i,w'_i$ (for odd $i$ with $1\leq i \leq n-2$) and $v'_i,x'_i$ (for even $i$ with $2\leq i \leq n-1$) in $U_v(\mathfrak{n}')$ are defined analogously. Under the inclusion $U_v(\mathfrak{n}') \subset U_v(\mathfrak{n})$ they are literally the same as the corresponding elements except for
\begin{align*}
&v'_{n-1}=T_{n-2}(E_{n-1}),\\
&w'_{n-2}=T_{n-4}T_{n-2}T_{n-3}T_{n-1}(E_{n-2}),\\
&x'_{n-3}=T_{n-6}T_{n-4}T_{n-2}T_{n-5}T_{n-3}T_{n-1}T_{n-4}T_{n-2}(E_{n-3}),\\
&x'_{n-1}=T_{n-4}T_{n-2}T_{n-3}T_{n-1}T_{n-2}(E_{n-1}).
\end{align*}

\end{remark}

\subsection{The quantum shuffle algebra and Euler numbers}

Rosso \cite{Rosso} noted that we may embed the quantized enveloping algebra $U_v(\mathfrak{n}) \hookrightarrow \mathcal{F}$ into the \textit{quantum shuffle algebra} $\mathcal{F}$. Hence we may view $U_v(\mathfrak{n})$ as a subalgebra of $\mathcal{F}$. As we will see below $\mathcal{F}$ is defined in purely combinatorial terms. As Leclerc \cite[Section 2.5, 2.6]{Leclerc:Shuffles} observes, the embedding $U_v(\mathfrak{n}) \hookrightarrow \mathcal{F}$ is very useful for explicit calculations. 

\begin{definition}
Let $r,s$ be natural numbers. A permutation $\pi\in S_{r+s}$ is called a {\it shuffle} of type $(r,s)$ if $\pi(1)<\pi(2)<\cdots<\pi(r)$ and $\pi(r+1)<\pi(r+2)<\cdots<\pi(r+s)$.
\end{definition}
The following definition is due to Leclerc \cite[Section 2.5]{Leclerc:Shuffles}.

\begin{definition}
For every sequence $(i_1,i_2,\ldots,i_r) \in Q_0^r$ of elements in $Q_0$ of length $r\geq 0$ define a symbol $w[i_1,i_2,\ldots,i_r]$. (Especially, we have a symbol $w[\ ]$ for the empty sequence.) Let $\mathcal{F}$ be the $\mathbb{Q}(v)$-vector space generated by all $w[i_1,i_2,\ldots,i_r]$ for all $r \geq 0$. Define the {\it quantum shuffle product} on two basis elements by
\begin{align*}
w[i_1,i_2,\ldots,i_r]*w[i_{r+1},i_{r+2},\ldots,i_s]=\sum_{\genfrac{}{}{0pt}{}{\pi {\rm \ shuffle}}{{\rm of \ type \ }(r,s)}}v^{e(\pi)}w[i_{\pi(1)},i_{\pi(2)},\ldots,i_{\pi(r+s)}],
\end{align*}
where the function $e \colon S_{r+s} \to \mathbb{Z}$ is defined as $$e(\pi)=\sum_{k \leq r < l,\pi(k)<\pi(l)}(\alpha_{i_{\pi{k}}},\alpha_{i_{\pi(l)}}).$$ It is easy to see that the product is associative. We extend the product bilinearly to a map $*\colon \mathcal{F} \times \mathcal{F} \to \mathcal{F}$. The algebra $(\mathcal{F},*)$ is called the {\it quantum shuffle algebra}.
\end{definition}

\begin{remark}
The shuffle product $*$ on the quantum shuffle algebra $\mathcal{F}$ is not commutative. Hence, $\mathcal{F}$ is a non-commutative $\mathbb{Q}(v)$-algebra, but it degenerates to the classical commutative shuffle algebra when we specialize $v=1$. Furthermore, the quantum shuffle algebra $\mathcal{F}$ is graded by the root lattice if we set ${\rm deg}(w[i])=\alpha_i$ for all $i \in \left\lbrace 1,2,\ldots,n\right\rbrace$. 
\end{remark}

\begin{lemma}
The map $E_i \to w[i]$ extends to an embedding of graded algebras $U_v(\mathfrak{n}) \hookrightarrow \mathcal{F}$. In other words, $U_v(\mathfrak{n})$ is isomorphic to the subalgebra of $\mathcal{F}$ generated by all $w[i]$ for $i \in Q_0$.
\end{lemma}

\begin{proof}
See Leclerc \cite[Theorem 4]{Leclerc:Shuffles}.
\end{proof}

From now on we view $U_v(\mathfrak{n})$ as a subalgebra of $\mathcal{F}$. In the rest of the section we expand the generators $u_i,v_i,w_i,x_i \in U_v^+(w)$ in the shuffle basis. The following elements will be important for the description.

\begin{definition}
For integers $i,j$ such that $1 \leq i \leq j \leq n$ put $$X_{i,j}=T_i^{(-1)^{i-1}}T_{i+1}^{(-1)^{i}}\cdots T_{j-1}^{(-1)^{j-2}}(E_j).$$
By definition, $X_{i,j} \in U_v(\mathfrak{n}) \subset \mathcal{F}$.
\end{definition}

\begin{example}
Let us give some examples of $X_{i,j}$ expanded in the shuffle basis: First of all, we have $X_{1,1}=E_1=u_1=w[1]$. Moreover,
\begin{align*}
X_{1,2}&=T_1(E_2)=E_2E_1-v^{-1}E_1E_2=w[2]*w[1]-v^{-1}w[1]*w[2]\\
&=w[1,2]+v^{-1}w[2,1]-v^{-1}\Big(w[2,1]+v^{-1}w[1,2]\Big)=(1-v^{-2})w[1,2]
\end{align*}
is a second example.
\end{example}

The next lemma shows that we can compute the expansion of $X_{i,j}$ for all pairs $(i,j)$ in the shuffle basis explicitly.
\begin{lemma}
\label{EulerShuffle}
Let $i,j$ be integers such that $1 \leq i \leq j \leq n$. Then
\begin{align*}
X_{i,j}=(1-v^{-2})^{j-i}\sum_{\pi}w[\pi(i),\pi(i+1),\ldots,\pi(j)]
\end{align*}
where the sum runs over all permutations $\pi$ of $\left\lbrace i,i+1,\ldots,j\right\rbrace$ such that for every even number $k$ with $i \leq k \leq j-1$ we have $\pi^{-1}(k)>\pi^{-1}(k+1)$ and for every even number $k$ with $i+1 \leq k \leq j$ we have $\pi^{-1}(k)>\pi^{-1}(k-1)$. 
\end{lemma}

\begin{proof}
By backwards induction on $i$ we see that the $X_{i,j}$ (for $1 \leq i < j \leq n$) satisfy the following recursion:
\begin{align*}
X_{i,j}=\begin{cases}
E_jX_{i,j-1}-v^{-1}X_{i,j-1}E_j, &{\rm if \ } j {\rm \ is \ even};\\
X_{i,j-1}E_j-v^{-1}E_jX_{i,j-1}, &{\rm if \ } j {\rm \ is \ odd}.\\
\end{cases}
\end{align*}

Now fix $i$. We proceed by induction on $j-i$. The statement is trivial for $j=i$. Suppose that $j>i$ and that $X_{i,j-1}=(1-v^{-2})^{j-1-i}\sum_{\pi}w[\pi(i),\pi(i+1),\cdots,\pi(j-1)],$ where the sum is taken over all permutations of $\left\lbrace i,i+1,\ldots,j-1\right\rbrace$ such that for every even number $k$ with $i \leq k \leq j-2$ we have $\pi^{-1}(k)>\pi^{-1}(k+1)$ and for every even number $k$ with $i+1 \leq k \leq j-1$ we have $\pi^{-1}(k)>\pi^{-1}(k-1)$. 

We consider two cases. First of all, assume that $j$ is even. Let $\pi$ be a permutation of $\left\lbrace i,i+1,\ldots,j-1\right\rbrace$ as above. When shuffling the sequence $(j)$ of length $1$ with the sequence $(\pi(i),\pi(i+1),\ldots,\pi(j-1))$ of length $j-i$, we get $j-i+1$ permutations of $\left\lbrace i,i+1,\ldots,j\right\rbrace$. Among these we distinguish two kinds of permutations. The permutations $\pi_1$ where $j$ comes after $j-1$ satisfy $\pi_1^{-1}(k)>\pi_1^{-1}(k+1)$ and $\pi_1^{-1}(k)>\pi_1^{-1}(k-1)$ for all even numbers $k$ such that $i \leq k \leq j-1$ or $i+1 \leq k \leq j$, respectively. Conversely, every permutation $\pi_1$ of $\left\lbrace i,i+1,\ldots,j\right\rbrace$ satisfying these conditions is uniquely obtained from shuffling $(j)$ with a such a sequence $(\pi(i),\pi(i+1),\ldots,\pi(j-1))$ such that $j$ comes after $j-1$.

We also get permutations $\pi_2$ where $j$ occurs before $j-1$. Now we see that 
\begin{align*}
&w[j]*w[\pi(i),\pi(i+1),\ldots,\pi(j-1)]\\&=\sum_{\pi_1}w[\pi_1(i),\ldots,\pi_1(j)]+v^{-1}\sum_{\pi_2}w[\pi_2(i),\ldots,\pi_2(j)],\\
&w[\pi(i),\pi(i+1),\ldots,\pi(j-1)]*w[j]\\&=v^{-1}\sum_{\pi_1}w[\pi_1(i),\ldots,\pi_1(j)]+\sum_{\pi_2}w[\pi_2(i),\ldots,\pi_2(j)].
\end{align*}
It follows by induction hypothesis that $$X_{i,j}=w[j]*X_{i,j-1}-v^{-1}X_{i,j-1}*w[j]=(1-v^{-2})^{j-i}\sum_{\pi_1}w[\pi_1].$$ The other case where $j$ is odd is proved similarly.
\end{proof}

\begin{remark}
The number $a(i,j)$ of permutations of $\left\lbrace i,i+1,\ldots,j-1\right\rbrace$ such that for every even number $k$ with $i \leq k \leq j-2$ we have $\pi^{-1}(k)>\pi^{-1}(k+1)$ and for every even number $k$ with $i+1 \leq k \leq j-1$ we have $\pi^{-1}(k)>\pi^{-1}(k-1)$ only depends on $j-i$. The table displays some values of $a(i,j)$. 

\begin{center}
\begin{tabular}{|c||c|c|c|c|c|c|c|}\hline
j-i&0&1&2&3&4&5&6\\ \hline
$a(i,j)$&1&1&2&5&16&61&272\\ \hline
\end{tabular}
\end{center}The numbers are called {\it Euler numbers}. The sequence of Euler numbers is listed as A000111 in Sloane's Encyclopedia of Integer Sequences \cite{Sloane}. Its exponential generating function is $\sec(x)+\tan(x)$, see Stanley \cite[Proposition 1.61]{Stan}.
\end{remark}

\begin{lemma}
\label{GeneratorsAreX}
The following formulae for the generators of $U_v^+(w)$ are valid:
\begin{align*}
&u_i=E_i, &&{\rm for \ odd \ } i {\rm \ with \ } 1 \leq i \leq n;\\
&v_i=T_{i-1}T_i^{-1}(E_{i+1}), &&{\rm for \ even \ } i {\rm \ with \ } 2 \leq i \leq n-2;\\
&w_1=T_2^{-1}(E_3);\\
&w_i=T_{i-2}T_{i-1}^{-1}T_iT_{i+1}^{-1}(E_{i+2}), &&{\rm for \ odd \ } i {\rm \ with \ } 3 \leq i \leq n-3;\\
&w_n=T_{n-2}(E_{n-1});\\
&x_2=T_2^{-1}T_3T_4^{-1}(E_5);\\
&x_i=T_{i-3}T_{i-2}^{-1}T_{i-1}T_i^{-1}T_{i+1}T_{i+2}^{-1}(E_{i+3}), &&{\rm for \ even \ } i {\rm \ with \ } 4 \leq i \leq n-4;\\
&x_{n-1}=T_{n-4}T_{n-3}^{-1}T_{n-2}(E_{n-1});
\end{align*}
\end{lemma}

\begin{proof}
The equation $u_i=E_i$ for odd $i$ follows from definition. Note that by definition of Lusztig's $T$-automorphisms we have $T_{i+1}(E_i)=T_i^{-1}(E_{i+1})$ for $1 \leq i \leq n-1$ and that $T_{i-1}(E_i)=T_i^{-1}(E_{i-1})$ for $2 \leq i \leq n$. The first equation is equivalent to $T_iT_{i+1}(E_i)=E_{i+1}$, the second one is equivalent to $T_iT_{i-1}(E_i)=E_{i-1}$. Therefore, for even $i$ we have $v_i=T_{i-1}T_{i+1}(E_i)=T_{i-1}T_i^{-1}(E_{i+1})$.

For all further calculations the formula $$T_iT_{i-1}T_{i+1}(E_i)=T_{i-1}^{-1}T_i(E_{i+1})=T_{i+1}^{-1}T_i(E_{i-1})$$ which holds for $2 \leq i \leq n-1$ will be crucial. To verify the formula note that by the braid relation we have $T_{i-1}T_iT_{i-1}T_{i+1}(E_i)=T_iT_{i-1}T_iT_{i+1}(E_i)=T_iT_{i-1}(E_{i+1})=T_i(E_{i+1})$. Application of $T_{i-1}^{-1}$ gives the first part of the equation, the second part is proved analogously.

Now we compute $w_1=T_1T_3T_2(E_1)=T_3T_1T_2(E_1)=T_3(E_2)=T_2^{-1}(E_3)$. Similarly, $w_n=T_{n-2}T_nT_{n-1}(E_n)=T_{n-2}(E_{n-1})$. Furthermore, for odd $i$ with $3 \leq i \leq n-2$ we compute 
\begin{align*}
w_i&=T_{i-2}T_iT_{i+2}T_{i-1}T_{i+1}(E_i)=T_{i-2}T_{i+2}T_iT_{i-1}T_{i+1}(E_i)\\
&=T_{i-2}T_{i+2}T_{i-1}^{-1}T_i(E_{i+1})=T_{i-2}T_{i-1}^{-1}T_iT_{i+2}(E_{i+1})=T_{i-2}T_{i-1}^{-1}T_iT_{i+1}^{-1}(E_{i+2}).
\end{align*}

Moreover, we see that $x_2=T_1T_3T_5T_2T_4T_1T_3(E_2)=T_1T_3T_5T_4T_1^{-1}T_2(E_3)=T_3T_5T_4T_2(E_3)=T_5T_2^{-1}T_3(E_4)=T_2^{-1}T_3T_5(E_4)=T_2^{-1}T_3T_4^{-1}(E_5)$, and 
\begin{align*}
x_{n-1}&=T_{n-4}T_{n-2}T_nT_{n-3}T_{n-1}T_{n-2}T_n(E_{n-1})\\&=T_{n-4}T_{n-2}T_nT_{n-3}T_n^{-1}T_{n-1}(E_{n-2})\\
&=T_{n-4}T_{n-2}T_{n-3}T_{n-1}(E_{n-2})=T_{n-4}T_{n-3}^{-1}T_{n-2}(E_{n-1}).
\end{align*}
Finally, for even $i$ with $4 \leq i \leq n-3$, the equation
\begin{align*}
x_i&=T_{i-3}T_{i-1}T_{i+1}T_{i+3}T_{i-2}T_iT_{i+2}T_{i-1}T_{i+1}(E_i)\\
&=T_{i-3}T_{i-2}^{-1}T_{i-2}T_{i-1}T_{i-2}T_{i+1}T_{i+3}T_{i+2}T_iT_{i-1}T_{i+1}(E_i)\\
&=T_{i-3}T_{i-2}^{-1}T_{i-1}T_{i-2}T_{i-1}T_{i+1}T_{i+3}T_{i+2}T_{i-1}^{-1}T_i(E_{i+1})\\
&=T_{i-3}T_{i-2}^{-1}T_{i-1}T_{i-2}T_{i+1}T_{i+3}T_{i+2}T_i(E_{i+1})\\
&=T_{i-3}T_{i-2}^{-1}T_{i-1}T_{i-2}T_{i+3}T_i^{-1}T_{i+1}(E_{i+2})\\
&=T_{i-3}T_{i-2}^{-1}T_{i-1}T_i^{-1}T_{i-2}T_{i+3}T_{i+1}(E_{i+2})\\
&=T_{i-3}T_{i-2}^{-1}T_{i-1}T_i^{-1}T_{i+1}T_{i+2}^{-1}(E_{i+3}).
\end{align*} holds which is the last equation to be checked.
\end{proof}

\begin{remark}
\label{RemGenerators}
Lemma \ref{GeneratorsAreX} shows all generators $u_i,v_i,w_i,x_i$ (for appropriate $i$) of $U_v^+(w)$ are of the form $X_{i',j'}$ (for appropriate $i',j'$). Hence, the formula of Lemma \ref{EulerShuffle} applies. In each case, $V_{[i',j']}$ is the associated $kQ$-module from Figure \ref{fig:QuiverA}. In other words, in each case ${\rm deg}(X_{i',j'})=\alpha_{i'}+\alpha_{i'+1}+\ldots+\alpha_{j'}$.
\end{remark}

\begin{remark} With the same argument we can conclude that for $Q'$ the equations 
\begin{align*}
&v'_{n-1}=T_{n-2}(E_{n-1})\\
&w'_{n-2}=T_{n-4}T_{n-3}^{-1}T_{n-2}(E_{n-1})\\
&x'_{n-1}=T_{n-4}(E_{n-3})\\
&x'_{n-3}=T_{n-6}T_{n-5}^{-1}T_{n-4}T_{n-3}^{-1}T_{n-2}(E_{n-1})
\end{align*}hold. Hence, Lemma \ref{GeneratorsAreX} and Remark \ref{RemGenerators} also hold true for the generators $u'_i,v'_i,w'_i,x'_i$ (for appropriate $i$) of $U_v^+(w')$.
\end{remark}

\subsection{The straightening relations for the generators of $U_v^+(w)$}
The following lemma expands every $E(\beta_j)E(\beta_i)$ with $1 \leq i < j \leq 2n$ in the Poincar\'{e}-Birkhoff-Witt basis $\mathcal{P}$. The relations of Lemma \ref{StraighteningRelations} are called {\it straightening relations}. Iterative use of the straightening relations from Lemma \ref{StraighteningRelations} allows us to write an arbitrarily ordered monomial in the generators $E(\beta_k)$ with $1 \leq k \leq 2n$ (and hence every element in $U_v^+(w)$) as a linear combination of Poincar\'{e}-Birkhoff-Witt basis elements $E[\textbf{a}]$ with $\textbf{a} \in \mathbb{N}^{2n}$.
\begin{lemma}
\label{StraighteningRelations}
The generators of $U_v^+(w)$ satisfy the following relations
\begin{align*}
&v_{i+1}u_i=vu_iv_{i+1}, && \for i=1,3,\cdots,n-2,\\ 
&v_{i-1}u_i=vu_iv_{i-1}, && \for i=3,5,\ldots,n,\\ \\
&w_{i+2}u_i=vu_iw_{i+2}, && \for i=1,3,\ldots,n-2,\\ 
&w_{i-2}u_i=vu_iw_{i-2}, && \for i=3,5,\ldots,n,\\ 
&w_1u_1=v^{-1}u_1w_1+v_2, &&\\ 
&w_iu_i=u_iw_i+(v-v^{-1})v_{i-1}v_{i+1}, && \for i=3,5,\cdots,n-2,\\ 
&w_nu_n=v^{-1}u_nw_n+v_{n-1}, && \\\\
&x_{i+3}u_i=vu_ix_{i+3}, && \for i=1,3,\ldots,n-4,\\
&x_{i-3}u_i=vu_ix_{i-3}, && \for i=5,7,\ldots,n,\\
&x_{i-1}u_i=u_ix_{i-1}+(v-v^{-1})v_{i+1}w_{i-2}, && \for i=3,5,\ldots,n-2,\\
&x_{n-1}u_n=v^{-1}u_nx_{n-1}+w_{n-2}, &&\\
&x_2u_1=v^{-1}u_1x_2+w_3, &&\\
&x_{i+1}u_i=u_ix_{i+1}+(v-v^{-1})v_{i-1}w_{i+2}, && \for i=3,5,\ldots,n-2,\\\\
&w_{i+1}v_i=vv_iw_{i+1}, && \for i=2,4,\ldots,n-1,\\ 
&w_{i-1}v_i=vv_iw_{i-1}, && \for i=2,4,\ldots,n-1,\\ \\
\end{align*}
\begin{align*}
&x_{i+2}v_i=vv_ix_{i+2}, && \for i=2,4,\ldots,n-3, \\
&x_{i-2}v_i=vv_ix_{i-2}, && \for i=4,6,\ldots,n-1, \\ 
&x_iv_i=v_ix_i+(v-v^{-1})w_{i-1}w_{i+1}, && \for i=2,4,\ldots,n-1,\\\\
&x_{i+1}w_i=vw_ix_{i+1}, && \for i=1,3,\ldots,n-2, \\ 
&x_{i-1}w_i=vw_ix_{i-1}, && \for i=3,5,\ldots,n.
\end{align*} 
For every $i,j$ with $1 \leq i < j \leq 2n$ such that $E(\beta_j)E(\beta_i)$ is not listed on the left-hand side above the commutativity relation $E(\beta_j)E(\beta_i)=E(\beta_i)E(\beta_j)$ holds.
\end{lemma}

\begin{proof}
Let $i$ be an integer with $1 \leq i \leq n-1$. We have 
\begin{align}T_iT_{i+2}(E_{i+1})&=T_i(E_{i+1}E_{i+2}-v^{-1}E_{i+2}E_{i+1})\nonumber\\
&=(E_{i+1}E_i-v^{-1}E_iE_{i+1})E_{i+2}-v^{-1}E_{i+2}(E_{i+1}E_i-v^{-1}E_iE_{i+1})\nonumber\\
&=E_{i+1}E_iE_{i+2}-v^{-1}E_iE_{i+1}E_{i+2}\nonumber\\&\quad-v^{-1}E_{i+2}E_{i+1}E_i+v^{-2}E_iE_{i+2}E_{i+1}.
\label{PfStraighten1}
\end{align}
Now let $i$ be an odd integer with $1 \leq i \leq n-2$. Then  $u_i=E_i$ and $v_{i+1}=T_iT_{i+2}(E_{i+1})$. By equation (\ref{PfStraighten1}) the following relations hold:
\begin{align*}
&v_{i+1}u_i=\left(E_{i+1}E_i^2-v^{-1}E_iE_{i+1}E_i\right)E_{i+2}+E_{i+2}\left(v^{-2}E_iE_{i+1}E_i-v^{-1}E_{i+1}E_i^2\right),\\
&u_iv_{i+1}=\left(E_iE_{i+1}E_i-v^{-1}E_i^2E_{i+1}\right)E_{i+2}+E_{i+2}\left(v^{-2}E_i^2E_{i+1}-v^{-1}E_iE_{i+1}E_i\right).
\end{align*}
By equation (\ref{qs1}) we get $v_{i+1}u_i-vu_iv_{i+1}=0$. The equation $v_{i-1}u_i=vu_iv_{i-1}$, for odd $i$ with $3 \leq i \leq n$, is proved analogously. Application of $T$ to the last two equations yields $x_{i+1}w_i=vw_ix_{i+1}$, for odd integers $i$ with $1 \leq i \leq n-2$, and $x_{i-1}w_i=vw_ix_{i-1}$, for odd integers $i$ with $3 \leq i \leq n$.

Now let $i$ be an even integer with $2 \leq i \leq n-1$. Then $v_i=T_{i-1}T_{i+1}(E_i)$ and $w_{i+1}=T_{i-3}T_{i-1}T_{i+1}T_{i+3}T_{i}T_{i+2}(E_{i+1})$. With he same argument as above we have $T_iT_{i+2}(E_{i+1})E_i=vE_iT_iT_{i+2}(E_{i+1})$. Apply the automorphism $T_{i-3}T_{i-1}T_{i+1}T_{i+3}$ to the last equation. Using the equation $T_{i-3}T_{i-1}T_{i+1}T_{i+3}(E_i)=T_{i-1}T_{i+1}(E_i)=v_i$ we get $w_{i+1}v_i=vv_iw_{i+1}$. Similarly, the equation $w_{i-1}v_i=vv_iw_{i-1}$ holds. 

Let $i$ be an integer with $1 \leq i \leq n-1$. We have 
\begin{align*}
-T_{i+1}(E_{i+2})F_iK_i&=-v^{-1}(E_{i+2}E_{i+1}-v^{-1}E_{i+1}E_{i+2})F_iK_i\\
&=-F_iK_i(E_{i+2}E_{i+1}-v^{-1}E_{i+1}E_{i+2})=-F_iK_iT_{i+1}(E_{i+2}). 
\end{align*}
Application of the composition of automorphisms $T_iT_{i+2}T_{i+4}T_{i+3}$ yields 
\begin{align}
T_iT_{i+2}T_{i+4}T_{i+1}T_{i+3}(E_{i+2})E_i=vE_iT_iT_{i+2}T_{i+4}T_{i+1}T_{i+3}(E_{i+2}).
\label{PfStraighten2}
\end{align}
Now let $i$ be more specifically an odd integer with $1 \leq i \leq n-2$. The previous equation (\ref{PfStraighten2}) asserts that $w_{i+2}u_i=vu_iw_{i+2}$. The equation $w_{i-2}u_i=vu_iw_{i-2}$, for odd $i$ with $3 \leq i \leq n$, is proved analogously. Now let $i$ be an even integer with $2 \leq i \leq n-3$. From (\ref{PfStraighten2}) we see after application of $T_{i-1}T_{i+1}T_{i+3}T_{i+5}$ that $x_{i+2}v_i=vv_ix_{i+2}$. The equation $x_{i-2}v_i=vv_ix_{i-2}$, for even $i$ with $4 \leq i \leq n-1$, is proved analogously.

Furthermore, there holds:
\begin{align*}
&-F_1K_1(E_1E_2-v^{-1}E_2E_1)\\
&\quad =-vF_1E_1E_2K_1+F_1E_2E_1K_1\\
&\quad =v\left(\frac{K_1-K_1^{-1}}{v-v^{-1}}-E_1F_1\right)E_2K_1+E_2\left(E_1F_1-\frac{K_1-K_1^{-1}}{v-v^{-1}}\right)K_1\\
&\quad =-v(E_1E_2-v^{-1}E_2E_1)F_1K_1+\frac{1}{v-v^{-1}}(E_2K_1^2-v^2E_2-E_2K_1^2+E_2)\\
&\quad =-v(E_1E_2-v^{-1}E_2E_1)F_1K_1-vE_2.
\end{align*}
Application of the map $T_1T_3$ yields to $u_1w_1=vw_1u_1-vv_2$ which is equivalent to $w_1u_1=v^{-1}u_1w_1+v_2$. The next equation $w_nu_n=v^{-1}u_nw_n+v_{n-1}$ is proved analogously.

Let $i$ be an integer with $2 \leq i \leq n-1$. Introduce the abbreviation 
\begin{align*}
S&=T_{i-1}T_{i+1}(E_i)\\&=E_iE_{i-1}E_{i+1}-v^{-1}E_{i-1}E_iE_{i+1}-v^{-1}E_{i+1}E_iE_{i-1}+v^{-2}E_{i-1}E_{i+1}E_i. 
\end{align*}
We have $K_iS=SK_i$, so $-F_iK_iS$ is equal to
\begin{align}
&(-F_iE_iE_{i-1}E_{i+1}+v^{-1}E_{i-1}F_iE_iE_{i+1}\nonumber\\
&\hspace{2cm}+v^{-1}E_{i+1}F_iE_iE_{i-1}-v^{-2}E_{i-1}E_{i+1}F_iE_i)K_i\nonumber\\
&=-SF_iK_i+\frac{1}{v-v^{-1}}\Big[(K_i-K_i^{-1})E_{i-1}E_{i+1}-v^{-1}E_{i-1}(K_i-K_i^{-1})E_{i+1}\nonumber\\
&\hspace{2cm}-v^{-1}E_{i+1}(K_i-K_i^{-1})E_{i-1}+v^{-2}E_{i-1}E_{i+1}(K_i-K_i^{-1})\Big]\nonumber\\
&=-SF_iK_i+\frac{1}{v-v^{-1}}\Big[(v^{-2}-v^{-2}-v^{-2}+v^{-2})E_{i-1}E_{i+1}K_i^2\nonumber\\
&\hspace{6cm}+(-v^2+2-v^{-2})E_{i-1}E_{i+1}\Big]\nonumber\\
&=-SF_iK_i+(v^{-1}-v)E_{i+1}E_{i-1}.
\label{PfStraighten3}
\end{align}
Now let $i$ be more specifically an odd integer with $3 \leq i \leq n-2$. After application of $T_{i-2}T_iT_{i+2}$ the previous equation (\ref{PfStraighten3}) asserts that $u_iw_i=w_iu_i+(v^{-1}-v)v_{i+1}v_{i-1}$. If $i$ is an even integer with $2 \leq i \leq n-1$, then application of the composition $T_{i-3}T_{i-1}T_{i+1}T_{i+3}T_{i-2}T_iT_{i+2}$ to equation (\ref{PfStraighten3}) yields $v_ix_i=x_iv_i+(v^{-1}-v)w_{i+1}w_{i-1}$.

Let $i$ be an odd integer with $1 \leq i \leq n-4$. Since $T_{i+1}T_{i+2}T_{i+4}(E_{i+3})$ is a linear combination of monomials in $E_{i+1}$, $E_{i+2}$, $E_{i+3}$, and $E_{i+4}$ with each factor appearing once, we see that 
\begin{align}
-T_{i+1}T_{i+2}T_{i+4}(E_{i+3})F_iK_i=-vF_iK_iT_{i+1}T_{i+2}T_{i+4}(E_{i+3}).
\label{PfStraighten4}
\end{align}
Applying $T_iT_{i+2}T_{i+4}T_{i+6}T_{i+3}T_{i+5}$ to (\ref{PfStraighten4}) yields $x_{i+3}u_i=vu_ix_{i+3}$. The equation $x_{i-3}u_i=vu_ix_{i-3}$, for odd $i$ with $5 \leq i \leq n$, is proved analogously.

Consider the three elements $T_2^{-1}(-F_1K_1)$, $T_1T_3(E_2)$, and $E_3$. We introduce the abbreviation $X=T_2^{-1}(-F_1K_1)=(vF_2F_1-F_1F_2)K_1K_2$. We have
\begin{align*}
(vF_2F_1-F_1F_2)E_1&=vF_2\left(E_1F_1-\frac{K_1-K_1^{-1}}{v-v^{-1}}\right)-\left( E_1F_1-\frac{K_1-K_1^{-1}}{v-v^{-1}}\right)F_2\\
&=E_1(vF_2F_1-F_1F_2)\\&\quad+\frac{1}{v-v^{-1}}\left[-vF_2(K_1-K_1^{-1})+(K_1-K_1^{-1})F_2\right]\\
&=E_1(vF_2F_1-F_1F_2)+F_2K_1^{-1}.
\end{align*}
Therefore, $XE_1=vE_1X+vF_2K_2$. Similarly, $XE_2=vE_2X+vF_1K_1K_2^2$. Furthermore, $XE_3=v^{-1}E_3X$. Hence
\begin{align*}
&X(E_2E_1-v^{-1}E_1E_2)\\
&\quad=(vE_2X+vF_1K_1K_2^2)E_1-(E_1X+F_2K_2)E_2\\
&\quad=v(vE_1X+vF_2K_2)+v^2F_1K_1E_1-E_1(vE_2X+vF_1K_1K_2^2)-F_2K_2E_2\\
&\quad=v^2(E_2E_1-v^{-1}E_1E_2)X+v^2(E_2F_2-F_2E_2)K_2+v(F_1E_1-E_1F_1)K_1K_2^2.
\end{align*}
Introduce the abbreviations $Y=E_2E_1-v^{-1}E_1E_2$ and $R=v^2(E_2F_2-F_2E_2)K_2+v(F_1E_1-E_1F_1)K_1K_2^2=\frac{1}{v-v^{-1}}\Big[v^2K_2^2-v^2-vK_1^2K_2^2+vK_2^2\Big]$.
It follows that $RE_3-v^{-2}E_3R=-vE_3$. Note that $T_1T_3(E_2)=YE_3-v^{-1}E_3Y$ is equal to a $v$-commutator. Hence
\begin{align}
XT_1T_3(E_2)&=X(YE_3-v^{-1}E_3Y)\nonumber\\
&=(v^2YX+R)E_3-v^{-2}E_3(v^2Yx+R)\nonumber\\
&=v(YE_3-v^{-1}E_3Y)-RE_3-v^{-2}E_3R\nonumber\\
&=vT_1T_3(E_2)-vE_3.
\label{PfStraighten5}
\end{align}
Application of $T_1T_3T_5T_2T_4$ to equation (\ref{PfStraighten5}) yields to $u_1x_2=vx_2u_1-vw_3$ which is equivalent to $x_2u_1=v^{-1}u_1x_2+w_3$. The equation $x_{n-1}u_n=v^{-1}u_nx_{n-1}+w_{n-2}$ is proved analogously.

Now let $i$ be an odd integer such that $3 \leq i \leq n-2$. Let us consider the four elements $T_{i+1}^{-1}(-F_iK_i)$, $T_{i-1}T_iT_{i+2}(E_{i+1})$, $E_{i+2}$, and $E_{i-1}$. Now we denote by $X$ the element $T_{i+1}^{-1}(-F_iK_i)$. Similarly as above, the equations $XE_i=vE_iX+vF_{i+1}K_{i+1}$, $XE_{i+1}=vE_{i+1}X+vF_iK_iK_{i+1}^2$, and $XE_{i+2}=v^{-1}E_{i+2}X$ hold. Furthermore, we see that $XE_{i-1}=(vF_{i+1}F_i-F_iF_{i+1})K_iK_{i+1}E_{i-1}=v^{-1}E_{i-1}X$. Note that 
\begin{align*}
&T_{i-1}T_iT_{i+2}(E_{i+1})\\
&\quad =T_{i-1}(E_{i+1}E_iE_{i+2}-v^{-1}E_iE_{i+1}E_{i+2}-v^{-1}E_{i+2}E_{i+1}E_i+v^{-2}E_iE_{i+2}E_{i+1})\\
&\quad=T_iT_{i+2}(E_{i+1})E_{i-1}-v^{-1}E_{i-1}T_iT_{i+2}(E_{i+1}).
\end{align*}
With the same argument as above one can prove that $$XT_iT_{i+2}(E_{i+1})=vT_iT_{i+2}(E_{i+1})X-vE_{i+2}.$$ From this equation it follows that
\begin{align}
&XT_{i-1}T_iT_{i+2}(E_{i+1})\nonumber\\
&=(vT_iT_{i+2}(E_{i+1})X-vE_{i+2})E_{i-1}-v^{-2}E_{i-1}(vT_iT_{i+2}(E_{i+1})X-vE_{i+2})\nonumber\\
&=(T_iT_{i+2}(E_{i+1})E_{i-1}-v^{-1}E_{i-1}T_iT_{i+2}(E_{i+1}))X+(v^{-1}-v)E_{i+2}E_{i-1}.
\label{PfStraighten6}
\end{align}
Application of the automorphism $T_{i-2}T_iT_{i+2}T_{i+4}T_{i+1}T_{i+3}$ to equation (\ref{PfStraighten6}) gives $u_ix_{i+1}=x_{i+1}u_i+(v^{-1}-v)w_{i+2}v_{i-1}$. The equation $u_ix_{i-1}=x_{i-1}u_i+(v^{-1}-v)w_{i-2}v_{i+1}$ is proved analogously.

After multiplying with appropriate $T$-automorphisms, all others pairs $E(\beta_i)$ and $E(\beta_j)$ of generators become $\mathbb{Q}(v)$-linear combinations of momomials $E_{i_1}E_{i_2}\cdots E_{i_k}$ and $E_{i'_1}E_{i'_2}\cdots E_{i'_{k'}}$, respectively, where the two occuring sequences $(i_1,i_2,\ldots,i_k)$ and $(i'_1,i'_2,\ldots,i'_{k'})$ of indices come from two intervals of distance at least two. Hence they commute.
\end{proof}

\begin{remark}
\label{HighestExponent}
In the straightening relations of Lemma \ref{StraighteningRelations}, for all $i,j$ with $1 \leq i < j \leq 2n$, the coefficient in front of $E(\beta_i)E(\beta_j)$ in the expansion of $E(\beta_j)E(\beta_i)$ in the Poincar\'{e}-Birkhoff-Witt basis is $v^{(\beta_i,\beta_j)}$. 
\end{remark}

\begin{remark}
Similarly, there are straightening relations for the generators $u'_i$, $v'_i$, $w'_i$, $x'_i$ (for appropriate $i$) of $U_v^+(w')$ that enable us to expand every element of $U_v^+(w')$ in the Poincar\'{e}-Birkhoff-Witt basis. Let us describe them. First of all, note that $v'_{n-1}=T_n^{-1}(v_{n-1})$, $w'_{n-1}=T_n^{-1}w_{n-1}$, $x'_{n-3}=T_n^{-1}x_{n-3}$, and that $T_n^{-1}$ leaves all generators invariant except for $v'_{n-1}$,  $w'_{n-1}$, $x'_{n-3}$, and $x'_{n-1}$. Therefore, the straightening relations of $U_v^+(w')$ are the same as the ones for $U_v^+(w)$ except for the straightening relations involving $x'_{n-1}$. 

Calculations similar to those in Lemma \ref{StraighteningRelations} show that the straightening relations involving $x'_{n-1}$ are commutativity relations except for:
\begin{align*}
&x'_{n-1}w'_{n-2}=vw'_{n-2}x'_{n-1},\\
&x'_{n-1}v'_{n-3}=vv'_{n-3}x'_{n-1},\\
&x'_{n-1}v'_{n-1}=v^{-1}v'_{n-1}x'_{n-1}+w'_{n-2},\\
&x'_{n-1}u'_{n-2}=v^{-1}u'_{n-2}x'_{n-1}+v'_{n-3},\\
&x'_{n-1}u'_{n-4}=vu'_{n-4}x'_{n-1}.
\end{align*}Note that Remark \ref{HighestExponent} is also true in this case.
\end{remark}

\begin{remark}
The commutation exponent of Remark \ref{HighestExponent} and a weaker (non-explicit) form the straightening relations of Lemma \ref{StraighteningRelations} is given by the Lemma of Levendorki\u{\i}-Soibelman \cite[Proposition 5.5.2]{LevendorskiiSoibelman}. See also Kimura \cite[Theorem 4.24]{Kimura}.
\end{remark}

\subsection{The dual Poincar\'{e}-Birkhoff-Witt basis}

Kashiwara \cite{Kashiwara} introduced operators $E_i'\in {\rm End}(U_v(\mathfrak{n}))$ for $1 \leq i \leq n$ such that the following two properties hold: First of all, $E_i'(E_j)=\delta_{i,j}$ for all $i,j$. Secondly, the Leibniz rule $E_i'(xy)=E_i'(x)y+v^{(\alpha_i,\vert x\vert)}xE'_i(y)$ holds for all $i$ and all homogeneous elements $x,y \in U_v(\mathfrak{n})$. Furthermore, Kashiwara \cite{Kashiwara} introduced a non-degenerate symmetric bilinear form $$\left(\cdot,\cdot \right) \colon U_v(\mathfrak{n}) \times U_v(\mathfrak{n}) \to \mathbb{Q}(v).$$ It is characterized by the assumption that the endomorphism $E_i'$ of $U_v(\mathfrak{n})$ is adjoint to the left multiplication with $E_i$, i.e., $(E'_i(x),y)=(x,E_iy)$ for all $x,y \in U_v(\mathfrak{n})$ and $i \in \left\lbrace 1,2,\ldots,n\right\rbrace$. 

The algebra $U_v(\mathfrak{n})$ is a {\it Hopf algebra}. The $E_i'$ may be viewed as elements in the graded dual Hopf algbera $U_v(\mathfrak{n})_{gr}^{\ast}$. We refer to Berenstein-Zelevinsky \cite[Appendix]{BZ:StringBases} for details.

Lusztig \cite[Section 1.2]{Lusztig:QuantumGroups} defined a different non-degenerate symmetric bilinear form $\left(\cdot,\cdot \right)_{L} \colon U_v(\mathfrak{n}) \times U_v(\mathfrak{n}) \to \mathbb{Q}(v)$. In this paper we use Kashiwara's form. Both forms can be compared, see Leclerc \cite[Section 2.2]{Leclerc:Shuffles}.

According to Lusztig \cite[Proposition 38.2.3]{Lusztig:QuantumGroups} the Poincar\'{e}-Birkhoff-Witt basis is orthogonal with respect to Lusztig's bilinear form. The comparison between both forms shows that it is also orthogonal with respect to Kashiwara's bilinear form. More precisely, for all $\beta, \gamma \in \Delta^+$, we have (see Leclerc \cite[Equation 21]{Leclerc:Shuffles} where the author uses the variable $q^{-1}$ intead of $v$) $\left(E(\beta),E(\gamma)\right)=0$ for $\beta \neq \gamma$, and $$\left(E(\beta),E(\beta)\right)=\frac{\prod_{i=1}^n(1-v^{-(\alpha_i,\alpha_i)})^{d_i}}{1-v^{-(\beta,\beta)}}$$ for $\beta=\sum_{i=1}^n d_i\alpha_i \ {\rm with \ } d_i \in \mathbb{N}$. Compare also with Kimura \cite[Proposition 4.18]{Kimura}.

\begin{remark}
Every $\beta \in \Delta^+$ fulfills $(\beta,\beta)=2$; every $\alpha_i$ with $i \in Q_0$ fulfills $(\alpha_i,\alpha_i)=2$. Hence, if $\beta=\alpha_i+\alpha_{i+1}+\cdots+\alpha_j \in \Delta^+$, then $(E(\beta),E(\beta))=(1-v^{-2})^{j-i}$.
\end{remark}

\begin{definition}
The {\it dual Poincar\'{e}-Birkhoff-Witt basis} $\mathcal{P}^*$ of $U_v^+(w)$ is defined to be the basis adjoint to the Poincar\'{e}-Birkhoff-Witt basis with respect to Kashiwara's form. For every natural number $k$ with $1 \leq k \leq 2n$ we denote by $E^*(\beta_k) \in \mathcal{P}^*$ the dual of $E(\beta_k) \in \mathcal{P}$, i.e., the unique scalar multiple of $E(\beta_k)$ such that $(E(\beta_k),E^*(\beta_k))=1$.
\end{definition}

\begin{remark}
\label{ShuffleExpForDuals}
Assume that $k$ is an integer with $1 \leq k \leq 2n$. Let $i,j$ be the integers with $1 \leq i \leq j \leq n$ such that we can write $\beta_k\in\Delta^+$ as $\beta_k=\alpha_i+\alpha_{i+1}+\ldots+\alpha_j$. By Lemma \ref{EulerShuffle} we have 
\begin{align*}
E^*(\beta_k)=(1-v^{-2})^{i-j}X_{i,j}=\sum_{\pi}w[\pi(i),\pi(i+1),\ldots,\pi(j)]
\end{align*}
where the sum runs over all permutations $\pi$ of $\left\lbrace i,i+1,\ldots,j\right\rbrace$ such that for every even number $k$ with $i \leq k \leq j-1$ we have $\pi^{-1}(k)>\pi^{-1}(k+1)$ and for every even number $k$ with $i+1 \leq k \leq j$ we have $\pi^{-1}(k)>\pi^{-1}(k-1)$. 
\end{remark}

\begin{definition}
We also introduce a shorthand notation for $E^*(\beta_k)$ with $1 \leq k \leq 2n$. To reflect similarities with the cluster algebra from Section \ref{SubSection:ClusterAlgebra} we put
\begin{align*}
&y_i=u_i^*, &&{\rm for} \ {\rm odd} \ i \ {\rm with} \ 1 \leq i \leq n,\\ 
&z_i=v_i^*, &&{\rm for} \ {\rm even} \ i \ {\rm with} \ 2 \leq i \leq n-1,\\ 
&z_i=w_i^*, &&{\rm for} \ {\rm odd} \ i \ {\rm with} \ 1 \leq i \leq n,\\ 
&y_i=x_i^*, &&{\rm for} \ {\rm even} \ i \ {\rm with} \ 2 \leq i \leq n-1.
\end{align*}
\end{definition}

\begin{remark}
Let $1\leq i\leq n$. If $i$ is odd, then we have $E^*(\beta_i)=y_i$ and $E^*(\beta_{n+i})=z_i$; if $i$ is even, then we have $E^*(\beta_i)=z_i$ and $E^*(\beta_{n+i})=y_i$.
\end{remark}

\begin{remark}
The straightening relations of Lemma \ref{StraighteningRelations} now become
\begin{align*}
&z_{i+1}y_i=vy_iz_{i+1}, && {\rm for} \ i \ {\rm odd \ with} \ 1 \leq i \leq n-2,\\ 
&z_{i-1}y_i=vy_iz_{i-1}, && {\rm for} \ i \ {\rm odd \ with} \ 3 \leq i \leq n,\\ \\
&z_{i+2}y_i=vy_iz_{i+2}, && {\rm for} \ i \ {\rm odd \ with} \ 1 \leq i \leq n-2,\\  
&z_{i-2}y_i=vy_iz_{i-2}, && {\rm for} \ i \ {\rm odd \ with} \ 3 \leq i \leq n,\\   
&z_1y_1=v^{-1}y_1z_1+(1-v^{-2})z_2, &&\\ 
\end{align*}
\begin{align*}
&z_iy_i=y_iz_i+(v-v^{-1})z_{i-1}z_{i+1}, && {\rm for} \ i \ {\rm odd \ with} \ 3 \leq i \leq n-2,\\ 
&z_ny_n=v^{-1}y_nz_n+(1-v^{-2})z_{n-1}, && \\\\
&y_{i+3}y_i=vy_iy_{i+3}, && {\rm for} \ i \ {\rm odd \ with} \ 1 \leq i \leq n-4,\\ 
&y_{i-3}y_i=vy_iy_{i-3}, && {\rm for} \ i \ {\rm odd \ with} \ 5 \leq i \leq n,\\ 
&y_{i-1}y_i=y_iy_{i-1}+(v-v^{-1})z_{i+1}z_{i-2}, && {\rm for} \ i \ {\rm odd \ with} \ 3 \leq i \leq n-2,\\ 
&y_{n-1}y_n=v^{-1}y_ny_{n-1}+(1-v^{-2})z_{n-2}, &&\\
&y_2y_1=v^{-1}y_1y_2+(1-v^{-2})z_3, &&\\ 
&y_{i+1}y_i=y_iy_{i+1}+(v-v^{-1})z_{i-1}z_{i+2}, && {\rm for} \ i \ {\rm odd \ with} \ 3 \leq i \leq n-2,\\\\
&z_{i+1}z_i=vz_iz_{i+1}, && {\rm for} \ i \ {\rm even \ with} \ 2 \leq i \leq n-1,\\
&z_{i-1}z_i=vz_iz_{i-1}, && {\rm for} \ i \ {\rm even \ with} \ 2 \leq i \leq n-1,\\\\
&y_{i+2}z_i=vz_iy_{i+2}, && {\rm for} \ i \ {\rm even \ with} \ 2 \leq i \leq n-3,\\
&y_{i-2}z_i=vz_iy_{i-2}, && {\rm for} \ i \ {\rm even \ with} \ 4 \leq i \leq n-1,\\ 
&y_iz_i=z_iy_i+(v-v^{-1})z_{i-1}z_{i+1}, && {\rm for} \ i \ {\rm even \ with} \ 2 \leq i \leq n-1,\\\\
&y_{i+1}z_i=vz_iy_{i+1}, && {\rm for} \ i \ {\rm odd \ with} \ 1 \leq i \leq n-2,\\ 
&y_{i-1}z_i=vz_iy_{i-1}, && {\rm for} \ i \ {\rm odd \ with} \ 3 \leq i \leq n. 
\end{align*} 
\end{remark}

\begin{definition}
For every element ${\textbf a} \in \mathbb{N}^{2n}$ denote by $E[{\textbf a}]^*\in U_v^+(w)$ the dual of $E[{\textbf a}] \in \mathcal{P}$, i.e., the unique scalar multiple of $E[{\textbf a}]$ such that $(E[{\textbf a}],E[{\textbf a}]^*)=1$. 

Consider $U_v^{+}(w)_{\mathbb{Z}}=\bigoplus_{{\textbf a} \in \mathbb{N}^{2n}}\mathbb{Z}[v,v^{-1}]E[{\textbf a}]^*$, the integral form of $U_v^{+}(w)$. Furthermore, put $\mathcal{A}(w)_1=\mathbb{Q} \otimes_{\mathbb{Z}[v,v^{-1}]} U_v^{+}(w)_{\mathbb{Z}}$. Here, the action of $v$ in the tensor product is given by $1$. We call the algebra $\mathcal{A}(w)_1$ the classical limit of $U_v^{+}(w)$ or the specialization of $U_v^{+}(w)$ at $v=1$. Furthermore, the $\mathbb{Z}[v^{\pm \frac12}]$-algebra $\mathcal{A}_v(w)=\bigoplus_{{\textbf a} \in \mathbb{N}^{2n}}\mathbb{Z}[v^{\pm \frac12 }]E[{\textbf a}]^*$ is the integral form of the algebra $\mathbb{Q}[v^{\pm \frac12}] \otimes_{\mathbb{Z}[v,v^{-1}]} \mathcal{A}(w)$ and will be useful in further considerations.
\end{definition}

\begin{remark}
Note that, by the form of the straightening relations for the dual variables from above, $\mathcal{A}(w)_1$ is a commutative algebra.
\end{remark}

\begin{definition}
Define a function $b \colon \mathbb{N}^{2n} \to \mathbb{Z}$ by $b(\textbf{a})=\sum_{k=1}^{2n}\genfrac(){0cm}{1}{a_k}{2}$ for a sequence $\textbf{a}=(a_1,a_2,\ldots,a_{2n}) \in \mathbb{N}^{2n}$.
\end{definition}

\begin{proposition}
\label{DualPBWFormula}
For every sequence $\textbf{a}=(a_1,a_2,\ldots,a_{2n}) \in \mathbb{N}^{2n}$ the equation $E[\textbf{a}]^*=v^{-b(\textbf{a})}E^*(\beta_1)^{a_1}E^*(\beta_2)^{a_2}\cdots E^*(\beta_{2n})^{a_{2n}}$ is true.
\end{proposition}

\begin{proof}
The proposition follows from Lusztig's evaluation \cite[Proposition 38.2.3]{Lusztig:QuantumGroups} for of the bilinear form at Poincar\'{e}-Birkhoff-Witt basis elements together with Leclerc's conversion formula \cite[Section 2.2]{Leclerc:Shuffles}, see Leclerc \cite[Section 5.5.3]{Leclerc:Shuffles}.
\end{proof}

\subsection{The dual canonical basis}
In this section we present the {\rm dual canonical basis} of $U_v^+(w)$. It is the dual of Lusztig's canonical basis from Lusztig \cite[Theorem 14.2.3]{Lusztig:QuantumGroups}. We need some auxiliary notations. The following definitions are special cases of Leclerc's definitions \cite[Section 2.7]{Leclerc:Shuffles} where $\mathfrak{g}$ is a general complex semisimple Lie algebra.

\begin{definition}
\label{Norm} For $\textbf{a}=(a_1,a_2,\ldots,a_{2n}) \in \mathbb{N}^{2n}$ write ${\rm deg}(E[\textbf{a}]^*)=\sum_{k=1}^{2n}a_k\beta_k \in Q^+$ as a $\mathbb{N}$-linear combination in the simple roots, i.e.,  ${\rm deg}(E[\textbf{a}]^*)=\sum_{k=1}^{2n}a_k\beta_k=\sum_{i=1}^nd_i\alpha_i$ with $d_i \in \mathbb{Z}$. Put $$N(\textbf{a})=\frac{1}{2}\Big( {\rm deg}(E[\textbf{a}]^*),{\rm deg}(E[\textbf{a}]^*)\Big) -\sum_{i=1}^nd_i.$$ We call $N$ the {\it norm} of the sequence $\textbf{a} \in \mathbb{N}^{2n}$. We also use the convention $$N(\sum_{i=1}^nd_i\alpha_i)=\frac{1}{2}\Big(\sum_{i=1}^nd_i\alpha_i,\sum_{i=1}^nd_i\alpha_i\Big)-\sum_{i=1}^nd_i$$ for elements in the root lattice.
\end{definition}

\begin{proposition}
For every natural number $k$ with $1 \leq k \leq 2n$ we have $\sigma(E^*(\beta_k))=v^{N(\beta_k)}E^*(\beta_k)$. 
\end{proposition}

\begin{proof}
Leclerc \cite[Lemma 7]{Leclerc:Shuffles} proves that a homogeneous element $f \in \mathcal{F}$ satisfies $\sigma(f)=v^{N(\vert f \vert)}f$ if and only if all coefficients in the expansion of $f$ in the basis of shuffles are invariant under $\sigma$. By Remark \ref{ShuffleExpForDuals} all coefficients are $0$ or $1$ in the case of $E^*(\beta_k)$.
\end{proof}

\begin{definition}
We define a partial order $\lhd$ on the parametrizing set $\mathbb{N}^{2n}$ of dual Poincar\'{e}-Birkhoff-Witt basis elements. For every $k$ with $1 \leq k \leq 2n$ let $\textbf{e}_k\in \mathbb{N}^{2n}$ be the vector satisfying $(\textbf{e}_k)_l=\delta_{k,l}$ for all $l$. For every integer $i$ with $1 \leq i \leq n$ let $k_i,l_i,m_i,n_i \in \left\lbrace 1,2,\ldots,2n\right\rbrace$ be the indices for which $\vert y_i\vert=\beta_{k_i}$, $\vert z_{i-1}\vert=\beta_{l_i}$, $\vert z_{i+1}\vert=\beta_{m_i}$, $\vert z_i\vert=\beta_{n_i}$. (Note that $k_1$ and $n_n$ are not defined. We put $\textbf{e}_{k_1}=\textbf{e}_{n_n}=0$.) Put $\textbf{v}_i=\textbf{e}_{k_i}-\textbf{e}_{l_i}-\textbf{e}_{m_i}+\textbf{e}_{n_i}$. Now we say that $\textbf{a},\textbf{b}\in\mathbb{N}^{2n}$ satisfy $\textbf{a}\lhd\textbf{b}$ if and only if $\textbf{b}-\textbf{a}\in\bigoplus_{i=1}^n\mathbb{N}\textbf{v}_i$.
\end{definition}

\begin{remark}
The straightening relations of Lemma \ref{StraighteningRelations} imply the following fact: If $\textbf{a} \in \mathbb{N}^{2n}$ and we expand $E^*(\beta_{2n})^{a_{2n}}E^*(\beta_{2n-1})^{a_{2n-1}}\cdots E^*(\beta_1)^{a_1}$ in the dual Poincar\'{e}-Birkhoff-Witt basis, then we get a $\mathbb{Q}(v)$-linear combination of $E[\textbf{b}]^*$ with $\textbf{b}\in S(\textbf{a})\cup \left\lbrace \textbf{a}\right\rbrace$. 
\end{remark}

\begin{definition}
For $\textbf{a}\in\mathbb{N}^{2n}$ put $S(\textbf{a})=\left\lbrace \textbf{b}\in\mathbb{N}^{2n} \colon \textbf{a}\lhd \textbf{b},\textbf{a}\neq\textbf{b}\right\rbrace$. 
\end{definition}

\begin{theorem}
\label{ExistenceOfDualCan}
There exist elements $B[\textbf{a}]^* \in U_v^+(w)$ parametrized by sequences $\textbf{a} \in \mathbb{N}^{2n}$ such that the set $\mathcal{B}^* = \left\lbrace B[\textbf{a}]^* \colon \textbf{a} \in \mathbb{N}^{2n} \right\rbrace$ is a basis of $U_v^+(w)$ and the following two properties hold.
\begin{enumerate}
\item[(1)] For every $\textbf{a} \in \mathbb{N}^{2n}$ we have $B[\textbf{a}]^*-E[\textbf{a}]^* \in \bigoplus_{\textbf{b}\in S(\textbf{a})}v^{-1}\mathbb{Z}[v^{-1}]E[\textbf{b}]^*$.
\item[(2)] For every $\textbf{a} \in \mathbb{N}^{2n}$ we have $\sigma(B[\textbf{a}]^*)=v^{N(\textbf{a})}B[\textbf{a}]^*$.
\end{enumerate}
The elements $B[\textbf{a}]^* \in U_v^+(w)$ are uniquely determined by these two properties.
\end{theorem}

\begin{proof}
Note that if $\textbf{a} \lhd \textbf{b}$, then $\vert E[\textbf{a}]^* \vert =\sum_{k=1}^{2n}a_k\beta_k= \sum_{k=1}^{2n}b_k\beta_k= \vert E[\textbf{b}]^* \vert$ since the straightening relations are relations in a $R$-graded algebra. For $\gamma \in R$ in the root lattice, consider the (finite) set $S_{\gamma} \subset \mathbb{N}^{2n}$ of all $\textbf{a}=(a_i)_{1 \leq i \leq 2n}$ with $\sum_{k=1}^{2n}a_k\beta_k=\gamma$. We extend the partial order $\lhd$ on $S_{\gamma}$ to a total order $<$. Let $\textbf{a}_1<\textbf{a}_2<\ldots<\textbf{a}_m$ be the elements of $S_{\gamma}$ written in increasing order. Now we prove by backward induction that for every $k=m,m-1,\ldots,2,1$ there exist linearly indpendent $B[\textbf{a}_k]$, $B[\textbf{a}_{k+1}],\ldots, B[\textbf{a}_m]$ satisfying (1) and (2).

Put $B[\textbf{a}_m]=E[\textbf{a}_m]$. It clearly satisfies property (1). Let $\textbf{a}_m=(a_1,a_2,\ldots,a_{2n})$. Since there are no $\textbf{b}\in S_{\gamma}$ such that $\textbf{a}_m < \textbf{b}$, the dual Poincar\'e-Birkhoff-Witt element $E[\textbf{a}_m]^*$ cannot be straightened, i.e., all $E^*(\beta_k)$ for which $a_k\neq 0$ are $v$-commutative. Therefore, by Remark \ref{HighestExponent} we have
\begin{align*}
\sigma(E[\textbf{a}_m]^*)&=\sigma(v^{-b(\textbf{a})}E^*(\beta_1)^{a_1}E^*(\beta_2)^{a_2}\cdots E^*(\beta_{2n})^{a_{2n}})\\
&=v^{b(\textbf{a})}v^{\sum_{k=1}^{2n}a_kN(\beta_k)}E^*(\beta_{2n})^{a_{2n}}\cdots E^*(\beta_2)^{a_2}E^*(\beta_1)^{a_1}\\
&=v^{b(\textbf{a})}v^{\sum_{k=1}^{2n}\frac12 a_k(\beta_k,\beta_k)-\sum_{k=1}^{2n}a_k\| \beta _k\|}\cdot v^{\sum_{k<l}(\beta_k,\beta_l)}E^*(\beta_1)^{a_1}E^*(\beta_2)^{a_2}\cdots E^*(\beta_{2n})^{a_{2n}}\\
&=v^{\frac12 (\sum_{k=1}^{2n} \beta_k,\sum_{k=1}^{2n} \beta_k)-\sum_{k=1}^{2n}a_k\| \beta _k\|}E[\textbf{a}_m]^*\\
&=v^{N(\textbf{a}_m)}E[\textbf{a}_m]^*.
\end{align*}
Here $\| \beta_k\|$ denotes the sum of the coefficients of $\beta_k$ when expanded as a $\mathbb{Z}$-linear combination of simple roots as in Definition \ref{Norm}. Hence, the variable $E[\textbf{a}_m]$ also satisfies property (2).

Now let $1 \leq k <m$ and assume that properties (1) and (2) hold for $\textbf{a}_{k+1}$, $\textbf{a}_{k+2},\ldots,\textbf{a}_m$. We expand $\sigma(E[\textbf{a}_k]^*)$ in the dual Poincar\'{e}-Birkhoff-Witt basis. Note that by the same argument as above (and ignoring terms of lower order) we see that the coefficient of the leading term $E[\textbf{a}_k]^*$ is $v^{N(\textbf{a}_k)}$. Thus, we have
\begin{align*}
\sigma(E[\textbf{a}_k]^*)=v^{N(\textbf{a}_k)}E[\textbf{a}_k]^*+\sum_{k<l\leq m} f_l E[\textbf{a}_l]^*
\end{align*}
for some $f_l \in \mathbb{Z}[v,v^{-1}]$. By induction hypothesis every $B[\textbf{a}_l]^*$ with $l>k$ is a $\mathbb{Z}[v,v^{-1}]$-linear combination of $E[\textbf{a}_l']^*$ with $l'>l$. By solving an upper triangular linear system of equations we see that every $E[\textbf{a}_l]^*$ with $l>k$ is a $\mathbb{Z}[v,v^{-1}]$-linear combination of $B[\textbf{a}_l']^*$ with $l'>l$. Hence, we may write
\begin{align*}
\sigma(E[\textbf{a}_k]^*)=v^{N(\textbf{a}_k)}E[\textbf{a}_k]^*+\sum_{k<l\leq m} g_l B[\textbf{a}_l]^*
\end{align*}
for some $g_l \in \mathbb{Z}[v,v^{-1}]$. We apply the antiinvolution $\sigma$ to the last equation:
\begin{align*}
E[\textbf{a}_k]^*=v^{-N(\textbf{a}_k)}\sigma(E[\textbf{a}_k]^*)+\sum_{k<l\leq m} v^{N(\textbf{a}_l)}\sigma(g_l) B[\textbf{a}_l]^*.
\end{align*}
Note that $N(\textbf{a}_k)=N(\textbf{a}_l)$ for all $l$. Comparing coefficients yields $v^{2N(\textbf{a}_l)}\sigma(g_l)=-g_l$. It follows that $\sigma (v^{-N(\textbf{a}_l)}g_l)=-v^{-N(\textbf{a}_l)}g_l$. Thus, we may write $v^{-N(\textbf{a}_l)}g_l=h_l-\sigma(h_l)$ for some $h_l \in v^{-1}\mathbb{Z}[v^{-1}]$. Now put
\begin{align*}
B[\textbf{a}_k]^*=E[\textbf{a}_k]^*+\sum_{k<l\leq m} h_l B[\textbf{a}_l]^*.
\end{align*}
It is easy to see that properties (1) and (2) are true for $B[\textbf{a}_k]^*$ and that $B[\textbf{a}_k]^*, B[\textbf{a}_{k+1}]^*,$ $\ldots,B[\textbf{a}_m]^*$ are linearly independent.

For the uniqueness, suppose that $k$ is some index such that there are variables $B[\textbf{a}_k]_1^*$ and $B[\textbf{a}_k]_2^*$ fulfilling the two properties of the theorem. Then their difference $B[\textbf{a}_k]_1^*-B[\textbf{a}_k]_2^*\in\bigoplus_{l>k}v^{-1}\mathbb{Z}[v^{-1}]B[\textbf{a}_l]^*$. Application of $\sigma$ and multiplication with $v^{-N(\textbf{a}_k)}$ afterwards yields $B[\textbf{a}_k]_1^*-B[\textbf{a}_k]_2^*\in\bigoplus_{l>k}v\mathbb{Z}[v]B[\textbf{a}_l]^*$, so $B[\textbf{a}_k]_1^*=B[\textbf{a}_k]_2^*$.
\end{proof}

\begin{remark}
It is known that the dual of Lusztig's canonical basis under Kashiwara's bilinear form obeys the two properties of Theorem \ref{ExistenceOfDualCan}, compare Leclerc \cite[Proposition 39]{Leclerc:Shuffles}. By uniqueness, the set $\mathcal{B}^*=\left\lbrace B[\textbf{a}]^* \colon \textbf{a} \in \mathbb{N}^{2n} \right\rbrace$ is the dual of Lusztig's canonical basis, or the {\it dual canonical basis} for short. 
\end{remark}

We call $E[\textbf{a}]^*$ (for $\textbf{a} \in \mathbb{N}^{2n}$) the {\it leading term} in the expansion of $B[\textbf{a}]^*$ in the Poincar\'{e}-Birkhoff-Witt basis. In what follows we use the convention $z_0=z_{n+1}=1$. Prominent elements in $\mathcal{B}^*$ are 
\begin{align*}
&p_i=y_iz_i-v^{-1}z_{i-1}z_{i+1}, &&{\rm for} \ i \ {\rm odd \ with} \ 1 \leq i \leq n,\\
&p_i=z_iy_i-v^{-1}z_{i-1}z_{i+1}, &&{\rm for} \ i \ {\rm even \ with} \ 2 \leq i \leq n-1.
\end{align*}
The first property of Theorem \ref{ExistenceOfDualCan} is obvious and the second follows easily from a calculation using the straightening relations. The variables are $v$-deformations of the $\delta$-functions of the $\mathcal{C}_M$-projective rigid $\Lambda$-modules from Section \ref{PreprojAlgRigidMod}. The non-deformed $\delta$-function associated with these modules are frozen cluster variables in Gei\ss -Leclerc-Schr\"{o}er's cluster algebra \ACM, compare Section \ref{SubSection:ClusterAlgebra}. 

\begin{lemma}
\label{PCommutation}
For every pair $(i,k)$ of natural numbers such that $1 \leq i \leq n$ and $1 \leq k \leq 2n$ the elements $p_i$ and $E^*(\beta_k)\in U_v^{+}(w)$ are $v$-commutative, i.e., there is an integer $a$ such that $p_iE^*(\beta_k)=v^aE^*(\beta_k)p_i$.
\end{lemma}

\begin{proof}
First of all, we claim that $(\vert z_i\vert,\vert y_i\vert)$ is equal to $-1$ if $i\in\{1,n\}$ and equal to $0$ if $2\leq i\leq n-1$. For a proof, note that $(\vert z_1\vert,\vert y_1\vert)=(\alpha_2+\alpha_3,\alpha_1)=-1$, and similarly $(\vert z_n\vert,\vert y_n\vert)=-1$. In the same way we see that $(\vert z_2\vert,\vert y_2\vert)=(\alpha_2+\alpha_3+\alpha_4+\alpha_5,\alpha_1+\alpha_2+\alpha_3)=0$, and similarly $(\vert z_{n-1}\vert,\vert y_{n-1}\vert)=0$. Furthermore, we have $(\alpha_i+\alpha_{i+1}+\ldots+\alpha_j,\alpha_k)=0$ for all $i+1\leq k\leq j-1$ from which it easily follows that $(\vert z_i\vert,\vert y_i\vert)=0$ for $i\in\{3,4,\ldots,n-2\}$.

The claim implies that for odd $i$ in the straightening relation for $z_iy_i$ the coefficient in front of $y_iz_i$ is $v^{-(\vert z_i\vert,\vert y_i\vert)}$. The same is true for the coefficent in front of $z_iy_i$ in the straightening relation for $y_iz_i$. Let $i$ be an odd integer such that $1 \leq i \leq n$. Note that $$p_i=y_iz_i-v^{-1}z_{i-1}z_{i+1}=v^{-(\vert z_i\vert,\vert y_i\vert)}z_iy_i-vz_{i-1}z_{i+1}$$ by property (2) of Theorem \ref{ExistenceOfDualCan}. From the straightening relations of Lemma \ref{StraighteningRelations} it is clear that $p_i$ commutes with every $y_j$ with $\vert j-i \vert \geq 4$ and with every $z_j$ with $\vert j-i \vert \geq 3$. If $i \geq 5$, then $y_{i-3}p_i=y_{i-3}y_iz_i-v^{-1}y_{i-3}z_{i-1}z_{i+1}=vy_iz_iy_{i-3}-z_{i-1}z_{i+1}y_{i-3}=vp_iy_{i-3}$. Now assume that $i \geq 3$. We have $y_{i-2}p_i=y_{i-2}y_iz_i-v^{-1}y_{i-2}z_{i-1}z_{i+1}=v^{-1}y_iz_iy_{i-2}-v^{-2}z_{i-1}z_{i+1}y_{i-2}=v^{-1}p_iy_{i-2}$. Furthermore, the we see that $z_{i-2}p_i=z_{i-2}y_iz_i-v^{-1}z_{i-2}z_{i-1}z_{i+1}=vy_iz_iz_{i-2}-z_{i-1}z_{i+1}z_{i-2}=vp_iz_{i-2}$. The calculation 
\begin{align*}
y_{i-1}p_i&=y_{i-1}y_iz_i-v^{-1}y_{i-1}z_{i-1}z_{i+1}\\
&=v^{(\vert y_{i-1}\vert,\vert y_i\vert)}\Big(y_iy_{i-1}+(v-v^{-1})z_{i+1}z_{i-2}\Big)z_i\\&\quad-v^{-1}\Big(z_{i-1}y_{i-1}+(v-v^{-1})z_iz_{i-2}\Big)z_{i+1}\\
&=v^{(\vert y_{i-1}\vert,\vert y_i\vert)}(vy_iz_iy_{i-1}-z_{i-1}z_{i+1}y_{i-1})=v^{1+(\vert y_{i-1}\vert,\vert y_i\vert)}p_iy_{i-1}
\end{align*}
shows that $p_i$ also $v$-commutes with $y_{i-1}$. It also $v$-commutes with $z_{i-1}$ as the calculation shows: $z_{i-1}p_i=z_{i-1}y_iz_i-v^{-1}z^2_{i-1}z_{i+1}=y_iz_iz_{i-1}-v^{-1}z_{i-1}z_{i+1}z_{i-1}=p_iz_{i-1}$. Now assume that $i \geq 1$.The following equation is true:
\begin{align*}
y_ip_i&=y^2_iz_i-v^{-1}y_iz_{i-1}z_{i+1}\\
&=y_i\Big(v^{-(\vert y_i\vert,\vert z_i\vert)}z_iy_i+(v^{-1}-v)z_{i-1}z_{i+1}\Big)-v^{-1}y_iz_{i-1}z_{i+1}\\
&=v^{-(\vert y_i\vert,\vert z_i\vert)}y_iz_iz_i-vy_iz_{i-1}z_{i+1}\\
&=v^{-(\vert y_i\vert,\vert z_i\vert)}(y_iz_i-v^{-1}z_{i-1}z_{i+1})y_i=v^{-(\vert y_i\vert,\vert z_i\vert)}p_iy_i.
\end{align*}
Finally we see that $z_ip_i=z_iy_iz_i-v^{-1}z_iz_{i-1}z_{i+1}=v^{(\vert y_i\vert,\vert z_i\vert)}(v^{-(\vert y_i\vert,\vert z_i\vert)}z_iy_i-vz_{i-1}z_{i+1})z_i=v^{(\vert y_i\vert,\vert z_i\vert)}p_iz_i$. The $v$-commutativity relations of $p_i$ with elements with index $j>i$ are proved in the same way.

The case of even $i$ can be handled with similar arguments.
\end{proof}

\begin{remark}
\begin{enumerate}
\item The $v$-commutativity relations of Lemma \ref{PCommutation} will be crucial for the construction of the initial quantum seed of $\mathcal{A}(w)_v$. Another verification of $v$-commutativity relations for the initial seed is due to Kimura \cite[Section 6]{Kimura}. 
\item Multiplicative properties of (dual) canonical basis elements have also been studied by Reineke \cite{Reineke:Multi}.
\item As observed by Leclerc, the $v$-deformations of the $\delta$-functions of $\mathcal{C}_M$-projective rigid $\Lambda$-modules also $v$-commute with the generators of $U^+_v(w)$ in the Kronecker cases for $w$ of length $4$, see the author \cite[Section 4.1]{Lampe}.
\item A consideration of the leading terms of the two occuring variables in Lemma \ref{PCommutation} is sufficient to determine the integer $a$. 
\end{enumerate}
\end{remark}

\begin{remark}
The techniques in this section work with the same proofs for the case $U_v^+(w')$ as well. We use a similar notation $y'_i,z_i'\in U_v^+(w')$ for the elements dual to $u'_i,v'_i,w'_i,x_i'\in U_v^+(w)$ (for appropriate indices $i$). The straightening relations for the dual variables can be computed using the same methods.
\end{remark}

\section{The quantum cluster algebra structure induced by the dual canonical basis}

In this section we are going to prove that the integral form $\bigoplus_{\textbf{a}\in\mathbb{N}^{2n}}\mathbb{Z}[v^{\pm \frac12}]E[\textbf{a}]^*$ is a quantum cluster algebra in the sense of Bereinstein-Zelevinsky \cite{BZ:QuantumClusterAlg}. The corresponding non-quantized cluster algebra is Gei\ss -Leclerc-Schr\"oer's \cite{GLS:Unipotent} cluster algebra $\mathcal{A}(w)$. 

Natural Quantum cluster algebra structures have only been observed in very few cases, see for example Grabowski-Launois \cite{GrabowskiLaunois}, Rupel \cite{Rupel:QuantumCCFormula}, and the author \cite{Lampe}. For a study of bases of quantum cluster algebras of type $\tilde{A}_1^{(1)}$ see Ding-Xu \cite{DingXu}.

\begin{definition}
\label{DefDelta}
For $1 \leq i \leq j \leq n$ define $\Delta_{i,j}^v \in \mathcal{B}^*$ to be the dual canonical basis element with leading term $\prod_{i \leq r \leq j,r \ {\rm odd}}y_r\prod_{i \leq r \leq j, \ {r \rm even}}y_r$.
\end{definition}

\begin{remark}
\label{ExamplesOfDualCan}
We provide some examples of elements in the dual canonical basis $\mathcal{B}^*$ of the form $\Delta_{i,j}^v$ with $1 \leq i \leq j \leq n$. We focus on examples where the interval $[i,j]$ is small, i.e., $j \leq i+2$. First of all, we clearly have $\Delta_{i,i}^v=y_i$ for all $i$. Furthermore, an elementary calculation using the straightening relations shows that: $\Delta_{1,2}^v=y_1y_2-v^{-1}z_3=vy_2y_1-vz_3$, $\Delta_{n-1,n}^v=y_ny_{n-1}-v^{-1}z_{n-2}=vy_{n-1}y_n-vz_{n-2}$, and that for $2 \leq i \leq n-2$ \begin{align*}
\Delta_{i,i+1}^v=
\begin{cases}
y_iy_{i+1}-v^{-1}z_{i-1}z_{i+2}=y_{i+1}y_i-vz_{i+2}z_{i-1},&{\rm if} \ i \ {\rm is \ odd};\\
y_{i+1}y_i-v^{-1}z_{i+2}z_{i-1}=y_iy_{i+1}-vz_{i-1}z_{i+2},&{\rm if} \ i \ {\rm is \ even}.\\
\end{cases}
\end{align*}
Recall the convention $z_0=z_{n+1}=1$. The formulae simplify to $\Delta_{i,i+1}^v=y_iy_{i+1}-v^{-1}z_{i-1}z_{i+2}$ for odd $i$ and $\Delta_{i,i+1}^v=y_{i+1}y_i-v^{-1}z_{i+2}z_{i-1}$ for even $i$. With the same convention we can compute:
\begin{align*}
\Delta_{1,3}&=y_1y_3y_2-v^{-1}y_1z_4z_1-v^{-1}y_3z_3+v^{-2}z_2z_4\\
&=vy_2y_1y_3-vz_3y_3-v^2z_1z_4y_1+v^2z_2z_4,\\
\Delta_{n-2,n}&=y_ny_{n-2}y_{n-1}-v^{-1}y_nz_{n-3}z_n-v^{-1}y_{n-2}z_{n-2}+v^{-2}z_{n-1}z_{n-3}\\
&=vy_{n-1}y_ny_{n-2}-vz_{n-2}y_{n-2}-v^2z_nz_{n-3}y_n+v^2z_{n-1}z_{n-3}.
\end{align*}
For odd $i$ such that $3 \leq i \leq n-4$ we have:
\begin{align*}
\Delta_{i,i+2}&=y_iy_{i+2}y_{i+1}-v^{-1}y_iz_{i+3}z_i-v^{-1}y_{i+2}z_{i-1}z_{i+2}+v^{-2}z_{i-1}z_{i+1}z_{i+3}\\
&=y_{i+1}y_{i+2}y_i-vz_iz_{i+3}y_i-vz_{i+2}z_{i-1}y_{i+2}+v^2z_{i-1}z_{i+1}z_{i+3}.
\end{align*}
For odd $i$ such that $3 \leq i \leq n-4$ we have:
\begin{align*}
\Delta_{i,i+2}&=y_{i+1}y_iy_{i+2}-v^{-1}z_iz_{i+3}y_i-v^{-1}z_{i+2}z_{i-1}y_{i+2}+v^2z_{i-1}z_{i+1}z_{i+3}\\
&=y_iy_{i+2}y_{i+1}-vy_iz_{i+3}z_i-vy_{i+2}z_{i+2}z_{i-1}+v^2z_{i-1}z_{i+1}z_{i+3}.
\end{align*}
\end{remark}

\begin{remark}
In what follows we prove recursions for the $\Delta_{i,j}^v$ with $1 \leq i \leq j \leq n$. The formulae turn out to be quantized versions of the formulae for the $\Delta_{i,j}$ from Section \ref{DescriptionOfClusterVariables}. The quantized recursions will be crucial for the verification of the quantum cluster algebra struture on $U_v^{+}(w)$ in the next section. 
\end{remark}

To formulate the recursion effectively the following definitions are helpful. First of all we introduce the convention $\Delta_{i,i-1}^v=1$ for all $i$. Furthermore, we give the following two definitions.

\begin{definition}
For $1 \leq i,j \leq n$ put 
\begin{align*}
s_{i,j}=\sum_{i \leq s \leq j}\vert y_s \vert; \quad
o_{i,j}=\sum_{\genfrac{}{}{0pt}{}{i \leq s \leq j}{s {\rm \ odd}}}\vert y_s \vert; \quad
e_{i,j}=\sum_{\genfrac{}{}{0pt}{}{i \leq s \leq j}{s {\rm \ even}}}\vert y_s \vert.
\end{align*} 
\end{definition}

\begin{definition}
For all $i,j$ with $1 \leq i,j \leq n$ and $j-i \geq 2$ define
\begin{align*}
A_{i,j}=\begin{cases}
-1, & {\rm if \ }j-i \leq 3;\\
-2, & {\rm if \ }j-i \geq 4.
\end{cases}
\end{align*}
\end{definition}

\begin{lemma}
\label{FirstLemma}
For all pairs $(i,j)$ of integers such that $1 \leq i \leq j \leq n$ the following equation is true:
\begin{align*}
N(s_{i,j})-N(s_{i,j-1})-N(\vert y_j\vert)=
\begin{cases}(\vert y_j\vert,o_{i,j-1}), & {\rm if \ } j {\rm \ is \ even};\\
(\vert y_j\vert,e_{i,j-1}), & {\rm if \ } j {\rm \ is \ odd}.
\end{cases}
\end{align*}
\end{lemma}

\begin{proof}
The proposition follows easily from the observation $\frac{1}{2}(s_{i,j},s_{i,j})=\frac{1}{2}(e_{i,j}+o_{i,j},e_{i,j}+o_{i,j})=(j-i+1)+(e_{i,j},o_{i,j})$.
\end{proof}

\begin{lemma}
\label{Lemma2}
Let $i,j$ be integers such that $1 \leq i \leq j \leq n$  and $j \geq i+2$. The following equation holds:
\begin{align*}
N(s_{i,j})-N(s_{i,j-3})-&N(\vert p_{j-2}\vert)-N(\vert z_{j+1}\vert)=\\
&\begin{cases}(\vert y_j\vert + \vert y_{j-2}\vert,o_{i,j-3})+(\vert y_{j-1}\vert,e_{i,j-4}), & {\rm if \ } j {\rm \ is \ even};\\
(\vert y_j\vert + \vert y_{j-2}\vert,e_{i,j-3})+(\vert y_{j-1}\vert,o_{i,j-4}), & {\rm if \ } j {\rm \ is \ odd}.
\end{cases}
\end{align*}
\end{lemma}

\begin{proof}
If $j$ is odd, then the elements $y_j,y_{j-1},z_{j+1},z_{j-2}$ have degrees $\beta_j,\beta_{n+j_1},\beta_{j+1}$ and $\beta_{n+j-2}$, respectively. If $j$ is odd, then the elements $y_j,y_{j-1},z_{j+1},z_{j-2}$ have degrees $\beta_{n+j},\beta_{j_1},\beta_{n+j+1}$ and $\beta_{j-2}$, respectively. It is easy to see in Figure \ref{fig:AR} that in all cases the equation $\vert y_j\vert+\vert y_{j-1}\vert=\vert z_{j+1}\vert+\vert z_{j-2}\vert$ holds. With this fact the proposition follows just as above from the observation $\frac{1}{2}(s_{i,j},s_{i,j})=\frac{1}{2}(e_{i,j}+o_{i,j},e_{i,j}+o_{i,j})=(j-i+1)+(e_{i,j},o_{i,j})$.
\end{proof}

\begin{remark}
\label{PBWExpansionOfDelta}
By definition, the leading term of the variable $\Delta^v_{i,j}$ (where $1 \leq i \leq j \leq n$) is $\prod_{i \leq r \leq j,r \ {\rm odd}}y_r\prod_{i \leq r \leq j, \ {r \rm even}}y_r$. Therefore, $\Delta^v_{i,j}$ is a $\mathbb{Z}[v,v^{-1}]$-linear combination of terms of the form 
\begin{align*}
\prod_{\genfrac{}{}{0pt}{}{i \leq r \leq j}{r \ {\rm odd}}}y_r^{1-a_r}\prod_{\genfrac{}{}{0pt}{}{i-1 \leq r \leq j+1}{r \ {\rm even}}}z_r^{a_{r+1}+a_{r-1}-a_r}\prod_{\genfrac{}{}{0pt}{}{i-1 \leq r \leq j+1}{r \ {\rm odd}}}z_r^{a_{r+1}+a_{r-1}-a_r} \prod_{\genfrac{}{}{0pt}{}{i \leq r \leq j}{r \ {\rm even}}}y_r^{1-a_r}
\end{align*}
for some $\textbf{a}=(a_1,a_2,\ldots,a_n)\in \left\lbrace0,1\right\rbrace^n$ such that $a_1=a_2=\ldots=a_{i-1}=0$, $a_{j+1}=a_{j+1}=\ldots=a_n=0$, and $a_{r+1}+a_{r-1}-a_r\geq0$ for all $r$. We denote this term by $\Delta^v_{i,j}[\textbf{a}]$. 
\end{remark}

We know that $p_{j+1}$ commutes, up to a power of $v$, with every $\Delta^v_{i,j}[\textbf{a}]$. The following proposition shows that much more is true: The commutation exponent only depends on the pair $(i,j)$, but not on $\textbf{a}\in\left\lbrace0,1\right\rbrace^n$.

\begin{proposition}
Let $i,j$ be integers such that $1\leq i \leq j \leq n-1$. The variables $\Delta^v_{i,j}[\textbf{a}]$ and $p_{j+1}$ are $v$-commutative. More precisely: If $j$ is even, then $$\Delta^v_{i,j}[\textbf{a}]p_{j+1}=v^{(\vert y_{j+1}\vert,e_{i,j})+(\vert z_{j+1}\vert,e_{i,j})-(\vert z_{j+1}\vert,o_{i,j-1})}p_{j+1}\Delta^v_{i,j}[\textbf{a}],$$ and if $j$ is odd, then $$\Delta^v_{i,j}[\textbf{a}]p_{j+1}=v^{-(\vert y_{j+1}\vert,o_{i,j})+(\vert z_{j+1}\vert,e_{i,j})-(\vert z_{j+1}\vert,o_{i,j-1})}p_{j+1}\Delta^v_{i,j}[\textbf{a}].$$ 
\end{proposition}

\begin{proof}
Assume that $j$ be even. By Lemma \ref{PCommutation} we know that $p_{j+1}$ commutes, up to a power of $v$, with every $\Delta^v_{i,j}[\textbf{a}]$. To determine the appropriate power of $v$, we compare the leading terms. The leading term of $p_{j+1}$ is $y_{j+1}z_{j+1}$. By Remark \ref{HighestExponent} we see that
\begin{align*}
\Delta^v_{i,j}[\textbf{a}]p_{j+1}=&v^{a_j(\vert y_{j+1}\vert,\vert z_{j+1}\vert+\vert z_{j-1}\vert-\vert y_j\vert-\vert-z_j\vert)}\cdot v^{a_{j-1}(\vert y_{j+1}\vert,\vert z_j\vert+\vert z_{j-2}\vert-\vert z_{j-1}\vert)} \\
&\cdot v^{a_{j-2}(\vert y_{j+1}\vert,\vert z_{j-1}\vert-\vert z_{j-2}\vert)}\cdot v^{a_j(\vert z_{j+1}\vert,\vert y_j\vert-\vert z_j\vert)}\cdot v^{a_{j-1}(\vert z_{j+1}\vert,\vert y_{j-1}\vert-\vert z_j\vert+\vert z_{j-2}\vert)} \\
&\cdot v^{(\vert y_{j+1}\vert,e_{i,j})+(\vert z_{j+1}\vert,e_{i,j})-(\vert z_{j+1}\vert,o_{i,j-1})}p_{j+1} \Delta^v_{i,j}[\textbf{a}].
\end{align*}
Notice that first of all $(\vert y_k\vert,\vert y_{k-2}\vert)=(\vert y_k\vert,\vert y_{k-4}\vert)=0$ for all $k$, that secondly $(\vert y_k\vert,\vert y_l\vert)=0$ for $\vert k-l\vert \geq 4$, that thirdly $(\vert y_k\vert,\vert z_l\vert)=0$ for $\vert k-l\vert \geq 3$, and that finally $y_j\vert+\vert z_j\vert=\vert z_{j+1}\vert+\vert z_{j-1}\vert$ for all $k$. Furthermore, note that $(\vert z_{k+1}\vert,\vert y_k\vert)=(\vert z_{k+1},\vert z_k\vert)$ for all $k$. Hence, we have $$\Delta^v_{i,j}[\textbf{a}]p_{j+1}=v^{(\vert y_{j+1}\vert,e_{i,j})+(\vert z_{j+1}\vert,e_{i,j})-(\vert z_{j+1}\vert,o_{i,j-1})}p_{j+1}\Delta^v_{i,j}[\textbf{a}].$$ 

By a similar argument we can show that for odd numbers $j$ the equation $$\Delta^v_{i,j}[\textbf{a}]p_{j+1}=v^{-(\vert y_{j+1}\vert,o_{i,j})+(\vert z_{j+1}\vert,e_{i,j})-(\vert z_{j+1}\vert,o_{i,j-1})}p_{j+1}\Delta^v_{i,j}[\textbf{a}].$$
\end{proof}

The independence of the commutation exponent on $\textbf{a}\in\left\lbrace0,1\right\rbrace^n$ implies the corollary.

\begin{corollary}
Let $i,j$ be integers as above. Then the variables $\Delta^v_{i,j}$ and $p_{j+1}$ are $v$-commutative. More precisely: If $j$ is even, then $$\Delta^v_{i,j}p_{j+1}=v^{(\vert y_{j+1}\vert,e_{i,j})+(\vert z_{j+1}\vert,e_{i,j})-(\vert z_{j+1}\vert,o_{i,j-1})}p_{j+1}\Delta^v_{i,j},$$ and if $j$ is odd, then $$\Delta^v_{i,j}p_{j+1}=v^{-(\vert y_{j+1}\vert,o_{i,j})+(\vert z_{j+1}\vert,e_{i,j})-(\vert z_{j+1}\vert,o_{i,j-1})}p_{j+1}\Delta^v_{i,j}.$$ 
\end{corollary}

Having established these conventions, definitions and proposistions, we are now able to formulate a theorem that provides a recursion for the $\Delta_{i,j}^v$ and a quantized version of the exchange relation. These are the parts (a) and (b) of Theorem \ref{QuantRecursionForClusterVar}. For a proof of the theorem, we proceed by induction. For a functioning induction step we include also the $v$-commutator relations in part (c) and (d).

\begin{theorem}
\label{QuantRecursionForClusterVar}
Let $i,j$ be integers such that $1 \leq i,j \leq n$ and $j-i \geq 2$.
\begin{enumerate}
\item[(a)] The dual canonical basis element $\Delta_{i,j}^v$ can be computed recursively from elements $\Delta_{i,j'}^v$ with $j'<j$. More precisely, if $j$ is even, then we have: 
\begin{align*}
\Delta^v_{i,j}&=\Delta^v_{i,j-1}y_j-v^{A_{i,j}}\Delta^v_{i,j-3}p_{j-2}z_{j+1}\\
&=v^{-(\vert y_j\vert,o_{i,j-1})}y_j\Delta^v_{i,j-1}\\&\quad-v^{-A_{i,j}-(\vert y_j\vert +\vert y_{j-2}\vert,o_{i,j-3})-(\vert y_{j-1}\vert,e_{i,j-4})}z_{j+1}p_{j-2}\Delta^v_{i,j-3}.
\end{align*}
If $j$ is odd, the we have:
\begin{align*}
\Delta^v_{i,j}&=y_j\Delta^v_{i,j-1}-v^{A_{i,j}}z_{j+1}p_{j-2}\Delta^v_{i,j-3}\\
&=v^{-(\vert y_j\vert,e_{i,j-1})}\Delta^v_{i,j-1}y_j\\&\quad-v^{-A_{i,j}-(\vert y_j\vert + \vert y_{j-2}\vert,e_{i,j-3})-(\vert y_{j-1}\vert,o_{i,j-4})}\Delta^v_{i,j-3}p_{j-2}z_{j+1}.
\end{align*}

\item[(b)] Furthermore, the following quantum cluster exchange relation holds:
\begin{align*}
\Delta^v_{i,j}z_j=\begin{cases}
\Delta^v_{i,j-1}p_j+v^{1-(\vert y_{j-1}\vert,e_{i,j-2})}\Delta^v_{i,j-2}p_{j-1}z_{j+1}, & {\rm if \ } j {\rm \ is \ even};\\
v^{-(\vert y_j\vert,e_{i,j-1})}\Delta^v_{i,j-1}p_j+v^{-1-(\vert y_j\vert,e_{i,j-1})}\Delta^v_{i,j-2}p_{j-1}z_{j+1}, & {\rm if \ } j {\rm \ is \ odd}.
\end{cases}
\end{align*}

\item[(c)] If $j+1 \leq n$, then the following $v$-commutator relation holds. If $j$ is even, then 
\begin{align*}
y_{j+1}\Delta^v_{i,j}=v^{-(\vert y_{j+1}\vert,e_{i,j})}\Delta_{i,j}y_{j+1}+v^{-1-(\vert y_{j-1}\vert,e_{i,j-2})}(v^{-1}-v)\Delta^v_{i,j-2}p_{j-1}z_{j+2}.
\end{align*}
If $j$ is odd, then
\begin{align*}
y_{j+1}\Delta^v_{i,j}=v\Delta^v_{i,j}y_{j+1}+v^{1-(\vert y_j\vert,e_{i,j-3})}(v-v^{-1})\Delta_{i,j-2}^vp_{j-1}z_{j+2}.
\end{align*}

\item[(d)] If $j+2 \leq n$, then $\Delta_{i,j}^v$ and $y_{j+2}$ are $v$-commutative. More precisely, if $j$ is even, then $\Delta_{i,j}^vy_{j+2}=v^{-1}y_{j+2}\Delta_{i,j}^v$, and if $j$ is odd, then $\Delta_{i,j}^vy_{j+2}=vy_{j+2}\Delta_{i,j}^v$.
\end{enumerate}

\end{theorem}

\begin{proof}
We proceed by induction on $j-i$. Using the explicit formulae provided by Remark \ref{ExamplesOfDualCan} it is easy to see that Theorem \ref{QuantRecursionForClusterVar} is true for $j-i=2$. Now let $j-i \geq 3$ and assume that Theorem \ref{QuantRecursionForClusterVar} is true for all smaller values of $j-i$. We distinguish two cases.

First of all assume that $j$ is even. It follows that $4 \leq j \leq n-1$. Put $D=\Delta^v_{i,j-1}y_j-v^{A_{i,j}}\Delta^v_{i,j-3}p_{j-2}z_{j+1}$. We have to prove that $D=\Delta^v_{i,j}$, i.e., we have to show that $D$ satisfies properties (1) and (2) of Theorem \ref{ExistenceOfDualCan}. First of all, we verify property (1). We expand the dual canonical basis elements $\Delta^v_{i,j-1}, \Delta^v_{i,j-3}$ according to Remark \ref{PBWExpansionOfDelta}. We see that 
\begin{align*}
\Delta^v_{i,j-1}=\sum_{\textbf{a}}f_{\textbf{a}}v^{\sum -\genfrac(){0cm}{1}{a_{r+1}+a_{r-1}-a_r}{2}}&\prod_{\genfrac{}{}{0pt}{}{i \leq r \leq j-3}{r \ {\rm odd}}}y_r^{1-a_r}\prod_{\genfrac{}{}{0pt}{}{i-1 \leq r \leq j}{r \ {\rm even}}}z_r^{a_{r+1}+a_{r-1}-a_r}\\
&\cdot \prod_{\genfrac{}{}{0pt}{}{i-1 \leq r \leq j}{r \ {\rm odd}}}z_r^{a_{r+1}+a_{r-1}-a_r} \prod_{\genfrac{}{}{0pt}{}{i \leq r \leq j-1}{r \ {\rm even}}}y_r^{1-a_r}
\end{align*}
where the sum is taken over all the admissible sequences $\textbf{a}=(a_1,a_2,\ldots,a_n)\in \left\lbrace0,1\right\rbrace^n$ as in Remark \ref{PBWExpansionOfDelta} and $f_{\textbf{a}} \in v^{-1}\mathbb{Z}[v^{-1}]$ except for $f_0=1$. Here, we have used that $-\genfrac(){0cm}{1}{1-a_r}{2}=0$ for all terms $a_r$ in such a sequence. It is clear that $\Delta^v_{i,j-1}y_j-\Delta^v_{i,j-1}[0]y_j\in\bigoplus_{\textbf{b}\in S(\textbf{a})}v^{-1}\mathbb{Z}[v^{-1}]E[\textbf{b}]^*$ and that $\Delta^v_{i,j-1}[0]y_j$ is the dual Poincar\'e-Birkhoff-Witt basis element from Definition \ref{DefDelta} that serves as leading term.

Furthermore, we have
\begin{align*}
\Delta^v_{i,j-3}=\sum_{\textbf{a}}g_{\textbf{a}}v^{\sum -\genfrac(){0cm}{1}{a_{r+1}+a_{r-1}-a_r}{2}}&\prod_{\genfrac{}{}{0pt}{}{i \leq r \leq j-3}{r \ {\rm odd}}}y_r^{1-a_r}\prod_{\genfrac{}{}{0pt}{}{i-1 \leq r \leq j-2}{r \ {\rm even}}}z_r^{a_{r+1}+a_{r-1}-a_r}\\
&\cdot \prod_{\genfrac{}{}{0pt}{}{i-1 \leq r \leq j-3}{r \ {\rm odd}}}z_r^{a_{r+1}+a_{r-1}-a_r} \prod_{\genfrac{}{}{0pt}{}{i \leq r \leq j-4}{r \ {\rm even}}}y_r^{1-a_r}
\end{align*}
where the sum is taken over all the admissible sequences $\textbf{a}=(a_1,a_2,\ldots,a_n)\in \left\lbrace0,1\right\rbrace^n$ as in Remark \ref{PBWExpansionOfDelta} and $g_{\textbf{a}} \in v^{-1}\mathbb{Z}[v^{-1}]$ except for $g_0=1$. Now we consider $v^{A_{i,j}}\Delta^v_{i,j-3}p_{j-2}z_{j+1}$. Note that the generator $z_{j+1}$ commutes with all occuring terms in this expansion. For $p_{j-2}$, by Lemma \ref{PCommutation} the $v$-commutativity relation 
\begin{align*}
\left( \prod_{\genfrac{}{}{0pt}{}{i \leq r \leq j-4}{r \ {\rm even}}}y_r^{1-a_r}\right) p_{j-2}=v^{(1-a_{j-4})(\vert z_{j-2}\vert, \vert y_{j-4}\vert)}p_{j-2}\left( \prod_{\genfrac{}{}{0pt}{}{i \leq r \leq j-4}{r \ {\rm even}}}y_r^{1-a_r}\right) 
\end{align*}
holds. Hence, in each summand we can transfer $p_{j-2}$ to the left of the product. We get a factor $v^{(1-a_{j-4})}$ since $(\vert z_{j-2}\vert, \vert y_{j-4}\vert)=1$ for all $j$. We concentrate on a single summand. Write $p_{j-2}=z_{j-2}y_{j-2}-v^{-1}z_{j-3}z_{j-1}$. This decomposition splits the sum into two parts. Consider the summand coming from $z_{j-2}y_{j-2}$. To write this term in the dual PBW basis we have to tranfer $z_{j-2}$ to the left of the product of the odd $z_r$. We get a factor $v^{a_{j-4}-a_{j-3}}$. Now all the monomials are in the right order. The generators $z_{j+1}$ and $y_{j-2}$ do not occur in the expansion of $\Delta^v_{i,j-3}$ in the dual PBW basis, but $z_{j-2}$ may. In the summands where $z_{j-2}$ occurs, we have increased the exponent from $1$ to $2$. So all coefficients in the dual PBW expansion of these summands have the form $$g_{\textbf{a}}v^{A_{i,j}}v^{1-a_{j-4}}v^{a_{j-4}-a_{j-3}}v^{\genfrac(){0cm}{1}{1+a_{j-3}}{2}}.$$ Note that $A_{i,j}\leq -1$ and that $\genfrac(){0cm}{1}{1+a_{j-3}}{2}-a_{j-3}=0$ for $a_{j-3} \in \left\lbrace 0,1\right\rbrace$. Hence, all coefficients are in $v^{-1}\mathbb{Z}[v^{-1}]$. Now consider the summand coming from $z_{j-3}z_{j-1}$. The monomials are already in the right order. We may have increased the exponent of $z_{j-3}$ from $1$ to $2$. It is easy to see that all coefficients are in $v^{-1}\mathbb{Z}[v^{-1}]$. This shows that $D$ satisfies property (1).

To conclude $D=\Delta^v_{i,j}$, we have to verify property (2) of Theorem \ref{QuantRecursionForClusterVar}. We use parts (c) and (d) of the induction hypothesis for the for the pair $(i,j-1)$ to obtain the following equation. Note that if $j-i \leq 3$, then $(\vert y_{j-1}\vert,\vert e_{i,j-4}\vert)=0$, and if $j-i \geq 4$, then $(\vert y_{j-1}\vert,\vert e_{i,j-4}\vert)=1$, therefore
\begin{align*}
D&=\Delta^v_{i,j-1}y_j-v^{A_{i,j}}\Delta^v_{i,j-3}p_{j-2}z_{j+1}\\
&=v^{-1}y_j\Delta^v_{i,j-1}+v^{-(\vert y_{j-1}\vert,\vert e_{i,j-4}\vert)}(v^{-1}-v)\Delta^v_{i,j-3}p_{j-2}z_{j+1}\\
&\quad\quad -v^{A_{i,j}}\Delta^v_{i,j-3}p_{j-2}z_{j+1}\\
&=v^{-1}y_j\Delta^v_{i,j-1}+v^{1-(\vert y_{j-1}\vert,\vert e_{i,j-4}\vert)}\Delta^v_{i,j-3}p_{j-2}z_{j+1}.
\end{align*}
It is easy to see that $z_{j+1}$ commutes with $p_{j-2}$. Furthermore, it commutes with $\Delta^v_{i,j-3}$ since it commutes with every $E^*(\beta_k)$ in every summand in the dual PBW expansion of $\Delta^v_{i,j-3}$ according to Remark \ref{PBWExpansionOfDelta}. Lemma \ref{PCommutation} implies
\begin{align*}
\Delta^v_{i,j-3}p_{j-2}=v^{-(\vert y_{j-2}\vert,o_{i,j-3})-(\vert z_{j-2}\vert,o_{i,j-3})+(\vert z_{j-2}\vert,e_{i,j-4})}p_{j-2}\Delta^v_{i,j-3}.
\end{align*}
The assumptions $j-i\geq 3$ and $j<n$ imply the equations 
\begin{align*}
&(\vert y_j\vert,o_{i,j-1})=(\vert y_j\vert,o_{i,j-3})=1,\\
&(\vert z_{j-2}\vert,o_{i,j-3})=1.
\end{align*} 
The observation $(\vert z_{j-2}\vert,e_{i,j-4})=0$ for $j-i \leq 3$, and $(\vert z_{j-2}\vert,e_{i,j-4})=1$ for $j-i \geq 4$ yields $(\vert z_{j-2}\vert,e_{i,j-4})=-A_{i,j}-1$. It follows that
\begin{align*}
D&=v^{-(\vert y_j\vert,o_{i,j-1})}y_j\Delta^v_{i,j-1}-v^{-A_{i,j}-(\vert y_j\vert +\vert y_{j-2}\vert,o_{i,j-3})-(\vert y_{j-1}\vert,e_{i,j-4})}z_{j+1}p_{j-2}\Delta^v_{i,j-3}.
\end{align*}
By Lemmas \ref{FirstLemma} and \ref{Lemma2} the last equation is equivalent to property (2). Thus $A=\Delta^v_{i,j}$. Incidentally, we have verified part (a) for the pair $(i,j)$.

Now we prove that the new defined $\Delta^v_{i,j}$ satisfies the quantized cluster recursion (b) of Theorem \ref{QuantRecursionForClusterVar} using the induction hypothesis for part (b):
\begin{align*}
\Delta^v_{i,j}z_j&=\Delta^v_{i,j-1}(p_j+vz_{j-1}z_{j+1})-v^{A_{i,j}}\Delta^v_{i,j-3}p_{j-2}z_{j+1}z_j\\
&=\Delta^v_{i,j-1}p_j+\Big(v\Delta^v_{i,j-1}z_{j-1}-v^{1+A_{i,j}}\Delta^v_{i,j-3}p_{j-2}z_j\Big)z_{j+1}\\
&=\Delta^v_{i,j-1}p_j+\Big(v\Delta^v_{i,j-1}z_{j-1}-v^{-(\vert y_{j-1}\vert,e_{i,j-2})}\Delta^v_{i,j-3}p_{j-2}z_j\Big)z_{j+1}\\
&=\Delta^v_{i,j-1}p_j+v^{1-(\vert y_{j-1}\vert,e_{i,j-2})}\Delta^v_{i,j-2}p_{j-1}z_{j+1}.
\end{align*}

Now we prove that the new defined $\Delta^v_{i,j}$ satisfies property (c) of Theorem \ref{QuantRecursionForClusterVar}. By induction hypothesis we know that part (c) is true for the pair $(i,j-1)$. Note that $(\vert y_{j+1}\vert,\vert z_{j+1}\vert)=(\vert y_j\vert,\vert y_{j+1}\vert)=-\delta_{j,n-1}$. The following calculation verifies part (c) of the theorem:
\begin{align*}
y_{j+1}\Delta^v_{i,j}&=y_{j+1}\Delta^v_{i,j-1}y_j-v^{A_{i,j}}y_{j+1}\Delta^v_{i,j-3}p_{j-2}z_{j+1}\\
&=v^{-1}\Delta^v_{i,j-1}\Big(v^{-(\vert y_j\vert,\vert y_{j+1}\vert)}y_jy_{j+1}+(v^{-1}-v)z_{j+2}z_{j-1}\Big)\\
&\quad-v^{A_{i,j}-1}\Delta^v_{i,j-3}p_{j-2}\Big(v^{-(\vert y_{j+1}\vert,\vert z_{j+1}\vert)}z_{j+1}y_{j+1}+(v^{-1}-v)z_jz_{j+2}\Big)\\
&=v^{-(\vert y_{j+1}\vert,e_{i,j})}\Delta_{i,j}y_{j+1}\\
&\quad+v^{-1}(v^{-1}-v)\Big(\Delta^v_{i,j-1}z_{j-1}-v^{A_{i,j}}\Delta^v_{i,j-3}p_{j-2}z_j\Big)z_{j+2}\\
&=v^{-(\vert y_{j+1}\vert,e_{i,j})}\Delta_{i,j}y_{j+1}+(v^{-1}-v)v^{-1-(\vert y_{j-1}\vert,e_{i,j-2})}\Delta^v_{i,j-2}p_{j-1}z_{j+1}.
\end{align*}

It remains to verify part (d). Note that in each monomial in the dual PBW expansion of $\Delta^v_{i,j-1}$ either the variable $y_{j-1}$ or the variable $z_j$ occurs (depending on whether $a_{j-1}$ in Remark \ref{PBWExpansionOfDelta} is zero or one). The variable $y_{j+2}$ commutes with all other terms. Thus, we see that $\Delta^v_{i,j-1}y_{j+2}=v^{-1}y_{j+2}\Delta^v_{i,j-1}$. It follows that 
\begin{align*}
\Delta_{i,j}y_{j+2}&=y_j\Delta^v_{i,j-1}y_{j+2}-v^{A_{i,j}}z_{j+1}p_{j-2}\Delta^v_{i,j-3}y_{j+2}\\
&=v^{-1}y_jy_{j+2\Delta^v_{i,j-1}}-v^{-1+A_{i,j}}y_{j+2}z_{j+1}p_{j-2}\Delta^v_{i,j-3}=v^{-1}y_{j+2}\Delta_{i,j}.
\end{align*}

Now assume that $j$ is odd. Put $D=y_j\Delta^v_{i,j-1}-v^{A_{i,j}}z_{j+1}p_{j-2}\Delta^v_{i,j-3}$. We prove that $D=\Delta^v_{i,j}$. By the same arguments as above $D$ fulfills property (1) of the theorem. To conclude to $D=\Delta^v_{i,j}$, we have to verify property (2) of Theorem \ref{QuantRecursionForClusterVar}. We use parts (c) and (d) of the induction hypothesis for the pair $(i,j-1)$ to obtain the following equation:
\begin{align*}
D&=y_j\Delta^v_{i,j-1}-v^{A_{i,j}}z_{j+1}p_{j-2}\Delta^v_{i,j-3}\\
&=v^{-(\vert y_j\vert, e_{i,j-1})}\Delta^v_{i,j-1}y_j+v^{-1-(\vert y_{j-2}\vert,e_{i,j-3})}(v^{-1}-v)\Delta^v_{i,j-3}p_{j-2}z_{j+1}\\
&\quad-v^{A_{i,j}}v^{-(\vert y_{j-2}\vert,e_{i,j-3})-(\vert z_{j-2}\vert,e_{i,j-3})+(\vert z_{j-2}\vert,o_{i,j-4})}\Delta^v_{i,j-3}p_{j-2}z_{j+1}.
\end{align*}
Note that $(\vert z_{j-2}\vert,e_{i,j-3})=1$, and that $(\vert z_{j-2}\vert,o_{i,j-4})=0$ for $j-i \leq 3$ and $(\vert z_{j-2}\vert,o_{i,j-4})=1$ for $j-i \geq 4$. Hence, $A_{i,j}-(\vert z_{j-2}\vert,o_{i,j-4})(\vert z_{j-2}\vert,o_{i,j-4})=-2$ yielding
\begin{align*}
D&=v^{-(\vert y_j\vert, e_{i,j-1})}\Delta^v_{i,j-1}y_j-v^{-(\vert y_{j-2}\vert,e_{i,j-3})}\Delta^v_{i,j-3}p_{j-2}z_{j+1}.
\end{align*}
The equation $-A_{i,j}-(\vert y_j\vert,e_{i,j-3})-(\vert y_{j-1}\vert,o_{i,j-4})$ finishes the proof of the second equation of part (a). By Lemmas \ref{FirstLemma} and \ref{Lemma2} the last equation is equivalent to property (2). Hence, $D=\Delta^v_{i,j}$.

Now we prove that the new defined $\Delta^v_{i,j}$ satisfies property (b) of Theorem \ref{QuantRecursionForClusterVar}. First of all assume that $j<n$. We obtain: 
\begin{align*}
\Delta_{i,j}^vz_j&=v^{-(\vert y_j\vert, e_{i,j-1})}\Delta^v_{i,j-1}(p_j+v^{-1}z_{j-1}z_{j+1})\\&\quad-v^{-(\vert y_{j-2}\vert,e_{i,j-3})}\Delta^v_{i,j-3}p_{j-2}z_{j+1}z_{j-1}.\\
&=v^{-(\vert y_j\vert, e_{i,j-1})}\Delta^v_{i,j-1}p_j\\
&\quad+\Big(v^{-1-(\vert y_j\vert, e_{i,j-1})}\Delta^v_{i,j-1}z_{j-1}-v^{-1-(\vert y_{j-2}\vert,e_{i,j-3})}\Delta^v_{i,j-3}p_{j-2}z_{j-1}\Big)z_{j+1}\\
&=v^{-(\vert y_j\vert, e_{i,j-1})}\Delta^v_{i,j-1}p_j+v^{-1-(\vert y_j\vert, e_{i,j-1})}\Delta^v_{i,j-2}p_{j-1}z_{j+1}.
\end{align*}
Here we used he fact that $(\vert y_j\vert,e_{i,j-1})=1$ to adopt the induction hypothesis. For $j=n$ we have $(\vert y_j\vert,e_{i,j-1})=0$, but this defect is compensated by the relation $z_{j-1}z_{j+1}=z_{j+1}z_{j-1}$ (due to $z_{j+1}=1$) instead of $z_{j-1}z_{j+1}=v^{-1}z_{j+1}z_{j-1}$.

Now we prove that the new defined $\Delta^v_{i,j}$ satisfies property (c) of Theorem \ref{QuantRecursionForClusterVar}. By induction hypothesis we know that part (b) is true for the pair $(i,j-1)$. Note that $(\vert y_j\vert,e_{i,j-3})=(\vert y_j\vert,e_{i,j-1})$ and that $1=-A_{i,j}-(\vert y_{j-1}\vert,o_{i,j-4})$. We also use the fact that $z_{j+1}$ commutes with $\Delta^v_{i,j-3}$ since it commutes with every factor of every monomial in the PBW expansion of $\Delta^v_{i,j-3}$. Thus, we see that $y_{j+1}\Delta^v_{i,j}$ is equal to:
\begin{align*}
&v^{-(\vert y_j\vert,e_{i,j-1})}y_{j+1}\Delta^v_{i,j-1}y_j\\&\quad\quad\quad-v^{-A_{i,j}-(\vert y_j\vert + \vert y_{j-2}\vert,e_{i,j-3})-(\vert y_{j-1}\vert,o_{i,j-4})}y_{j+1}\Delta^v_{i,j-3}p_{j-2}z_{j+1}\\
&\quad\quad=v^{-(\vert y_j\vert,e_{i,j-3})+1}\Big[\Delta^v_{i,j-1}\Big(y_jy_{j+1}+(v-v^{-1})z_{j-1}z_{j+2}\Big)\\
&\quad\quad\quad\quad\quad\quad\quad\quad\quad-v^{1-(\vert y_{j-2}\vert,e_{i,j-3})}\Delta^v_{i,j-3}p_{j-2}\Big(z_{j+1}y_{j+1}+(v-v^{-1})z_jz_{j+2}\Big)\Big]\\
&\quad\quad=v\Delta^v_{i,j}y_{j+1}+v^{-(\vert y_j\vert,e_{i,j-3})+1}(v-v^{-1})\Big[\Delta^v_{i,j-1}z_{j-1}\\
&\quad\quad\quad\quad\quad\quad\quad\quad\quad\quad\quad\quad\quad\quad\quad\quad -v^{1-(\vert y_{j-2}\vert,e_{i,j-3})}\Delta^v_{i,j-3}p_{j-2}z_j\Big]z_{j+2}\\
&\quad\quad=v\Delta^v_{i,j}y_{j+1}+v^{-(\vert y_j\vert,e_{i,j-3})+1}(v-v^{-1})\Delta_{i,j-2}^vp_{j-1}z_{j+2}.
\end{align*}

Part (d) is proved similarly as in the case where $j$ is even.
\end{proof}

\begin{remark}
By symmetry there is also a recursion for every $\Delta_{i,j}^v$ (with $j-i\geq 2$) in terms of various $\Delta_{i',j}^v$ with $i'>i$.
\end{remark}

\textit{Quantum cluster algebras} were introduced by Berenstein-Zelevinsky \cite{BZ:QuantumClusterAlg}. We ready to conclude with our main theorem which establishes a quantum cluster algebra structure on the $\mathbb{Z}[v^{\pm \frac12}]$-algebra $\mathcal{A}_v(w)=\bigoplus_{{\textbf a} \in \mathbb{N}^{2n}}\mathbb{Z}[v^{\pm \frac12 }]E[{\textbf a}]^*$.

\begin{theorem}
The $\mathbb{Z}[v^{\pm \frac12}]$-algebra $\mathcal{A}_v(w)=\bigoplus_{{\textbf a} \in \mathbb{N}^{2n}}\mathbb{Z}[v^{\pm \frac12 }]E[{\textbf a}]^*$ is a quantum cluster algebra of type $A_n$. Every mutable quantum cluster variables is up to a power of $v$ an element in the dual canonical basis $\mathcal{B}^*$. The following properties hold:
\begin{itemize}
\item[(a)] The quantum cluster variables $v^{\frac12}z_i$ for $1 \leq i \leq n$ together with the frozen quantum cluster variables are $v^{\frac14(\vert y_i\vert+\vert z_i\vert,\vert y_i\vert+\vert z_i\vert)}p_i$ for $1 \leq i \leq n$ form an initial cluster of $A_v(w)$ whose $B$-matrix is encoded in the quiver from Figure \ref{fig:SnakeCluster}.  
\item[(b)] The remaining quantum cluster variables are $v^{\frac14 (s_{i,j},s_{i,j})}\Delta^v_{i,j}$ for $1 \leq i\leq j \leq n$.
\end{itemize}
\end{theorem}

Note that $\frac12(s_{i,j},s_{i,j})=j-i+1+(e_{i,j-1},o_{i,j-1})\in\mathbb{Z}$.

\begin{proof}
We quantize the proof of Lemma \ref{ClusterVariables}. We construct a seed of a quantum cluster algebra similar to the base seed in Figure (\ref{fig:SnakeCluster}). The cluster variables $Z_i,P_i$ (for $1 \leq i \leq n$) are replaced by the quantum cluster variables $v^{\frac12}z_i, v^{\frac14(\vert y_i\vert+\vert z_i\vert,\vert y_i\vert+\vert z_i\vert)}p_i$ (for $1 \leq i \leq n$). Notably, every pair of quantum cluster variables in the base seed forms a quantum torus, i.e., it satisfies a $v$-commutativity relation. (The $v$-commutativity relations among the $z_i$ follow from the straigthening relations; the $v$-commutativity relations between the $P_i$ and the $z_i$ follow from Lemma \ref{PCommutation}; the $v$-commutativity relations among the $P_i$ can be checked using the straightening relations.) These relations are strict commutativity relations except for the following $v$-commutativity relations:
\begin{align*}
&z_iz_j=v^{-(\vert z_i\vert,\vert z_j \vert)}z_jz_i, &&i \ {\rm even}, j \ {\rm odd},\\
&z_ip_j=v^{(\vert z_i\vert,\vert y_j \vert)}p_jz_i, &&i,j \ {\rm even},\\
&z_ip_j=v^{-(\vert z_i\vert,\vert y_j \vert)}p_jz_i, &&i,j \ {\rm odd},\\
&p_ip_j=v^{-(\vert z_i\vert,\vert z_j \vert)+(\vert z_i\vert,\vert y_j \vert)+(\vert y_i\vert,\vert z_j \vert)+(\vert y_i\vert,\vert y_j \vert)}p_jz_i, &&i \ {\rm even}, j \ {\rm odd}.\\
&p_ip_j=v^{-(\vert z_i\vert,\vert y_j \vert)+(\vert y_i\vert,\vert z_j \vert)}p_jz_i, &&i,j \ {\rm even},\\
&p_ip_j=v^{(\vert z_i\vert,\vert y_j \vert)-(\vert y_i\vert,\vert z_j \vert)}p_jz_i, &&i,j \ {\rm odd}.
\end{align*}
Now it is easy to see that the $B$-matrix induced from the quiver in Figure \ref{fig:SnakeCluster} and the $\Lambda$-matrix induced from the $v$-commutativity relations form a {\it compatible pair} in the sense of Berenstein-Zelevinsky \cite[Definition 3.1]{BZ:QuantumClusterAlg}. Hence, the base seed is a valid initial quantum cluster.

Now fix an integer $i$ such that $1 \leq i \leq n$. Beginning with the base we perform mutations at vertices $i,i+1,\ldots,j$, consecutively, as in the proof of Lemma \ref{ClusterVariables}. We prove by induction on $j$ that the new quantum cluster variable $X$ that occurs after the sequence of mutations from above is equal to $\Delta^v_{i,j}$. The case $i=j$ is trivial. Note that part (c) of Theorem \ref{QuantRecursionForClusterVar} makes also sense for $j=i+1$. We distinguish two cases. First of all, assume that $j$ is even. By Berenstein-Zelevinsky \cite[Proposition 4.9]{BZ:QuantumClusterAlg} we have $$X=M'+M''$$ where definition of $M'$ and $M''$ involves $v$-commutativity relations among the quantum cluster variables in the previous seed. To describe the variable $M'$ we compute
$
\Delta^v_{i,j-1}p_jz_j^{-1}=v^{-(\vert y_j\vert,o_{i,j-1})}z_j^{-1}p_j\Delta^v_{i,j-1}.
$
Therefore, the summand $M'$ is given by the following equation:
\begin{align*}
M'&=v^{\frac12 (\vert y_j\vert,o_{i,j-1})}\cdot v^{\frac12(j-i)+\frac12(e_{i,j-2},o_{i,j-1})}\Delta^v_{i,j-1}\cdot vp_j \cdot v^{-\frac12}z_j^{-1}\\
&=v^{\frac12(j-i+1)+\frac12(e_{i,j},o_{i,j-1})}\Delta^v_{i,j-1}p_jz_j^{-1}.
\end{align*}
Furthermore, to describe $M''$ we obtain:
\begin{align*}
\Delta^v_{i,j-2}p_{j-1}z_{j+1}z_j^{-1}&=v^{-(\vert z_j\vert,\vert z_{j+1}\vert)-(\vert z_j\vert,\vert z_{j-1}\vert)+(\vert z_j\vert,\vert y_{j-1}\vert)-(\vert z_{j-1}\vert,\vert y_{j-1}\vert)}\\
&\quad \cdot v^{(e_{i,j-2},\vert y_{j-1}\vert+\vert z_{j-1}\vert-\vert z_j\vert)}\cdot v^{(o_{i,j-3},-\vert y_{j-1}\vert-\vert z_{j-1}\vert+\vert z_j\vert)}\\
&\quad \cdot z_j^{-1}z_{j+1}p_{j-1}\Delta^v_{i,j-2}\\
&=v^{-2+(e_{i,j-2},\vert y_{j-1}\vert)-(o_{i,j-3},\vert y_j\vert)}z_j^{-1}z_{j+1}p_{j-1}\Delta^v_{i,j-2}
\end{align*}
Therefore, the summand $M'$ is given by the following equation:
\begin{align*}
M''&=v^{1-\frac12 (e_{i,j-2},\vert y_{j-1}\vert)+\frac12(o_{i,j-3},\vert y_j\vert)} \cdot v^{\frac12(j-i-1)+\frac12(e_{i,j-2},o_{i,j-3})}\Delta^v_{i,j-2}\\
& \quad \cdot v^{1-\frac12(\vert y_j\vert,\vert y_{j-1}\vert)}\cdot v^{\frac12}z_{j+1} \cdot v^{-\frac12}z_j^{-1}\\
&=v^{1-(\vert y_{j-1}\vert,e_{i,j-2})}v^{\frac12(j-i+1)+\frac12(e_{i,j},o_{i,j-1})}\Delta^v_{i,j-2}p_{j-1}z_{j+1}z_j^{-1}.
\end{align*}
A comparison with the formula in part (c) of Theorem \ref{QuantRecursionForClusterVar} shows that $X=\Delta^v_{i,j}$ which completes the induction step.

The case with odd $j$ is treated similarly.
\end{proof}

\begin{remark}
Similarly, by adjusting the bilinear form we can equip the algebra $U_v^+(w')$ attached to $Q'$ with the structure of a quantum cluster algebra. 
\end{remark}

\begin{remark}
Note that ${\rm deg}(\Delta_{i,j})=\underline{{\rm dim}}(M_{i,j})=s_{i,j}$ for all $1<i<j<n$. Thus, the object $M_{i,j}$ from Remark \ref{Rohleder} is the indecomposable rigid object in $\mathcal{A}(\mathcal{C}_M)$ corresponding to the cluster variable $\Delta_{i,j}$. The corresponding quantum cluster variable is the dual canonical basis element $\Delta^v_{i,j}$ scaled by a factor $v^{\frac14 (s_{i,j},s_{i,j})}$. By \cite[Lemma 3.12]{GLS:Rigid} the exponent can be interpreted as $\frac14 (s_{i,j},s_{i,j})=\frac12{\rm dim}({\rm End}(M_{i,j}))$. The same relation is true for the other (mutable and frozen) quantum cluster variables. 
\end{remark}

\subsection*{Acknowledgements}
This paper is a part of my Ph.D. studies at the University of Bonn supervised by Jan Schr\"oer whom I would like to thank for many productive discussions and steady encouragement. I am also grateful to Bernard Leclerc for valuable explanations. I would like to thank an anonymous referee for various comments and remarks.

\end{document}